\documentclass[12pt]{amsart}
\usepackage{amsmath}
\usepackage{amsfonts}
\usepackage{amssymb}
\usepackage{amscd}
\usepackage{mathtools}
\usepackage{amsxtra}
\usepackage{amsthm}
\usepackage{graphicx}
\usepackage[abbrev,alphabetic]{amsrefs}
\RequirePackage[dvipsnames,usenames]{color}
\usepackage{soul,xcolor}
\setstcolor{red}
\usepackage{stmaryrd}
\usepackage{booktabs}
\usepackage{multirow}
\newtagform{tiny}{\tiny(}{)}
\usepackage{aliascnt}

\usepackage{mathtools}
\usepackage{hyperref}
\usepackage[margin=1.25in]{geometry}
\usepackage{mathrsfs}

\usepackage{amsthm}
\usepackage{comment}
\usepackage[all,cmtip]{xy}
\usepackage{tikz-cd}
\usetikzlibrary{cd}

\tikzcdset{
  cells={font=\everymath\expandafter{\the\everymath\displaystyle}},
}

\usepackage[all]{xy}

\usepackage{cleveref}

\makeatletter
\def\@tocline#1#2#3#4#5#6#7{\relax
  \ifnum #1>\c@tocdepth 
  \else
    \par \addpenalty\@secpenalty\addvspace{#2}%
    \begingroup \hyphenpenalty\@M
    \@ifempty{#4}{%
      \@tempdima\csname r@tocindent\number#1\endcsname\relax
    }{%
      \@tempdima#4\relax
    }%
    \parindent\z@ \leftskip#3\relax \advance\leftskip\@tempdima\relax
    \rightskip\@pnumwidth plus4em \parfillskip-\@pnumwidth
    #5\leavevmode\hskip-\@tempdima
      \ifcase #1
       \or\or \hskip 1em \or \hskip 2em \else \hskip 3em \fi%
      #6\nobreak\relax
    \hfill\hbox to\@pnumwidth{\@tocpagenum{#7}}\par
    \nobreak
    \endgroup
  \fi}
\makeatother

\renewcommand{\P}{\mathbb{P}}
\newcommand{\Z}{\mathbb{Z}}

\newcommand{\F}{\mathbb{F}}

\newcommand{\cHom}{\mathcal{H}om}

\newcommand{\cO}{\mathcal{O}}

\newcommand{\red}{\mathrm{red}}

\DeclareMathOperator{\res}{res}

\makeatletter
\DeclareRobustCommand{\graded}{%
    \@ifnextchar\bgroup{\graded@with}{\ensuremath{\mathrm{gr}}}%
}
\newcommand{\graded@with}[1]{\ensuremath{#1\mathrm{-gr}}}
\makeatother

\DeclareMathOperator{\pfd}{pfd}

\makeatletter
\DeclareRobustCommand{\comp}[1]{%
    \@ifnextchar\bgroup{\comp@with{#1}}{\comp@without{#1}}%
}
\newcommand{\comp@without}[1]{\ensuremath{#1\mathrm{-comp}}}
\makeatother

\def\var{\overline}

\DeclareMathOperator{\Spec}{Spec}

\DeclareMathOperator{\Hom}{Hom}

\DeclareMathOperator{\Ext}{Ext}

\DeclareMathOperator{\dR}{dR}

\DeclareMathOperator{\BC}{BC}

\DeclareMathOperator{\src}{scr}
\DeclareMathOperator{\tgt}{tgt}

\DeclareMathOperator{\PerfPrism}{PerfPrism}

\newcommand{\proj}{\mathrm{proj}}

\theoremstyle{plain}
\newtheorem{theorem}{Theorem}[section]

\newtheorem{theoremA}{Theorem}

\crefname{theoremA}{Theorem}{Theorems}
\Crefname{theoremA}{Theorem}{Theorems}

\newaliascnt{lemma}{theorem}
\newtheorem{lemma}[lemma]{Lemma}
\aliascntresetthe{lemma}
\crefname{lemma}{Lemma}{Lemmas}
\Crefname{lemma}{Lemma}{Lemmas}

\newaliascnt{proposition}{theorem}
\newtheorem{proposition}[proposition]{Proposition}
\aliascntresetthe{proposition}
\crefname{proposition}{Proposition}{Propositions}
\Crefname{proposition}{Proposition}{Propositions}

\newaliascnt{corollary}{theorem}
\newtheorem{corollary}[corollary]{Corollary}
\aliascntresetthe{corollary}
\crefname{corollary}{Corollary}{Corollaries}
\Crefname{corollary}{Corollary}{Corollaries}

\newaliascnt{claim}{theorem}

\aliascntresetthe{claim}
\crefname{claim}{Claim}{Claims}
\Crefname{claim}{Claim}{Claims}

 \newtheorem*{claim*}{Claim}

\theoremstyle{definition}
\newaliascnt{definition}{theorem}
\newtheorem{definition}[definition]{Definition}
\aliascntresetthe{definition}
\crefname{definition}{Definition}{Definitions}
\Crefname{definition}{Definition}{Definitions}

\newaliascnt{notation}{theorem}
\newtheorem{notation}[notation]{Notation}
\aliascntresetthe{notation}
\crefname{notation}{Notation}{Notations}
\Crefname{notation}{Notation}{Notations}

\newaliascnt{example}{theorem}
\newtheorem{example}[example]{Example}
\aliascntresetthe{example}
\crefname{example}{Example}{Examples}
\Crefname{example}{Example}{Examples}

\theoremstyle{remark}
\newaliascnt{remark}{theorem}
\newtheorem{remark}[remark]{Remark}
\aliascntresetthe{remark}
\crefname{remark}{Remark}{Remarks}
\Crefname{remark}{Remark}{Remarks}

\newaliascnt{setting}{theorem}

\aliascntresetthe{setting}
\crefname{setting}{Setting}{Settings}
\Crefname{setting}{Setting}{Settings}

\newaliascnt{construction}{theorem}
\newtheorem{construction}[construction]{Construction}
\aliascntresetthe{construction}
\crefname{construction}{Construction}{Constructions}
\Crefname{construction}{Construction}{Constructions}

\newaliascnt{warning}{theorem}

\aliascntresetthe{warning}
\crefname{warning}{Warning}{Warnings}

\numberwithin{equation}{section}

\makeatletter
\newcommand{\OdR}[1]{\mathcal{O}^{\dR}_{#1}}



\usepackage{mymacros_common_Ryo, mymacros_english_paper_Ryo}

\title[Hodge--Tate splitting and Akizuki--Nakano vanishing]{Hodge--Tate splitting and Akizuki--Nakano vanishing in positive characteristic}
\author{Ryo Ishizuka}
\address{Institute of Science Tokyo, Tokyo 152-8551, Japan}
\email{ishizuka.r.ac@m.titech.ac.jp}

\author{Shou Yoshikawa}
\address{Institute of Science Tokyo, Tokyo 152-8551, Japan}
\email{yoshikawa.s.9fe9@m.isct.ac.jp}



\begin{document}

\begin{abstract}
We introduce the notion of Hodge--Tate splitting for schemes in
positive characteristic.
For a smooth variety \(X\) over a perfect field \(k\) of characteristic $p>0$, we say that
\(X\) is Hodge--Tate split if the natural morphism
\[
\mcalO_X
 \to 
F_*\Omega^\bullet_{X/k}
\]
induced by the absolute Frobenius admits a splitting in
\(\mcalD_{\qcoh}(X)\).
We prove that this condition is equivalent to the existence of a
decomposition
\[
F_*\Omega^\bullet_{X/k}
\simeq
\bigoplus_{i=0}^{\dim X}\Omega^i_{X/k}[-i]
\]
of its de Rham complex.
Consequently, for smooth projective varieties, Hodge--Tate splitting
implies Akizuki--Nakano vanishing and the \(E_1\)-degeneration of the
Hodge-to-de Rham spectral sequence.

Furthermore, we establish criteria and permanence properties for
Hodge--Tate splitting and use them to construct many new examples of
varieties whose de Rham complexes decompose.
These include blow-ups of quasi-\(F\)-split varieties along strata of
simple normal crossings divisors, complete intersections in toric
varieties, and linearly reductive quotients of Hodge--Tate split varieties.
Among these examples, we obtain smooth projective varieties whose Hodge-to-de Rham spectral sequences degenerate at \(E_1\), whereas their
Hochschild--Kostant--Rosenberg spectral sequences do not degenerate.

As an application in mixed characteristic, we prove an
Akizuki--Nakano-type vanishing theorem for smooth projective globally
\(+\)-regular varieties over the Witt ring of a perfect field.
\end{abstract}

\maketitle 

\tableofcontents

\section{Introduction}

\subsection{Decomposition of de Rham complexes}
Let \(X\) be a smooth projective variety of dimension \(d\)
over a field of characteristic zero, and let \(L\) be an ample line
bundle on \(X\).  
The Kodaira vanishing theorem asserts that
\[
H^i(X,L^{-1})=0
\qquad
(i<d);
\]
see \cite{KodairaVanishing} and
\cite{EsnaultViehwegVanishing}*{Chapter~I}.
Its refinement, the Akizuki--Nakano vanishing theorem, states that
\begin{equation}\label{eq:AN-char-zero}
H^i(X,\Omega_X^j\otimes L^{-1})=0
\quad\text{if }i+j<d,
\end{equation}
see \cite{AkizukiNakano}, \cite{NakanoVanishing}, and
\cite{EsnaultViehwegVanishing}*{Chapter~II}.

These theorems play a central role throughout algebraic geometry.
For instance, they enter cohomological proofs of Lefschetz-type
theorems for ample divisors, the study of Picard groups and
fundamental groups of hyperplane sections, deformation theory, and
unobstructedness results for Fano varieties; see
\cite{ShiffmanSommeseVanishing}*{Chapter~3} and \cite{LazarsfeldPositivityI}*{Section~4.3}.
They are also indispensable in the classification theory of
projective varieties and in the study of adjoint linear systems.

On the other hand, Kodaira vanishing fails in general in positive characteristic, and
hence so does Akizuki--Nakano vanishing. 
Counterexamples were first constructed by Raynaud \cite{RaynaudCounterexample}; see also \cite{LauritzenRao} and \cite{MukaiKodairaCounterexample} for further
examples in higher dimensions.  
Similarly, the Hodge-to-de Rham spectral sequence
\begin{equation}\label{eq:HdR-ss}
E_1^{i,j}
=
H^j(X,\Omega_X^i)
 \Longrightarrow
H_{\dR}^{i+j}(X/k) \defeq H^{i+j}(X, \Omega^{\bullet}_{X/k})
\end{equation}
need not degenerate at the \(E_1\)-page in positive characteristic;
see, for example, \cite{MumfordPathologies}*{Section~I} and \cite{PetrovNondecomposable}.

A fundamental result in positive characteristic is due to Deligne and Illusie
\cite{DI87}.  
Let \(k\) be a perfect field of characteristic \(p>0\),
and let \(X\) be a smooth \(k\)-variety.  
If \(X\) admits a lifting to \(W_2(k)\) and \(d \defeq \dim X \leq p\), they obtain the decomposition
\begin{equation}\label{eq:DI-decomposition}
F_*\Omega_{X/k}^{\bullet}
\simeq
\bigoplus_{i=0}^{d}
\Omega_{X/k}^i[-i]
\end{equation}
in \(\mcalD_{\qcoh}(X)\).
The same method has a logarithmic counterpart for a liftable simple
normal crossings pair; see \cite{Kat89}*{Theorem~(4.12)}.

If $X$ is projective, the decomposition \eqref{eq:DI-decomposition} implies the \(E_1\)-degeneration of \eqref{eq:HdR-ss} and Akizuki--Nakano vanishing \eqref{eq:AN-char-zero} in positive characteristic; see \cite{DI87}*{Lemma~2.9}.  
Thus, for liftable varieties whose dimension does not exceed the characteristic, the result recovers two of the most important characteristic-zero consequences of Hodge theory by a purely algebraic argument.

The \(E_1\)-degeneration of the Hodge-to-de Rham spectral sequence is important for reasons extending beyond vanishing theorems.  
It identifies the dimension of de Rham cohomology with the sum of the Hodge numbers,
\[
\dim_k H_{\dR}^n(X/k)
=
\sum_{i+j=n}h^j(X,\Omega_X^i),
\]
and ensures that the Hodge filtration has the expected graded pieces.
Combined with torsion-freeness of crystalline cohomology, this is one
of the hypotheses entering the theory of Mazur--Ogus varieties and
the comparison between Frobenius and the Hodge filtration; see
\cite{MazurFrobeniusHodge} and \cite{OgusFCrystals}.

For families of Calabi--Yau varieties, degeneration and local
freeness of the Hodge bundles also enter the study of Hasse
invariants.  In particular, Ogus relates the vanishing order of the
Hasse invariant to the relative position of the conjugate line and
the Hodge filtration; see \cite{OgusHasseLocus}.  This is often
referred to as \emph{Ogus' principle}.  These examples illustrate that
\(E_1\)-degeneration is not merely a numerical property of
cohomology: it allows Frobenius, the Hodge filtration, and variation
in families to be compared in a controlled way.

More recently, Petrov found another sufficient condition for the
full decomposition of the de Rham complex \eqref{eq:DI-decomposition}.  The notion of
quasi-\(F\)-splitting was introduced by Yobuko as a Witt-vector
extension of Frobenius splitting
\cite{YobukoQuasiFSplitting}.  Petrov proved that if \(X\) is a smooth
quasi-\(F\)-split variety over a perfect field, then the full
decomposition \eqref{eq:DI-decomposition} holds in every dimension,
without any restriction comparing \(p\) and \(\dim X\)
\cite{petrov2025Decomposition}.  Consequently, every smooth projective
quasi-\(F\)-split variety satisfies Akizuki--Nakano vanishing, and its Hodge-to-de Rham
spectral sequence degenerates at \(E_1\).

\subsection{Hodge--Tate splitting}
The purpose of this paper is to isolate the splitting property
underlying decompositions of the form
\eqref{eq:DI-decomposition} and to study its behavior under
geometric constructions.  Our formulation applies to both ordinary
and logarithmic de Rham complexes.

For simplicity, let \(X\) be a smooth variety of dimension \(d\) over
a perfect field \(k\) of characteristic \(p>0\), and let \(D\) be a
simple normal crossings divisor on \(X\).  We consider the natural
morphism of complexes of \(\mcalO_X\)-modules (\Cref{DefdeRhamStack}):
\[
F^{\dR}_{(X,D)}
\colon
\mcalO_{X}
 \to 
F_{*}\Omega_{X/k}^{\bullet}(\log D)
\]
whose degree-zero component is induced by the absolute Frobenius and
whose components in positive degrees are zero. 
We introduce the central notion of log Hodge--Tate splitting as follows.

\begin{definition}[{\Cref{Def-log-HT-split}}]
We say that the pair \((X,D)\) is \emph{log Hodge--Tate split}, or \emph{log HT-split} for short, if \(F^{\dR}_{(X,D)}\) admits a splitting in \(\mcalD_{\qcoh}(X)\).  When \(D=0\), we simply say that \(X\) is \emph{HT-split}.

See \Cref{HTSplitMorphism} and \Cref{relative-vs-absolute} for the definition of HT-splitting for arbitrary \(k\)-schemes in terms of de Rham stacks.\footnote{The terminology ``Hodge--Tate splitting'' comes from the Hodge--Tate stack. In characteristic \(p\), the de Rham stack used here is canonically identified with the Hodge--Tate stack relative to the crystalline prism \((\setZ_p,(p))\).}
\end{definition}

\begin{theoremA}[{\Cref{criterion-decomp-log}}] \label{thm:intro-characterization}
Let \(X\) be a smooth \(k\)-variety of dimension \(d\), and let \(D\)
be a simple normal crossings divisor on \(X\).  Then the following
conditions are equivalent:
\begin{enumerate}
\item The pair \((X,D)\) is log HT-split.
\item There is an isomorphism
\begin{equation}\label{eq:log-decomposition-intro}
F_{*}\Omega_{X/k}^{\bullet}(\log D)
\simeq
\bigoplus_{i=0}^{d}
\Omega_{X/k}^i(\log D)[-i]
\end{equation}
in \(\mcalD_{\qcoh}(X)\).
\end{enumerate}
If \(X\) is projective, these equivalent conditions imply logarithmic
Akizuki--Nakano vanishing and the \(E_1\)-degeneration of the
logarithmic Hodge-to-de Rham spectral sequence (\Cref{vanishing}).
\end{theoremA}

The proof is based on the de Rham stack, which provides a geometric realization of the Frobenius pushforward of the de Rham complex. The de Rham stack was introduced by Simpson in characteristic zero \cite{simpson1996Homotopy} and developed by Bhatt in positive and mixed characteristic \cite{bhatt2022Prismatic}. For logarithmic de Rham complexes, we use the logarithmic de Rham stack introduced by Barz \cite{barz2025Logarithmic}.

The splitting condition also admits concrete cohomological
interpretations for several important classes of varieties.  For a
smooth Fano variety, HT-splitting is equivalent to the
Akizuki--Nakano vanishing theorem (\Cref{Fano-chara}). 
On the other hand, for a smooth proper variety with trivial canonical bundle, 
HT-splitting is equivalent to the
$E_1$-degeneration of the Hodge-to-de Rham spectral sequence
\eqref{eq:HdR-ss} (\Cref{chara-CY}).
Thus HT-splitting provides a common framework for two fundamental properties which appear naturally in the geometry of Fano and Calabi--Yau varieties.

We next establish several criteria and permanence properties for HT-splitting, summarized in the following theorem.

\begin{theoremA}\label{thm:intro-constructions}
Let \(k\) be a perfect field of characteristic \(p>0\).
\begin{enumerate}
\item
Let \(X\) be a smooth quasi-\(F\)-split \(k\)-variety, and let \(D\)
be a simple normal crossings divisor on \(X\).  Then \((X,D)\) is
log HT-split (\Cref{qFs-to-HT}).

\item
Let \(X\) be a smooth \(k\)-variety, and let \(D\)
be a simple normal crossings divisor on \(X\).
Let \((X,D)\) be a log HT-split pair, and let
\[
\pi\colon Y \defeq \operatorname{Bl}_Z(X) \to  X
\]
be the blow-up along a stratum \(Z\) of \(D\) and $D_Y \defeq (\pi^{-1}D)_{\red}$. 
Then \((Y,D_Y)\) is log HT-split.  In particular, \(Y\) is HT-split (\Cref{HT-ascent-stratum-blowup}).

\item
Let \(T\) be a toric variety with Cox ring \(S_T\).  
Suppose that a smooth
closed subvariety \(Z\subseteq T\) is defined by a homogeneous regular sequence.
Then \(Z\) is HT-split (\Cref{complete-intersection}).

\item Let $\pi \colon X \to Y$ be a morphism of smooth $k$-varieties.
If $\cO_Y \to R\pi_*\cO_X$ splits and $X$ is HT-split, then $Y$ is HT-split (\Cref{finite-cover}).
In particular, if $\pi$ is a linearly reductive good quotient and $X$ is HT-split, then $Y$ is HT-split.
\end{enumerate}
\end{theoremA}

The first assertion extends \cite{petrov2025Decomposition} to
arbitrary simple normal crossings boundaries.
Combining \textup{(1)} and \textup{(2)}, every blow-up of a quasi-$F$-split variety along a stratum of a simple normal crossings divisor satisfies the decomposition of its de Rham complex
\eqref{eq:DI-decomposition}.
We note that such a blow-up is not quasi-$F$-split in general (\Cref{example:b-up}).

The third assertion contains smooth complete intersections in
projective space as a special case.
For such varieties, the $E_1$-degeneration of the Hodge-to-de Rham
spectral sequence was already known by
\cite{DeligneCompleteIntersections}*{Proposition~1.3 and
Theorem~2.3}.
Akizuki--Nakano-type vanishing for smooth complete intersections in
projective space was also established directly in positive
characteristic by \cite{NomaWeightedCompleteIntersections}.
Thus, even in this classical case, \Cref{thm:intro-constructions}~\textup{(3)} upgrades the previously known cohomological consequences to a decomposition of the de Rham complex itself.
Beyond this case, its Cox-ring formulation yields many further HT-split varieties, including broad classes of Calabi--Yau complete intersections in toric varieties.

\subsection{Examples of HT-split varieties}

\Cref{thm:intro-constructions} yields a wide range of examples of HT-split varieties and log HT-split pairs.  
For all of these examples, the corresponding ordinary or logarithmic de Rham complex decomposes.  
In particular, the Hodge-to-de Rham spectral sequence degenerates at \(E_1\), and the corresponding Akizuki--Nakano vanishing theorem holds.
We summarize the main applications below.

\medskip
\noindent
\textbf{Fano threefolds and del Pezzo varieties (\Cref{example-Fano-varieties}).}
Kawakami and Tanaka established Akizuki--Nakano-type vanishing for smooth Fano threefolds in positive characteristic \cites{KawakamiTanakaFanoLiftabilityI,KT25II}.
By \Cref{Fano-chara}, this implies that every smooth Fano threefold
is HT-split.  Hence its de Rham complex admits the decomposition
\eqref{eq:DI-decomposition}.
On the other hand, del Pezzo varieties are quasi-\(F\)-split by \cite{KawakamiTanakaWeakQuasiFSplitting}.  
Our logarithmic result therefore shows that every simple normal crossings pair on a smooth del Pezzo variety is log HT-split.

\medskip
\noindent
\textbf{K3 surfaces and abelian varieties (\Cref{CY-examples}).}
K3 surfaces and abelian varieties over a perfect field are HT-split.  
For K3 surfaces, this follows from their \(W_2(k)\)-liftability.  
For abelian varieties, the decomposition \eqref{eq:DI-decomposition} follows from \cite{ZhangDeRhamHiggsComparison}.

\medskip
\noindent
\textbf{Casagrande--Druel Fano fourfolds (\Cref{cor:CD-HT-split}).}
In characteristic zero, the Casagrande--Druel construction produces a large and geometrically significant class of Fano fourfolds with Lefschetz defect two \cite{CD15}.
The case of Picard number three was classified in \cite{Sec23}, and the cases of higher Picard number were subsequently classified in
\cite{Pas25}.
Such a fourfold is obtained by blowing up a split projective bundle over a Fano threefold along a naturally defined codimension-two center.

We study an analogous construction over a perfect field of positive characteristic.
We prove that every Fano fourfold arising from this construction is HT-split.

\medskip
\noindent
\textbf{Cynk--Hulek Calabi--Yau varieties (\Cref{prop:Cynk-Hulek-HT}).}
Cynk--Hulek varieties are higher-dimensional Calabi--Yau varieties
of Kummer type.  They are obtained as crepant resolutions of finite
quotients of products of elliptic curves and were introduced in order
to construct explicit higher-dimensional modular Calabi--Yau
manifolds \cite{CynkHulek}.

Their standard resolution is built from products of elliptic curves by successively blowing up fixed strata and then taking a quotient by a group isomorphic to \((\setZ/2\setZ)^{n-1}\).  
Combining \Cref{thm:intro-constructions} (2) and (4), every Cynk--Hulek Calabi--Yau variety in characteristic \(p>2\) is HT-split.

\medskip
\noindent
\textbf{Godeaux--Serre varieties and the HKR spectral sequence (\Cref{Godeaux-Serre-quotients}).}
\Cref{thm:intro-constructions} (3) and (4) also give an instructive class of Godeaux--Serre varieties.  
Let \(G\) be a finite linearly reductive group scheme, and let \(Z\) be a smooth complete intersection in a projective space carrying a free \(G\)-action.
Since \(Z\) is HT-split, so is the smooth quotient $X \defeq Z/G$.

The case \(G=\mu_p\) gives a particularly striking application.
Antieau, Bhatt, and Mathew constructed a smooth projective
\(2p\)-dimensional Godeaux--Serre variety \(X\) for which the
Hochschild--Kostant--Rosenberg (HKR) spectral sequence does not degenerate;
see \cite{AntieauBhattMathewHKR}*{Theorem~1.1 and Section~6}.
On the other hand, since \(X\) is HT-split, its Hodge-to-de Rham
spectral sequence degenerates at \(E_1\).
The relationship between the degeneration of the Hodge-to-de Rham
and Hochschild--Kostant--Rosenberg spectral sequences has been studied
in \cite{AntieauBhattMathewHKR}*{Remark~3.6} and
\cite{devalapurkar2025Lifting}*{Theorem~2}.

\medskip
\noindent
\textbf{Symmetric products of curves (\Cref{symmetric-product-HT-split}).}
We also give an HT-splitting criterion for smooth projective
fibrations, formulated in terms of \(W_2\)-liftability, the dimension
of the base, and Bott-type vanishing on the fibers.
As a first consequence, every projective bundle over a smooth projective curve is HT-split.

A more substantial application concerns symmetric products of curves.
Let \(C\) be a smooth projective curve of genus \(g\) over a perfect field of characteristic \(p>0\).
If \(m \geq 2g-1\) and \(p \geq g\),  then the symmetric product $\operatorname{Sym}^m(C)$ is HT-split.  

\subsection{Applications to vanishing in mixed characteristic}
Finally, we apply HT-splitting to Akizuki--Nakano-type vanishing for globally \(+\)-regular schemes in mixed characteristic.
Globally \(+\)-regularity is a mixed-characteristic analogue of global \(F\)-regularity introduced in \cite{BMPSTWW23} in connection with the minimal model program in mixed characteristic; see also
\cite{TY23}.
Our main mixed-characteristic application is the following.

\begin{theoremA}[\Cref{GPS-to-ANV}]\label{intro:ANV-G+R}
Let \(k\) be a perfect field of characteristic \(p > 0\) and let \(X\) be a smooth projective
scheme over \(W(k)\).
If \(X\) is globally \(+\)-regular, then its closed fiber $X_k \defeq X\times_{W(k)}k$ is HT-split.
Furthermore, we obtain the following Akizuki--Nakano-type vanishing in mixed characteristic: for every ample line bundle \(L\) on \(X\),
\[
R^i\Gamma_{(p)}R\Gamma\left(
X,
\Omega^j_{X/W(k)}\otimes L^{-1}
\right)
=0
\qquad
\text{whenever }i+j<\dim X.
\]
In particular, this vanishing implies \(H^i(X, \Omega^j_{X/W(k)} \otimes L^{-1}) = 0\) for \(i + j < \dim X - 1\).
\end{theoremA}

For \(j=0\), this recovers the Kodaira-type vanishing theorem of
\cite{bhatt2021CohenMacaulayness}; thus our result may be viewed as
an Akizuki--Nakano refinement of that theorem.

More generally, we can show that if \(X\) is \emph{lim-perfectoid split}, a notion introduced in \cite{ishizuka2026Localglobal}, then \(X_k\) is HT-split (\Cref{LimPerfectoidSplitImpliesHTSplit}).

In forthcoming work, using the absolute and relative prismatizations and Hodge--Tate stacks introduced in \cites{drinfeld2024Prismatization,bhatt2022Prismatization}, we will introduce mixed-characteristic HT-splitting, study decomposition theorems for relative Hodge--Tate cohomology and compare these notions with HT-splitting of the closed fibers.

\subsection*{Notation and terminology}

Throughout this paper, \(p\) is a prime number and we freely use the following notation and terminology.

\subsubsection{Animated and derived rings}
\begin{enumerate}
  \item We freely use the notion of \emph{animated rings} following \cites{cesnavicius2024Purity,raksit2026Hochschild}. We denote by \(\CAlg^{an}\) (resp., \(\DAlg\)) the \(\infty\)-category of animated rings (resp., derived rings). Note that there is a fully faithful embedding \(\CAlg^{an} \hookrightarrow \DAlg\) of \(\infty\)-categories whose essential image is the full subcategory of \(\DAlg\) spanned by derived rings \(A\) such that \(\pi_i(A) = 0\) for \(i < 0\).
  \item To get rid of the ambiguity of the terminology, when we say \emph{ring} or \emph{discrete ring} \(A\), we mean that \(A\) is just a usual commutative ring, i.e., \(A\) is an object of the ordinary category of commutative rings \(\CRing\) (or equivalently, \(A\) is a derived ring such that \(\pi_i(A) = 0\) for \(i \neq 0\)).
  \item Based on the theory of \(\delta\)-rings developed in \cites{buium1997Arithmetic,joyal1985Deltaanneaux,bhatt2022Prismsa}, there are notions of \emph{animated \(\delta\)-rings} and \emph{derived \(\delta\)-rings}, which are given in \cites{bhatt2022Prismatization, holeman2023Derived}.
  \item For any animated ring \(A\), or more generally, any derived ring \(A\), we can define the \emph{(\(p\)-typical) Witt vector ring} \(W(A)\) following \cites{bhatt2022Prismatization,magidson2025Witta}. When \(A\) is discrete, the Witt vector ring \(W(A)\) is the classical \(p\)-typical Witt vector ring of \(A\).
\end{enumerate}

\subsubsection{\(p\)-adic stuff}
\begin{enumerate}
  \item We freely use the notion of \emph{perfectoid rings} and \emph{prisms} following \cites{bhatt2018Integral,bhatt2022Prismsa}.
  \item A ring \(R\) is called \emph{\(p\)-quasisyntomic} if \(R\) is \(p\)-adically complete with bounded \(p^{\infty}\)-torsion, and the cotangent complex \(L_{R/\Z_p} \in \mcalD(R)\) has \(p\)-complete Tor-amplitude in \([-1, 0]\). A morphism \(R \to S\) of \(p\)-adically complete rings is called a \emph{\(p\)-quasisyntomic cover} if it is \(p\)-completely faithfully flat and the cotangent complex \(L_{S/R} \in \mcalD(S)\) has \(p\)-complete Tor-amplitude in \([-1, 0]\); see \cites{bhatt2019Topologicala,bhatt2022Absolute}. Note that a \(p\)-adically complete Noetherian ring \(R\) is \(p\)-quasisyntomic if and only if \(R\) is a locally complete intersection ring  (\cite{avramov1999Locally}*{Theorem 1.2}).
\end{enumerate}

\subsubsection{Derived algebraic geometry}
\begin{enumerate}
  \item In this paper, a \emph{stack} \(\mcalX\) is a functor from \(\CAlg^{an}\) (or often its full subcategory) to the \(\infty\)-category of anima \(\Ani\) which satisfies the fpqc descent condition.
  \item A \emph{representable morphism} \(f \colon \mcalX \to \mcalY\) of stacks is a morphism such that for any morphism \(T \to \mcalY\) from a scheme \(T\), the fiber product \(\mcalX \times_{\mcalY} T\) is a scheme (not an algebraic space).
  \item An \emph{fpqc algebraic stack} is a stack \(\mcalX\) such that there exists a representable fpqc cover \(\mcalU \to \mcalX\) by a scheme \(\mcalU\) and the diagonal morphism \(\Delta \colon \mcalX \to \mcalX \times \mcalX\) is representable by schemes.
  \item For a stack \(\mcalX\), we define the \(\infty\)-category \(\mcalD(\mcalX_{\fpqc}, \mcalO_{\mcalX})\) as the \(\infty\)-category of \(\mcalO_{\mcalX}\)-modules on the \(\infty\)-topos \(\Shv(\mcalX_{\fpqc})\) of sheaves of anima on the big fpqc site \(\mcalX_{\fpqc}\) of \(\mcalX\). On the other hand, following \cite{lurie2018Spectral}*{\S 6.2.2}, we define the \(\infty\)-category \(\mcalD_{\qcoh}(\mcalX)\) of quasi-coherent \(\mcalO_{\mcalX}\)-modules on \(\mcalX\), which is a full subcategory of \(\mcalD(\mcalX_{\fpqc}, \mcalO_{\mcalX})\) spanned by objects satisfying the quasi-coherence condition.
  \item The site-theoretic and quasi-coherent derived functors associated with a morphism of stacks are defined in \Cref{SiteTheoreticDerivedFunctors,QuasiCoherentDerivedFunctors}. Unless the superscript \(\mathrm{sh}\) is explicitly displayed, the notation \(Rf_*\) always means the right adjoint on quasi-coherent complexes.
\end{enumerate}

\subsection*{Acknowledgments}
The authors are grateful to Yoshinori Gongyo and Kojiro Matsumoto for helpful discussions and suggestions.
We thank Michael Barz, Bhargav Bhatt, Alexander Petrov, and Hiromu Tanaka for valuable comments on an earlier version of this work.
When preparing this manuscript, we discuss with AI.
The first-named author was supported by JSPS KAKENHI Grant number 24KJ1085.
The second-named author was supported by JSPS KAKENHI Grant number JP24K16889.

\section{Preliminaries} \label{SectionPreliminaries}

\subsection{Preliminaries on derived algebraic geometry}

We begin by fixing our conventions for derived pullback and
pushforward functors.

\begin{definition}
  \label{SiteTheoreticDerivedFunctors}
  Let \(f\colon\mcalY\to\mcalX\) be a morphism of stacks. By \citeSta{06NW}, the induced functor between the inherited fpqc sites gives a morphism of topoi. Moreover, the canonical morphism \(f^{-1}\mcalO_{\mcalX}\to\mcalO_{\mcalY}\) is an equivalence: on every object \(T\to\mcalY\), both sides are the ring \(\Gamma(T,\mcalO_T)\). We therefore obtain a morphism of ringed fpqc topoi. We denote the resulting derived adjunction by
  \begin{equation*}
    Lf_{\mathrm{sh}}^*\colon\mcalD(\mcalX_{\fpqc},\mcalO_{\mcalX})\rightleftarrows\mcalD(\mcalY_{\fpqc},\mcalO_{\mcalY})\colon Rf_*^{\mathrm{sh}};
  \end{equation*}
  see \citeSta{07A6}. Thus
  \begin{equation*}
    Lf_{\mathrm{sh}}^*(\mathord{-})\defeq\mcalO_{\mcalY}\otimes_{f^{-1}\mcalO_{\mcalX}}^Lf^{-1}(\mathord{-}).
  \end{equation*}
  We call \(Lf_{\mathrm{sh}}^*\) the \emph{site-theoretic derived pullback} and \(Rf_*^{\mathrm{sh}}\) the \emph{site-theoretic derived pushforward}.
\end{definition}

\begin{definition}
  \label{QuasiCoherentDerivedFunctors}
  For every stack \(\mcalX\), write
  \begin{equation*}
    \iota_{\mcalX}\colon\mcalD_{\qcoh}(\mcalX)\hookrightarrow\mcalD(\mcalX_{\fpqc},\mcalO_{\mcalX})
  \end{equation*}
  for the canonical fully faithful functor. Let \(f\colon\mcalY\to\mcalX\) be a morphism of stacks. Under the description of quasi-coherent complexes as compatible systems on affine objects in \cite{lurie2018Spectral}*{\S 6.2.2}, the site-theoretic derived pullback preserves quasi-coherent complexes. We denote its restriction by
  \begin{equation*}
    Lf^*\colon\mcalD_{\qcoh}(\mcalX)\to\mcalD_{\qcoh}(\mcalY).
  \end{equation*}
  Thus there is a canonical equivalence
  \begin{equation}
    \label{SiteQcohPullbackCompatibility}
    Lf_{\mathrm{sh}}^*\iota_{\mcalX}\xrightarrow{\sim}\iota_{\mcalY}Lf^*.
  \end{equation}
  Both sides of \eqref{SiteQcohPullbackCompatibility} are obtained by restriction along the functor from affine schemes over \(\mcalY\) to affine schemes over \(\mcalX\), followed by derived extension of scalars. Hence this functor is canonically the pullback functor constructed in \cite{lurie2018Spectral}*{\S 6.2.2}.

  Since \(\mcalD_{\qcoh}(\mcalX)\) and \(\mcalD_{\qcoh}(\mcalY)\) are presentable stable \(\infty\)-categories and \(Lf^*\) preserves small colimits, \(Lf^*\) admits a right adjoint. We denote this right adjoint by
  \begin{equation*}
    Rf_*\colon\mcalD_{\qcoh}(\mcalY)\to\mcalD_{\qcoh}(\mcalX)
  \end{equation*}
  and call it the \emph{derived pushforward}. Unless the superscript \(\mathrm{sh}\) is explicitly displayed, \(Rf_*\) always denotes this quasi-coherent right adjoint.

  If \(\mcalX\) is an fpqc algebraic stack, we write
  \begin{equation*}
    \mcalD^+_{\qcoh}(\mcalX)\defeq\bigcup_{c\in\setZ}\mcalD^{\geq c}_{\qcoh}(\mcalX)
  \end{equation*}
  for the full subcategory of cohomologically bounded-below quasi-coherent complexes, where \(\mcalD^{\geq c}\) means that \(H^i=0\) for every integer \(i<c\).
\end{definition}

\begin{remark}
Let \(f \colon Y\to X\) be a qcqs morphism of schemes.
Then the usual unbounded scheme-theoretic derived pushforward coincides with $Rf_*$ (see \citeSta{08D5}).
We use the same notation for them.
\end{remark}

\begin{remark} \label{BoundedBelow}
By \cite{lurie2018Spectral}*{Proposition 6.2.5.2~(1) and Definition 6.2.5.3}, the standard \(t\)-structure on quasi-coherent complexes is detected on affine points. Consequently, since the base change along morphisms of animated rings preserves connectivity, for every morphism \(f\colon\mcalY\to\mcalX\), the pullback
\begin{equation*}
  Lf^*\colon\mcalD_{\qcoh}(\mcalX)\longrightarrow\mcalD_{\qcoh}(\mcalY)
\end{equation*}
is right \(t\)-exact, namely, \(Lf^*(\mcalD_{\qcoh}^{\leq 0}(\mcalX)) \subseteq \mcalD_{\qcoh}^{\leq 0}(\mcalY)\).
Since \(Lf^*\dashv Rf_*\), the right \(t\)-exactness of the left adjoint implies that the right adjoint \(Rf_*\) is left \(t\)-exact. Hence
\begin{equation*}
  Rf_*\bigl(\mcalD_{\qcoh}^{\geq n}(\mcalY)\bigr) \subseteq \mcalD_{\qcoh}^{\geq n}(\mcalX)
\end{equation*}
for every \(n\in\setZ\), and in particular \(Rf_*\) preserves cohomologically bounded-below objects.
\end{remark}

\begin{definition}
  \label{QuasiCoherentBaseChangeMorphism}
  Consider a pullback square of stacks
  \begin{equation*}
    \begin{tikzcd}
      \mcalY' \arrow[r,"g'"] \arrow[d,"f'"'] & \mcalY \arrow[d,"f"] \\
      \mcalX' \arrow[r,"g"'] & \mcalX.
    \end{tikzcd}
  \end{equation*}
  For \(\mcalF\in\mcalD_{\qcoh}(\mcalY)\), the canonical equivalence \(Lf'^*Lg^*\simeq L(g')^*Lf^*\) and the counit \(Lf^*Rf_*\mcalF\to\mcalF\) give a morphism
  \begin{equation*}
    Lf'^*Lg^*Rf_*\mcalF\xrightarrow{\sim}L(g')^*Lf^*Rf_*\mcalF\to L(g')^*\mcalF.
  \end{equation*}
  Its adjoint under \(Lf'^*\dashv Rf'_*\) is denoted by
  \begin{equation}
    \label{QcohBaseChangeMorphism}
    \BC_{f,g}(\mcalF)\colon Lg^*Rf_*\mcalF\to Rf'_*L(g')^*\mcalF
  \end{equation}
   and is called the \emph{base change morphism}. This is a morphism in \(\mcalD_{\qcoh}(\mcalX')\).
\end{definition}

\begin{remark}
    By \cite{gaitsgory2017Study}*{Chapter 3, Proposition 2.2.2}, the morphism \(\BC_{f, g}(\mcalF)\) is an isomorphism if \(f\) is representable and quasi-compact.
\end{remark}

We recall the following fpqc hyperdescent property of quasi-coherent complexes on algebraic stacks:

\begin{lemma}
  \label{lemma:CAlg-qcoh-fpqc-hyperdescent}
  Let \(\mcalX\) be an fpqc algebraic stack and let \(U_{\bullet}\to\mcalX\) be a representable fpqc hypercover by schemes. Then pullback induces an equivalence
  \begin{equation*}
    \CAlg\parenlr{\mcalD_{\qcoh}(\mcalX)}
    \xrightarrow{\simeq}
    \lim_{[n]\in\Delta}
    \CAlg\parenlr{\mcalD_{\qcoh}(U_n)}.
  \end{equation*}
\end{lemma}

\begin{proof}
  Let \(\PStk \defeq \Fun(\CAlg^{an}, \Ani)\) denote the \(\infty\)-category of prestacks on \(\CAlg^{an}\), and let
  \begin{equation*}
    j\colon (\CAlg^{an})^{\opposite} \to \PStk; \qquad R \mapsto \Spec(R)
  \end{equation*}
  be the Yoneda embedding. Consider the functor
  \begin{equation*}
    \mscrM^\otimes\colon \CAlg^{an} \to \CAlg(\Pr\nolimits_{\mathrm{st}}^L), \qquad \Spec(R) \mapsto \mcalD(R)^\otimes.
  \end{equation*}
  By \cite{lurie2018Spectral}*{Corollary~D.6.3.3}, the functor \(\mscrM^\otimes\) satisfies hyperdescent for the fpqc topology.

  By \cite{lurie2018Spectral}*{Definition 6.2.2.1} and \cite{mathew2025Affine}*{Definition~3.14}, the functor
  \begin{equation*}
    \PStk^{\opposite} \ni Y \mapsto \mcalD_{\qcoh}(Y)^\otimes \in \CAlg(\Pr\nolimits^L_{\mathrm{st}})
  \end{equation*}
  is the right Kan extension of \(\mscrM^\otimes\) along \(j^{\opposite}\). Explicitly, there is a natural equivalence
  \begin{equation}
    \label{eq:qcoh-right-kan}
    \mcalD_{\qcoh}(Y)^\otimes
    \simeq
    \lim_{\Spec(R)\to Y}\mcalD(R)^\otimes.
  \end{equation}

  We explain why this right Kan extension computes descent along the fixed hypercover \(U_\bullet\to\mcalX\). By the universal property of the presheaf category, or equivalently by the dual of \cite{lurie2009Higher}*{Theorem~5.1.5.6}, the right Kan extension in \eqref{eq:qcoh-right-kan} carries colimits of prestacks to limits. Hence, if
  \begin{equation*}
    \left|U_\bullet\right|
    \defeq
    \colim_{[n]\in\Delta^{\opposite}}U_n
  \end{equation*}
  denotes the geometric realization in prestacks, then
  \begin{equation}
    \label{eq:qcoh-realization-limit}
    \mcalD_{\qcoh}\parenlr{\left|U_\bullet\right|}^\otimes
    \simeq
    \lim_{[n]\in\Delta}
    \mcalD_{\qcoh}(U_n)^\otimes.
  \end{equation}

  Since \(U_\bullet\to\mcalX\) is an fpqc hypercover, it becomes an equivalence after fpqc hypercomplete sheafification by \cite{lurie2009Higher}*{Theorem~6.5.3.12}. Since \(\mscrM^\otimes\) satisfies fpqc hyperdescent, its right Kan extension factors through fpqc hypercomplete sheafification. Consequently, the augmentation induces an equivalence
  \begin{equation*}
    \mcalD_{\qcoh}(\mcalX)^\otimes
    \xrightarrow{\simeq}
    \mcalD_{\qcoh}\parenlr{\left|U_\bullet\right|}^\otimes.
  \end{equation*}
  Combining this with \eqref{eq:qcoh-realization-limit}, we obtain
  \begin{equation*}
    \mcalD_{\qcoh}(\mcalX)^\otimes
    \xrightarrow{\simeq}
    \lim_{[n]\in\Delta}
    \mcalD_{\qcoh}(U_n)^\otimes.
  \end{equation*}

  Finally, by \cite{lurie2017Higher}*{Proposition~3.2.2.1 and Corollary~3.2.2.5}, the formation of commutative algebra objects commutes with limits of symmetric monoidal \(\infty\)-categories. Applying this to the preceding equivalence gives
  \begin{equation*}
    \CAlg\parenlr{\mcalD_{\qcoh}(\mcalX)}\xrightarrow{\simeq}\lim_{[n]\in\Delta}\CAlg\parenlr{\mcalD_{\qcoh}(U_n)},
  \end{equation*}
  as desired.
\end{proof}

Moreover, we recall some definitions of morphisms of stacks.

\begin{definition}
  \label{BGTorsorDefinition}
  Let \(\mcalG\) be a flat affine group scheme over an fpqc algebraic stack \(\mcalX\). A morphism \(f\colon\mcalY\to\mcalX\) of stacks is called a \emph{\(B\mcalG\)-torsor} if there exists a representable fpqc cover \(U\to\mcalX\) by a scheme and an equivalence
  \begin{equation*}
    \mcalY\times_{\mcalX}U\simeq B_U\mcalG_U
  \end{equation*}
  of stacks over \(U\), where \(\mcalG_U\defeq\mcalG\times_{\mcalX}U\). We use the expression \(B\mcalG\)-torsor only in this sense.
\end{definition}

We next establish the base-change results needed later in the paper.

\begin{lemma}[{\cite{bhatt2022Prismsa}*{Lemma 4.22}}]
  \label{TotalizationBaseChange}
  Let \(A\to A'\) be a flat morphism of commutative rings, and let \(K^\bullet\) be a cosimplicial object of \(\mcalD(A)\). Assume that there exists an integer \(c\) such that \(K^n\in\mcalD^{\geq c}(A)\) for every integer \(n\geq0\). Then the canonical morphism
  \begin{equation}
    \label{TotalizationBaseChangeMap}
    A'\otimes_A^L\Tot(K^\bullet)\to\Tot(A'\otimes_A^LK^\bullet)
  \end{equation}
  is an equivalence in \(\mcalD(A')\).
\end{lemma}

\begin{proof}
  This is a special case of \cite{bhatt2022Prismsa}*{Lemma 4.22}.
\end{proof}

We now give a calculation of the derived pushforward along \(B\mcalG\)-torsors.

\begin{lemma}[{cf.~\cite{bhatt2022Prismatic}*{Remark 2.5.5}}]
  \label{AffineGerbePushforwardCalculation}
  Let \(A\) be a commutative ring, let \(\mcalG\) be a flat affine group scheme over \(U\defeq\Spec(A)\), and let
  \begin{equation*}
    f\colon B_U\mcalG\to U
  \end{equation*}
  be the structure morphism.
  Let \(q\colon U\to B_U\mcalG\) be the canonical fpqc cover, let \(q_\bullet\colon U^\bullet\to B_U\mcalG\) be its \v{C}ech nerve, and let \(p_n\colon U^n\to U\) be the structure morphism, which can be identified with the projection \(\mcalG^n \to U\).

  For \(\mcalF\in\mcalD^+_{\qcoh}(B_U\mcalG)\), put
  \begin{equation*}
    \mcalF_n\defeq Lq_n^*\mcalF \in \mcalD_{\qcoh}(U^n),\qquad K_{\mcalF}^n\defeq R(p_n)_{*}\mcalF_n \in \mcalD_{\qcoh}(U) \cong \mcalD(A).
  \end{equation*}
  Then the following assertions hold.
  \begin{enumerate}
    \item There is a canonical equivalence
    \begin{equation}
      \label{AffineGerbePushforwardTotalization}
      Rf_*\mcalF\xrightarrow{\sim}\Tot(K_{\mcalF}^\bullet)
    \end{equation}
    in \(\mcalD_{\qcoh}(U)\simeq\mcalD(A)\).
    \item If \(\mcalF\in\mcalD^{\geq c}_{\qcoh}(B_U\mcalG)\), then \(K_{\mcalF}^n\in\mcalD^{\geq c}(A)\) for every integer \(n\geq0\).
    \item For every flat ring homomorphism \(A\to A'\), write \(g\colon U'\defeq\Spec(A')\to U\) for the induced morphism, put \(\mcalG_{U'}\defeq\mcalG\times_UU'\), and write \(f'\colon B_{U'}\mcalG_{U'}\to U'\) for the structure morphism, and \(g'\colon B_{U'}\mcalG_{U'}\to B_U\mcalG\) for the projection. Then the base change morphism
    \begin{equation}
      \label{AffineGerbeFlatBaseChange}
      A'\otimes_A^LRf_*\mcalF\to Rf'_*L(g')^*\mcalF
    \end{equation}
    is an equivalence in \(\mcalD_{\qcoh}(U')\simeq\mcalD(A')\).
  \end{enumerate}
\end{lemma}

\begin{proof}
  (1): Let \(\mcalM\in\mcalD_{\qcoh}(U)\). 
  Adjunction \(Lp_n^* \dashv R(p_n)_{*}\) and fpqc descent for quasi-coherent complexes give
  \begin{align*}
    \Map_{\mcalD_{\qcoh}(U)}(\mcalM,\Tot(K_{\mcalF}^\bullet))
    &\simeq\Tot_{[n]\in\Delta}\Map_{\mcalD_{\qcoh}(U)}(\mcalM,K_{\mcalF}^n) \\
    &\simeq\Tot_{[n]\in\Delta}\Map_{\mcalD_{\qcoh}(U^n)}(Lp_n^*\mcalM,\mcalF_n) \\
    &\simeq\Map_{\mcalD_{\qcoh}(B_U\mcalG)}(Lf^*\mcalM,\mcalF).
  \end{align*}
  Hence \(\Tot(K_{\mcalF}^\bullet)\) represents the right adjoint to \(Lf^*\) at \(\mcalF\), proving \eqref{AffineGerbePushforwardTotalization}.

  (2): Suppose that \(\mcalF\in\mcalD^{\geq c}_{\qcoh}(B_U\mcalG)\). Since every \(q_n \colon U^n \to B_U\mcalG\) is flat, \(\mcalF_n\in\mcalD^{\geq c}_{\qcoh}(U^n)\). Then the conclusion follows from \Cref{BoundedBelow}.

  (3): Let \(A\to A'\) be flat.
  Using the affine base change formula along the pullback square defined by \(g \colon U' \to U\) and \(p_n \colon U^n \to U\), we obtain canonical equivalences
  \begin{equation*}
    A'\otimes_A^LK_{\mcalF}^n\xrightarrow{\sim}K_{L(g')^*\mcalF}^n
  \end{equation*}
  in \(\mcalD(A')\) for every integer \(n\geq0\). Applying \Cref{TotalizationBaseChange} and \eqref{AffineGerbePushforwardTotalization}, we obtain
  \begin{equation*}
    A' \otimes_A^LRf_*\mcalF \simeq A' \otimes_A^L\Tot(K_{\mcalF}^\bullet) \simeq \Tot(A' \otimes_A^LK_{\mcalF}^\bullet) \simeq \Tot(K_{L(g')^*\mcalF}^\bullet) \simeq Rf'_*L(g')^*\mcalF
  \end{equation*}
  in \(\mcalD(A')\).
  The resulting equivalence is \eqref{AffineGerbeFlatBaseChange} so the base change morphism is an equivalence.
\end{proof}

\begin{proposition}
  \label{RemarkPushforward}
  Let \(f\colon\mcalY\to\mcalX\) be a \(B\mcalG\)-torsor for a flat affine group scheme \(\mcalG\) over an fpqc stack \(\mcalX\).
  Let \(u\colon V\to\mcalX\) be an fpqc cover such that $V$ is a small disjoint union of affine schemes and
  \[
  \mcalY \times_{\mcalX} V \simeq B_V\mcalG_V
  \]
  over $V$.
  We consider the pullback square
  \begin{equation*}
    \begin{tikzcd}
      \mcalY_V \arrow[r,"v"] \arrow[d,"f_V"'] & \mcalY \arrow[d,"f"] \\
      V \arrow[r,"u"'] & \mcalX.
    \end{tikzcd}
  \end{equation*}
  For $\mcalF \in \mcalD^+_{\qcoh}(\mcalY)$, the base change morphism
    \begin{equation}
      \label{GerbeFlatAffinePushforward}
      \BC_{f, u}(\mcalF) \colon Lu^*Rf_*\mcalF\to R(f_V)_*Lv^*\mcalF
    \end{equation}
    is an equivalence in \(\mcalD_{\qcoh}(V)\).
\end{proposition}

\begin{proof}
Choose a representable fpqc hypercover \(v^{\bullet} \colon V^{\bullet} \to\mcalX\) such that $V^0=V$ and every $V^i$ is a small disjoint union of affine schemes.
Put
  \begin{equation*}
    \mcalY^n\defeq\mcalY\times_{\mcalX}V^n,\qquad f_n\colon\mcalY^n\to V^n.
  \end{equation*}
  Let \(\mcalF_n\) denote the pullback of \(\mcalF\) to \(\mcalY^n\).
  Since \(\mcalY^n \to \mcalY\) is flat, \(\mcalF_n\) belongs to \(\mcalD_{\qcoh}^+(\mcalY^n)\).

   For each morphism \(s \colon V^m \to V^n\), we have a pullback diagram
    \begin{equation*}
    \begin{tikzcd}
    \mcalY^m \arrow[r, "s'"] \arrow[d, "f_{m}"] & \mcalY^n \arrow[d, "f_{n}"] \\
    V^m \arrow[r, "s"]                               & V^n.                        
    \end{tikzcd}
    \end{equation*}
    Since \(V^m\) and \(V^n\) are affine and \(\mcalY^m\) and \(\mcalY^n\) are trivial \(B\mcalG\)-torsors, we can apply \Cref{AffineGerbePushforwardCalculation}(3) to give compatible isomorphisms
    \begin{equation*}
        Ls^*R(f_{n})_*\mcalF_n \xrightarrow{\cong} R(f_{m})_*L(s')^*\mcalF_n
    \end{equation*}
    in \(\mcalD_{\qcoh}(V^m)\).

    Quasi-coherent complexes satisfy fpqc descent by \Cref{lemma:CAlg-qcoh-fpqc-hyperdescent}. Therefore the preceding descent datum gives an object \(E \in \mcalD_{\qcoh}(\mcalX)\) whose pullback to \(V^n\) is canonically \(R(f_n)_{*}\mcalF_n\) for every integer \(n\geq0\).
    
    Let \(\mcalM\in\mcalD_{\qcoh}(\mcalX)\), and let \(\mcalM_n\) denote its pullback to \(V^n\). Fpqc descent for mapping spaces and the scheme-theoretic adjunctions give
    \begin{align*}
      \Map_{\mcalD_{\qcoh}(\mcalX)}(\mcalM,E)
      &\simeq\Tot_{[n]\in\Delta}\Map_{\mcalD_{\qcoh}(V^n)}(\mcalM_n, R(f_n)_*\mcalF_n) \\
      &\simeq\Tot_{[n]\in\Delta}\Map_{\mcalD_{\qcoh}(\mcalY^n)}(Lf_n^*\mcalM_n,\mcalF_n) \\
      &\simeq\Map_{\mcalD_{\qcoh}(\mcalY)}(Lf^*\mcalM,\mcalF).
    \end{align*}
    Thus \(E\) represents the right adjoint to \(Lf^*\) at \(\mcalF\), and hence there is a canonical equivalence
    \begin{equation}
      \label{RepresentablePushforwardDescentIdentification}
      E\xrightarrow{\sim} Rf_*\mcalF \in \mcalD_{\qcoh}(\mcalX).
    \end{equation}
    This shows that
    \[
    L(V^n \to \mcalX)^*(Rf_*\mcalF) \simeq R(f_n)_*\mcalF_n. 
    \]
    In particular, the base change morphism \eqref{GerbeFlatAffinePushforward} is an equivalence.
\end{proof}

\begin{theorem}
  \label{FlatBaseChange}
  Consider a pullback square of fpqc algebraic stacks
  \begin{equation*}
    \begin{tikzcd}
      \mcalY' \arrow[r,"g'"] \arrow[d,"f'"'] & \mcalY \arrow[d,"f"] \\
      \mcalX' \arrow[r,"g"'] & \mcalX.
    \end{tikzcd}
  \end{equation*}
  Assume that \(g \colon \mcalX' \to \mcalX\) is flat and \(f \colon \mcalY \to \mcalX\) is a \(B\mcalG\)-torsor for a flat affine group scheme \(\mcalG\) over \(\mcalX\).
  For $\mcalF \in \mcalD_{\qcoh}^+(\mcalY)$, the base change morphism
  \begin{equation*}
    \BC_{f,g}(\mcalF)\colon Lg^*Rf_*\mcalF\to Rf'_*L(g')^*\mcalF
  \end{equation*}
  defined in \Cref{QuasiCoherentBaseChangeMorphism} is an equivalence in \(\mcalD_{\qcoh}(\mcalX')\).
\end{theorem}

\begin{proof}
Take an fpqc cover $u \colon V \to \mcalX$ such that $V$ is a small disjoint union of affine schemes and 
\[
\mcalY \times_{\mcalX} V \simeq B_{V}\mcalG_V.
\]
Take an fpqc cover $V' \to V \times_{\mcalX} \mcalX'$ such that $V'$ is a small disjoint union of affine schemes.
We consider the pullback diagrams
\[
\begin{tikzcd}
\mcalY \times_{\mcalX} V' \arrow[r] \arrow[d] & \mcalY \times_{\mcalX} V \arrow[r] \arrow[d] & \mcalY \arrow[d,"f"] \\
V' \arrow[r,"v"] & V \arrow[r,"u"] & \mcalX.
\end{tikzcd}
\]
By \Cref{RemarkPushforward}, the base change morphism $\BC_{f,u \circ v}(\mcalF)$ is an equivalence.

Next, we consider the pullback diagrams
\[
\begin{tikzcd}
\mcalY \times_{\mcalX} V' \arrow[r,"u''"] \arrow[d,"f''"] & \mcalY \times_{\mcalX} \mcalX' \arrow[r,"g'"] \arrow[d,"f'"] & \mcalY \arrow[d,"f"] \\
V' \arrow[r,"u'"] & \mcalX' \arrow[r,"g"] & \mcalX.
\end{tikzcd}
\]
Since $g'$ is flat, $L(g')^*\mcalF$ is bounded below.
By \Cref{RemarkPushforward} again, the morphism $\BC_{f',u'}(L(g')^*\mcalF)$ is an equivalence since \(u' \colon V' \to V \times_{\mcalX} \mcalX' \to \mcalX'\) is an fpqc cover and trivializes the gerbe \(f'\).

Therefore, we have
\begin{align*}
L(u')^*\BC_{f,g}(\mcalF)
&\simeq (L(g \circ u')^*Rf_*\mcalF \to L(u')^*Rf'_*L(g')^*\mcalF) \\
&\simeq (L(g \circ u')^*Rf_*\mcalF \to Rf''_*L(g' \circ u'')^*\mcalF) \\
&\simeq \BC_{f,g \circ u'}(\mcalF)\simeq \BC_{f,u \circ v}(\mcalF).
\end{align*}
Thus, $L(u')^*\BC_{f,g}(\mcalF)$ is an equivalence.
Since $u'$ is an fpqc cover, the functor $L(u')^*$ is conservative.
Thus, the base change morphism $\BC_{f,g}(\mcalF)$ is an equivalence, as desired.
\end{proof}

\begin{proposition}
\label{E-sharp-gerbe-arbitrary-base-change}
  Consider a pullback square of fpqc algebraic stacks
  \begin{equation*}
    \begin{tikzcd}
      \mcalY' \arrow[r,"g'"] \arrow[d,"f'"'] & \mcalY \arrow[d,"f"] \\
      \mcalX' \arrow[r,"g"'] & \mcalX.
    \end{tikzcd}
  \end{equation*}
  Assume that  \(f \colon \mcalY \to \mcalX\) is a \(BE^\sharp\)-torsor for a vector bundle \(E\) of finite rank on $\mcalX$.
  Then the base change morphism
  \begin{equation*}
    \BC_{f,g}(\cO_{\mcalY})\colon Lg^*Rf_*\cO_{\mcalY}\to Rf'_*\cO_{\mcalY'}
  \end{equation*}
  defined in \Cref{QuasiCoherentBaseChangeMorphism} is an equivalence in \(\mcalD_{\qcoh}(\mcalX')\).
\end{proposition}

\begin{proof}
Choose a representable fpqc cover
\[
u\colon U\longrightarrow \mcalX
\]
by a scheme which trivializes the gerbe $\mcalY$. Writing
$E_U \defeq u^*E$, there is an equivalence
\[
\mcalY\times_{\mcalX}U
\simeq
B_UE_U^\sharp
\]
over $U$.

Set
\[
U' \defeq U\times_{\mcalX}\mcalX'.
\]
Choose a representable fpqc cover
\[
v\colon T\longrightarrow U'
\]
by a scheme $T$, and denote by
\[
h\colon T\longrightarrow \mcalX',
\qquad
a\colon T\longrightarrow U
\]
the induced morphisms. 
Then $h$ is an fpqc cover.

Let
\[
\BC_{f,g}(\cO_{\mcalY})\colon
Lg^*Rf_*\cO_{\mcalY}
\longrightarrow
Rf'_*\cO_{\mcalY'}
\]
be the base change morphism. 
Since $h^*\colon \mcalD_{\qcoh}(\mcalX') \to \mcalD_{\qcoh}(T)$ is conservative, it suffices to show that $Lh^*\BC_{f,g}(\cO_{\mcalY})$ is an equivalence.
Furthermore, by the flatness of $h$ and \Cref{FlatBaseChange}, we have
\[
Lh^*\BC_{f,g}(\cO_{\mcalY}) \simeq \BC_{f,g \circ h}(\cO_{\mcalY}) \simeq \BC_{f,u \circ a}(\cO_{\mcalY}).
\]
Since $u$ is flat, the morphism $\BC_{f, u}$ is an equivalence by \Cref{FlatBaseChange}.
Therefore, it suffices to show that the base change morphism for the diagram
\[
\begin{tikzcd}
\mcalY \times_{\mcalX} T \simeq B_TE_T^{\sharp} \arrow[r] \arrow[d,"f_T"] & \mcalY \times_{\mcalX} U \simeq B_UE_U^{\sharp} \arrow[d,"f_U"] \\
T \arrow[r,"a"] & U
\end{tikzcd}
\]
with respect to the structure sheaf, that is, the morphism
\[
\BC_{f_U,a}(\cO_{B_UE_U^{\sharp}}) \colon La^*R(f_U)_*\cO_{B_UE_U^{\sharp}} \to R(f_T)_*\cO_{B_TE_T^{\sharp}}
\]
is an equivalence.

By \cite{bhatt2022Prismatic}*{Remark~2.4.6} and \cite{bhatt2022Prismatic}*{Lemma~7.8}, we obtain
\begin{equation*}
    Rf_{U,*}\cO_{B_UE_U^\sharp} \simeq \bigoplus_{i=0}^r \bigwedge^iE_U^\vee[-i] \quad \text{and} \quad Rf_{T,*}\cO_{B_TE_T^\sharp} \simeq \bigoplus_{i=0}^r\bigwedge^iE_T^\vee[-i].
\end{equation*}
Since $E_U$ is finite locally free, arbitrary derived pullback commutes with duals and exterior powers. 
Hence
\[
\begin{aligned}
La^*Rf_{U,*}\cO_{B_UE_U^\sharp}
&\simeq
\bigoplus_{i=0}^r
La^*\bigl(\bigwedge^iE_U^\vee\bigr)[-i] \\
&\simeq
\bigoplus_{i=0}^r
\bigwedge^iE_T^\vee[-i] \\
&\simeq
Rf_{T,*}\cO_{B_TE_T^\sharp}.
\end{aligned}
\]
By the naturality of the above descriptions, the morphism $\BC_{f_U,a}(\cO_{B_UE_U^{\sharp}})$ is an equivalence, as desired.
\end{proof}

\begin{proposition}
  \label{prop:tensor-formula-for-banded-gerbe}
  Let \(\mcalX\) and \(\mcalY\) be fpqc algebraic stacks, let \(G\) be a flat affine commutative group scheme over \(\mcalX\), and let \(\pi \colon \mcalY \to \mcalX\) be a \(B_{\mcalX}G\)-torsor. Let \(q \colon B_{\mcalX}G \to \mcalX\) be the structure morphism of the classifying stack.
  Assume that both \(R\pi_*\mcalO_{\mcalY}\) and \(Rq_*\mcalO_{B_{\mcalX}G}\) are perfect objects of \(\mcalD_{\qcoh}(\mcalX)\).
  Then there is an isomorphism
  \begin{equation*}
    (R\pi_*\mcalO_{\mcalY}) \otimes^L_{\mcalO_{\mcalX}} (R\pi_*\mcalO_{\mcalY}) \simeq (R\pi_*\mcalO_{\mcalY}) \otimes^L_{\mcalO_{\mcalX}} (Rq_*\mcalO_{B_{\mcalX}G})
  \end{equation*}
  in \(\Mod_{R\pi_*\mcalO_{\mcalY}}(\mcalD_{\qcoh}(\mcalX))\), where \(R\pi_*\mcalO_{\mcalY}\) acts through the first tensor factor.
\end{proposition}

\begin{proof}
  Consider the Cartesian diagram
  \begin{equation*}
    \begin{tikzcd}
      \mcalY\times_{\mcalX}\mcalY \arrow[r, "\pr_2"] \arrow[d, "\pr_1"'] & \mcalY \arrow[d, "\pi"] \\
      \mcalY \arrow[r, "\pi"'] & \mcalX.
    \end{tikzcd}
  \end{equation*}
  Since \(\pr_1\) is the base change of the \(G\)-banded gerbe \(\pi\), it is a gerbe banded by \(\pi^*G\). The relative diagonal
  \begin{equation*}
    \Delta_{\pi}\colon\mcalY\to\mcalY\times_{\mcalX}\mcalY
  \end{equation*}
  is a section of \(\pr_1\). Hence this gerbe is neutral, and the given banding together with \(\Delta_{\pi}\) induces an equivalence
  \begin{equation*}
    \mcalY\times_{\mcalX}\mcalY\simeq B_{\mcalY}(\pi^*G)
  \end{equation*}
  over \(\mcalY\).

  Write
  \begin{equation*}
    q_{\mcalY}\colon B_{\mcalY}(\pi^*G)\to\mcalY
  \end{equation*}
  for the structure morphism. Formation of the classifying stack commutes with base change, so there is a Cartesian diagram
  \begin{equation*}
    \begin{tikzcd}
      B_{\mcalY}(\pi^*G) \arrow[r] \arrow[d, "q_{\mcalY}"'] & B_{\mcalX}G \arrow[d, "q"] \\
      \mcalY \arrow[r, "\pi"'] & \mcalX.
    \end{tikzcd}
  \end{equation*}
  Applying flat base change from \Cref{FlatBaseChange} to this diagram gives
  \begin{equation*}
    Rq_{\mcalY,*}\mcalO_{B_{\mcalY}(\pi^*G)}\simeq L\pi^*Rq_*\mcalO_{B_{\mcalX}G}
  \end{equation*}
  in \(\mcalD_{\qcoh}(\mcalY)\).
  Under the equivalence \(\mcalY\times_{\mcalX}\mcalY\simeq B_{\mcalY}(\pi^*G)\), the morphism \(\pr_1\) corresponds to \(q_{\mcalY}\). Consequently,
  \begin{equation*}
    R\pr_{1,*}\mcalO_{\mcalY\times_{\mcalX}\mcalY}\simeq L\pi^*Rq_*\mcalO_{B_{\mcalX}G}
  \end{equation*}
  holds in \(\mcalD_{\qcoh}(\mcalY)\).

  On the other hand, applying flat base change from \Cref{FlatBaseChange} to the first Cartesian diagram gives
  \begin{equation*}
    R\pr_{1,*}\mcalO_{\mcalY\times_{\mcalX}\mcalY}\simeq L\pi^*R\pi_*\mcalO_{\mcalY}
  \end{equation*}
  in \(\mcalD_{\qcoh}(\mcalY)\).
  Combining the preceding two isomorphisms yields
  \begin{equation*}
    L\pi^*R\pi_*\mcalO_{\mcalY}\simeq L\pi^*Rq_*\mcalO_{B_{\mcalX}G}
  \end{equation*}
  in \(\mcalD_{\qcoh}(\mcalY)\).

  Applying \(R\pi_*\) and using the projection formula for perfect objects \(R\pi_*\mcalO_{\mcalY}\) and \(Rq_*\mcalO_{B_{\mcalX}G}\) as proved in \citeSta{0944}, we obtain
  \begin{align*}
    R\pi_*\mcalO_{\mcalY}\otimes^L_{\mcalO_{\mcalX}}R\pi_*\mcalO_{\mcalY} &\simeq R\pi_*L\pi^*R\pi_*\mcalO_{\mcalY} \simeq R\pi_*L\pi^*Rq_*\mcalO_{B_{\mcalX}G} \simeq R\pi_*\mcalO_{\mcalY}\otimes^L_{\mcalO_{\mcalX}}Rq_*\mcalO_{B_{\mcalX}G}
  \end{align*}
  in \(\mcalD_{\qcoh}(\mcalX)\).
  The base-change isomorphisms and the neutralization by \(\Delta_{\pi}\) are compatible with the commutative algebra structures. Hence the resulting isomorphism is \(R\pi_*\mcalO_{\mcalY}\)-linear for the actions through the first tensor factors.
\end{proof}

We next turn to the following general result on derived algebra.

\begin{lemma}
  \label{automatic-postnikov-compatibility}
  Let \(X\) be a stack and let
  \begin{equation*}
    A\in\CAlg\parenlr{\mcalD_{\qcoh}(X)}
  \end{equation*}
  be coconnective. Let \(K,L\in\mcalD_{\qcoh}(X)\), and suppose that \(\mcalH^i(L)\) is flat for every \(i\in\setZ\). Let
  \begin{equation*}
    \alpha\colon A\otimes^L K\longrightarrow A\otimes^L L
  \end{equation*}
  be a morphism in \(\Mod_A\parenlr{\mcalD_{\qcoh}(X)}\). Then \(\alpha\) canonically induces a morphism of filtered objects
  \begin{equation*}
    \alpha_{\leq\bullet}\colon\bracelr{A\otimes^L\tau_{\leq m}K}_{m\in\setZ}\longrightarrow\bracelr{A\otimes^L\tau_{\leq m}L}_{m\in\setZ}
  \end{equation*}
  in
  \begin{equation*}
    \Fun\parenlr{(\setZ,\leq),\Mod_A\parenlr{\mcalD_{\qcoh}(X)}}.
  \end{equation*}
  If, in addition, \(\mcalH^i(K)\) is flat for every \(i\in\setZ\) and \(\alpha\) is an isomorphism, then \(\alpha_{\leq\bullet}\) is an isomorphism of filtered objects.
\end{lemma}

\begin{proof}
  Set
  \begin{equation*}
    \mcalM_A\defeq\Mod_A\parenlr{\mcalD_{\qcoh}(X)}.
  \end{equation*}
  Fix \(m\in\setZ\). By the hyper-Tor spectral sequence
  \begin{equation*}
    E_2^{-s,t}=\bigoplus_{i+j=t}\operatorname{Tor}_s\parenlr{\mcalH^i(A),\mcalH^j\parenlr{\tau_{\geq m+1}L}}\Longrightarrow\mcalH^{t-s}\parenlr{A\otimes^L\tau_{\geq m+1}L},
  \end{equation*}
  we have
  \begin{equation*}
    A\otimes^L\tau_{\geq m+1}L\in\mcalD_{\qcoh}(X)^{\geq m+1}.
  \end{equation*}
  Indeed, \(A\) is coconnective, and the cohomology sheaves of \(\tau_{\geq m+1}L\) are flat and vanish in degrees less than \(m+1\).
  Therefore, for every pair of integers \(r \leq s\), we have
  \begin{equation} \label{TruncationEstimates}
    \tau_{\leq r}K\in\mcalD_{\qcoh}(X)^{\leq r}\qquad\text{and}\qquad A\otimes^L\tau_{\geq s+1}L\in\mcalD_{\qcoh}(X)^{\geq s+1}\subseteq\mcalD_{\qcoh}(X)^{\geq r+1}.
  \end{equation}
  
  Regard \(\setZ\) as a category by its usual order, and define objects
  \begin{equation*}
    P_K,P_L,Q_L\in\Fun\parenlr{(\setZ,\leq),\mcalM_A}
  \end{equation*}
  by
  \begin{equation*}
    P_K(m)\defeq A\otimes^L\tau_{\leq m}K,\qquad P_L(m)\defeq A\otimes^L\tau_{\leq m}L,\qquad Q_L(m)\defeq A\otimes^L\tau_{\geq m+1}L.
  \end{equation*}
  For \(M\in\mcalM_A\), write \(\underline{M}\) for the constant \(\setZ\)-indexed diagram with value \(M\). The truncation triangles give a pointwise cofiber sequence
  \begin{equation} \label{TruncationCofiberSequence}
    P_L\longrightarrow\underline{A\otimes^L L}\longrightarrow Q_L
  \end{equation}
  in \(\Fun\parenlr{(\setZ,\leq),\mcalM_A}\). Moreover, \(\alpha\) determines a natural morphism
  \begin{equation*}
    \overline{\alpha}\colon P_K\longrightarrow\underline{A\otimes^L K}\xrightarrow{\underline{\alpha}}\underline{A\otimes^L L}.
  \end{equation*}

  By the end formula of mapping spaces of functor categories (\cite{gepner2017Lax}*{Proposition~5.1}), for \(F,G\in\Fun\parenlr{(\setZ,\leq),\mcalM_A}\), there is a natural equivalence
  \begin{equation*}
    \Map_{\Fun((\setZ,\leq),\mcalM_A)}(F,G)\simeq\lim_{u \in \operatorname{Tw}(\setZ)}\Map_{\mcalM_A}\parenlr{F(r),G(s)}.
  \end{equation*}
  Here an object \(u\colon r\to s\) of \(\operatorname{Tw}(\setZ)\) is equivalently a pair of integers \(r\leq s\). For every such pair, the Hom-tensor adjunction gives
  \begin{equation*}
    \Map_{\mcalM_A}\parenlr{P_K(r),Q_L(s)} \simeq\Map_{\mcalD_{\qcoh}(X)}\parenlr{\tau_{\leq r}K,A\otimes^L\tau_{\geq s+1}L} \simeq *
  \end{equation*}
  by \eqref{TruncationEstimates}.
  It follows from the displayed end formula that
  \begin{equation*}
    \Map_{\Fun((\setZ,\leq),\mcalM_A)}(P_K,Q_L)\simeq *.
  \end{equation*}

  Applying \(\Map_{\Fun((\setZ,\leq),\mcalM_A)}(P_K,-)\) to the pointwise cofiber sequence \eqref{TruncationCofiberSequence} above gives a fiber sequence
  \begin{equation*}
    \Map_{\Fun((\setZ,\leq),\mcalM_A)}(P_K,P_L)\longrightarrow\Map_{\Fun((\setZ,\leq),\mcalM_A)}\parenlr{P_K,\underline{A\otimes^L L}}\longrightarrow\Map_{\Fun((\setZ,\leq),\mcalM_A)}(P_K,Q_L).
  \end{equation*}
  Since the last mapping space is contractible, the first morphism is an equivalence. Therefore \(\overline{\alpha}\) admits a contractible space of lifts through \(P_L\). We denote the resulting canonical lift by
  \begin{equation*}
    \alpha_{\leq\bullet}\colon P_K\longrightarrow P_L.
  \end{equation*}
  Its component at \(m\) agrees with the factorization of
  \begin{equation*}
    \overline{\alpha}_{\leq m} \colon A\otimes^L\tau_{\leq m}K \to A \otimes^L K \xrightarrow{\alpha} A \otimes^L L
  \end{equation*}
  through \(A\otimes^L\tau_{\leq m}L\) by \eqref{TruncationEstimates}.

  The same contractibility also shows that the construction preserves identity morphisms and composition. Indeed, whenever the constructions are defined, both \((\beta\circ\alpha)_{\leq\bullet}\) and \(\beta_{\leq\bullet}\circ\alpha_{\leq\bullet}\) are lifts of the same morphism to the corresponding constant diagram, and hence are canonically homotopic.

  Finally, suppose that \(\mcalH^i(K)\) is flat for every \(i\in\setZ\) and that \(\alpha\) is an isomorphism. Applying the construction to \(\alpha^{-1}\) gives a morphism
  \begin{equation*}
    (\alpha^{-1})_{\leq\bullet}\colon P_L\longrightarrow P_K.
  \end{equation*}
  Compatibility with identities and composition gives
  \begin{equation*}
    (\alpha^{-1})_{\leq\bullet}\circ\alpha_{\leq\bullet}\simeq\id_{P_K}\qquad\text{and}\qquad\alpha_{\leq\bullet}\circ(\alpha^{-1})_{\leq\bullet}\simeq\id_{P_L}.
  \end{equation*}
  Hence \(\alpha_{\leq\bullet}\) is an isomorphism in \(\Fun\parenlr{(\setZ,\leq),\mcalM_A}\).
\end{proof}

\subsection{Quasi-\texorpdfstring{\(F\)}{F}-splitting for projective bundles}

\begin{proposition}\label{prop:qFs-torus-torsor}
Let \(k\) be a perfect field of characteristic \(p>0\), and let
\(X\) be a \(k\)-scheme.
Let $T \defeq \setG_m^r$ and let $q\colon Q \to  X$ be a \(T\)-torsor.
If \(X\) is \(n\)-quasi-\(F\)-split, then \(Q\) is \(n\)-quasi-\(F\)-split.
\end{proposition}

\begin{proof}
Let $M \defeq X^*(T)\simeq\setZ^r$ be the character lattice of \(T\).
The coordinate ring of \(T\) is
\[
k[M]
=
\bigoplus_{m\in M}k\chi^m.
\]
We equip \(T\) with its standard Frobenius splitting $\tau\colon F_*k[M] \to  k[M]$ defined by
\[
\tau(\chi^m)
=
\begin{cases}
\chi^{m/p} & \text{if }m\in pM,\\
0 & \text{if }m\notin pM.
\end{cases}
\]

Fix a splitting
\[
\sigma_X\colon
F_{*}\overline W_n\mcalO_X
 \to 
\mcalO_X
\]
of $\Phi_X \colon \cO_X \to F_{*}\overline W_n\mcalO_X$.
Choose an affine open covering $X=\bigcup_\alpha U_\alpha$ which trivializes \(Q\).
Thus we have isomorphisms
\[
Q|_{U_\alpha}\simeq U_\alpha\times T.
\]

The restriction of
\(\sigma_X\) to \(U_\alpha\), together with the Frobenius splitting
\(\tau\) of \(T\), induces a splitting
\[
\Sigma_\alpha\colon
F_*\overline W_n
\mcalO_{U_\alpha\times T}
 \to 
\mcalO_{U_\alpha\times T}
\]
constructed in \cite{YobukoHodgeWitt}.
We verify that these local splittings are compatible with the transition functions of the torsor.

We recall the explicit form of the product splitting.
Put
\[
A \defeq k[M],
\qquad
B_\alpha \defeq \mcalO_X(U_{\alpha}).
\]
The Cartier tensor-product presentation identifies
\(\overline W_n(A\otimes_k B_\alpha)\) with a quotient of
\[
\bigl(W(A)\otimes_{W(k)}W(B_\alpha)\bigr)[V]/V^n
\]
by the relations
\[
(a\otimes Vb)V^j=(Fa\otimes b)V^{j+1},
\qquad
(Va\otimes b)V^j=(a\otimes Fb)V^{j+1}.
\]
In particular, it is additively generated by the symbols
\[
(a\otimes b)V^j,
\qquad
a\in W(A),\quad b\in W(B_\alpha),\quad 0\leq j<n.
\]
Then $\Sigma_{\alpha}$ is given by
\[
\Sigma_\alpha\bigl((a\otimes b)V^j\bigr)
=
\tau^{j+1}(Ra)\,
\sigma_X(V^j b),
\qquad
0\leq j<n,
\]
where $R\colon W(A) \to  A$ is the restriction morphism.

On an overlap
\[
U_{\alpha\beta} \defeq U_\alpha\cap U_\beta,
\]
write
\[
\theta_\alpha\colon
Q|_{U_\alpha}\xrightarrow{\sim}U_\alpha\times T
\]
for the chosen trivializations. 
For \(m\in M\), define $u_{\alpha\beta,m} \in \Gamma(U_{\alpha\beta},\mcalO_{U_{\alpha\beta}}^\times)$ by
\[
(\theta_\beta\circ\theta_\alpha^{-1})^{\sharp}(\chi^m)=u_{\alpha\beta,m} ^{-1}\chi^m.
\]

It suffices to show the commutativity of
\[
\begin{tikzcd}
\left(
W(k[M])\otimes_{W(k)}
W(\cO_X(U_{\alpha\beta}))
\right)[V]/V^n
\arrow[r, "\widetilde h_{\alpha\beta}^{\sharp}"]
\arrow[d, "\Sigma_\beta"']
&
\left(
W(k[M])\otimes_{W(k)}
W(\cO_X(U_{\alpha\beta}))
\right)[V]/V^n
\arrow[d, "\Sigma_\alpha"]
\\
k[M]\otimes_k \cO_X(U_{\alpha\beta})
\arrow[r, "h_{\alpha\beta}^{\sharp}"']
&
k[M]\otimes_k \cO_X(U_{\alpha\beta}),
\end{tikzcd}
\]
where $h_{\alpha\beta}
 \defeq 
\theta_\beta\circ\theta_\alpha^{-1}$ and $\widetilde h_{\alpha\beta}^{\sharp}$ is a natural lift of $h_{\alpha\beta}^{\sharp}$.

It is therefore enough to verify that
\begin{equation}\label{eq:Sigma}
h_{\alpha\beta}^{\sharp}\left(\Sigma_\beta
\left(
\bigl([\chi^m]\otimes [u_{\alpha\beta,m}]b\bigr)V^j
\right)\right)
=
\Sigma_\alpha
\left(
\bigl([\chi^m]\otimes b\bigr)V^j
\right)    
\end{equation}
for every $m \in M$.

If $m\notin p^{j+1}M$, then both sides in \eqref{eq:Sigma} are zero.

Suppose that $m=p^{j+1}\ell$.
Then $u_{\alpha\beta,m}=u_{\alpha\beta,\ell}^{p^{j+1}}$.
Thus, the left hand side in \eqref{eq:Sigma} is
\begin{align*}
h_{\alpha\beta}^{\sharp}\left(\Sigma_\beta
\left(
\bigl([\chi^m]\otimes [u_{\alpha\beta,m}]b\bigr)V^j
\right)\right)
&=h_{\alpha\beta}^{\sharp}\left(\chi^\ell \sigma_X(V^j([u_{\alpha\beta,m}]b))\right) \\
&=h_{\alpha\beta}^{\sharp}\left(\chi^\ell u_{\alpha\beta,\ell} \sigma_X(V^j(b))\right) \\
&=\chi^\ell \sigma_X(V^jb) \\
&=\Sigma_\alpha
\left(
\bigl([\chi^m]\otimes b\bigr)V^j
\right),
\end{align*}
as desired.

Hence the local splittings \(\Sigma_\alpha\) agree on all overlaps and
glue to an \(n\)-quasi-\(F\)-splitting
\[
\Sigma_Q\colon
F_{Q*}\overline W_n\mcalO_Q
 \to 
\mcalO_Q.
\]
Thus \(Q\) is \(n\)-quasi-\(F\)-split.
\end{proof}

\begin{corollary}\label{cor:qFs-split-projective-bundle}
Let \(k\) be a perfect field of characteristic \(p>0\), and let
\(X\) be a \(k\)-scheme.
Let $L_1,\ldots,L_r$ be line bundles on $X$.
If \(X\) is \(n\)-quasi-\(F\)-split, then
\[
\setP_X\left(\bigoplus_{i=1}^r L_i\right)
\]
is \(n\)-quasi-\(F\)-split.
\end{corollary}

\begin{proof}
Let
\[
q\colon
Q \defeq 
L_1^\times\times_X\cdots\times_X L_r^\times
 \to  X
\]
be the \(T \defeq \setG_m^r\)-torsor associated with \((L_1,\ldots,L_r)\).
By \Cref{prop:qFs-torus-torsor}, the scheme \(Q\) is \(n\)-quasi-\(F\)-split.

The torus \(T\) acts diagonally on \(\setP^{r-1}\) by
\[
(\lambda_1,\ldots,\lambda_r)
\cdot
[x_1:\cdots:x_r]
=
[\lambda_1x_1:\cdots:\lambda_rx_r].
\]
There is a natural isomorphism
\[
\setP_X\left(\bigoplus_{i=1}^r L_i\right)
\simeq
Q\times^T\setP^{r-1}.
\]

The projective space \(\setP^{r-1}\) is Frobenius split.
Hence the product theorem for quasi-\(F\)-splittings implies that $Q\times\setP^{r-1}$ is \(n\)-quasi-\(F\)-split.
Thus, its $T$-quotient $\setP_X\left(\bigoplus_{i=1}^r L_i\right)$ is $n$-quasi-$F$-split.
\end{proof}

\section{Hodge--Tate splitting in positive characteristic} \label{SectionHTSplitDef}

In this section, we introduce relative Hodge--Tate splitting for stacks and prove a decomposition theorem for the de Rham complex in the smooth representable case.

\subsection{Relative Hodge--Tate splitting for stacks}

Let us recall the notion of the de Rham stack of derived \(\setF_p\)-schemes from \cite{bhatt2022Prismatic}*{\S 2.5}, \cite{petrov2025Decomposition}*{\S 3} and \cite{barz2025Logarithmic}*{\S 2.2}.

\begin{definition} \label{DefdeRhamStack}
Let \(\iota \colon X \to S\) be a morphism of $\F_p$-stacks.
Let
\[
X^{(1)} \defeq X\times_{S,F_S}S
\]
be the relative Frobenius twist, and let
\[
F_{X/S}\colon X \to  X^{(1)}
\]
be the relative Frobenius.
  \begin{itemize}
      \item   The \emph{(absolute) de Rham stack} of \(X\) is defined as a functor
  \begin{equation*}
    X^{\dR} \colon \CAlg^{an}_{\setF_p} \to \Ani; \  R \mapsto X(W(R)/^L p),
  \end{equation*}
  where \(W(R)\) is the ring of Witt vectors of \(R\), \(-/^L p\) is the derived quotient by \(p\), and \(X(W(R)/^L p)\) is the mapping space \(\Map_{\dSch_{\setF_p}}(\Spec(W(R)/^L p), X)\) in the \(\infty\)-category of derived schemes over \(\setF_p\).
  Since this is isomorphic to the relative Hodge--Tate stack \((X/\setZ_p)^{\HT}\) of \(X\) relative to the crystalline prism \((\setZ_p, (p))\) introduced in \cite{bhatt2022Prismatization}*{Construction 7.1}, this is a stack on \(\CAlg^{an}_{\setF_p}\) (\cite{bhatt2022Prismatization}*{Lemma 7.3}).

  This is equipped with natural morphisms
  \begin{equation*}
    \pi_{X} \colon X \to X^{\dR} \quad \text{and} \quad \nu_X \colon X^{\dR} \to X
  \end{equation*}
  of \(\setF_p\)-stacks: For each \(R \in \CAlg^{an}_{\setF_p}\), each morphism is given by precomposition of morphisms
  \begin{equation*}
    R \xrightarrow{\Phi_R} W(R)/^L p \quad \text{and} \quad W(R)/^L p \xrightarrow{\res} R,
  \end{equation*}
  where \(\Phi_R\) is the animation of the morphism \(a \mapsto [a^p] \bmod p\) and \(\res\) is the restriction morphism.
  Note that the composition \(\nu_X \circ \pi_X\) is the absolute Frobenius morphism \(F_X \colon X \to X\) of \(X\).
\item The \emph{de Rham stack} of $X$ over $S$ is defined by
\[
(X/S)^{\dR} \defeq X^{\dR} \times_{S^{\dR},\pi_S} S.
\]
We define natural morphisms
  \begin{equation*}
    \pi_{X/S} \colon X \to (X/S)^{\dR} \quad \text{and} \quad \nu_{X/S} \colon (X/S)^{\dR} \to X^{(1)}
  \end{equation*}
  of \(\setF_p\)-stacks as follows:
The morphism $\pi_{X/S}$ is defined by
\[
\begin{tikzcd}
X \arrow[r,"\pi_{X/S}"'] \arrow[rd,bend right,"\iota"] \arrow[rr,bend left,"\pi_X"] & X^{\dR} \times_{S^{\dR}} S \arrow[r,"\pr_1"] \arrow[d,"\pr_2"] & X^{\dR} \arrow[d,"\iota^{\dR}"] \\
& S \arrow[r,"\pi_S"] & S^{\dR}.
\end{tikzcd}
\]
The morphism $\nu_{X/S} \colon (X/S)^{\dR} \to X^{(1)}$ is induced by the following commutative diagram
\[
\begin{tikzcd}
X^{\dR} \times_{S^{\dR}} S \arrow[r,"\pr_1"] \arrow[d,"\pr_2"] & X^{\dR} \arrow[d,"\iota^{\dR}"] \arrow[r,"\nu_X"] & X \arrow[d,"\iota"] \\
S \arrow[r,"\pi_S"] \arrow[rr,bend right,"F"] & S^{\dR} \arrow[r,"\nu_S"] & S.
\end{tikzcd}
\]
\item We assume that $\iota$ is smooth representable.
Then the cotangent complex \(L_{X/S}\) is a locally free \(\cO_X\)-module concentrated in degree \(0\).  
We set
\[
\Omega^1_{X/S} \defeq L_{X/S},
\qquad
\Omega^i_{X/S} \defeq 
\bigwedge_{\cO_X}^i\Omega^1_{X/S}.
\]
The universal \(S\)-derivation
\[
d_{X/S}\colon \cO_X \to  \Omega^1_{X/S}
\]
extends uniquely to the relative de Rham complex
\[
\Omega^\bullet_{X/S}
=
\left[
\cO_X
\xrightarrow{d_{X/S}}
\Omega^1_{X/S}
\xrightarrow{d_{X/S}}
\Omega^2_{X/S}
 \to \cdots
\right].
\]

We define a morphism of complexes of $\cO_{X^{(1)}}$-modules
\[
F^{\dR}_{X/S} \colon \cO_{X^{(1)}} \to F_{X/S,*}\Omega^{\bullet}_{X/S}
\]
as follows:
We have
\[
d_{X/S}\circ F_{X/S}^{\sharp}=0
\]
by checking smooth locally on \(X\).  
Therefore, the morphism
\[
F_{X/S}^{\sharp}\colon
\cO_{X^{(1)}} \to  F_{X/S,*}\cO_X
\]
defines a morphism of complexes
\[
F_{X/S}^{\dR}\colon
\cO_{X^{(1)}}
 \to 
F_{X/S,*}\Omega^\bullet_{X/S},
\]
whose degree \(0\) component is \(F_{X/S}^{\sharp}\) and whose
positive-degree components are zero.
\end{itemize}
\end{definition}

\begin{definition} \label{HTSplitMorphism}
Let \(\iota \colon X \to S\) be a morphism of $\F_p$-stacks.
We say that $X$ is \emph{Hodge--Tate split} (\emph{HT-split}, for short) over $S$ if the homomorphism
\[
\nu_{X/S}^{\sharp} \colon \cO_{X^{(1)}} \to R\nu_{X/S,*}\cO_{(X/S)^{\dR}} \eqdef \OdR{X/S}
\]
is ind-split in $\mcalD_{\qcoh}(X^{(1)})$, where \(X^{(1)}\) is the Frobenius twist of \(X\) over \(S\).

We say that an \(\setF_p\)-stack \(X\) is \emph{HT-split} if $X$ is HT-split over $\Spec(\F_p)$.
\end{definition}

\begin{remark}\label{rmk:perfect-base}
    Let \(X \to S\) be a morphism of \(\setF_p\)-stacks.
    If \(S\) is a perfect \(\setF_p\)-scheme, then \(X\) is HT-split over \(S\) if and only if \(X\) is HT-split. This follows from the isomorphism \((X/S)^{\dR} \xrightarrow{\cong} X^{\dR}\) proved in \cite{petrov2025Decomposition}*{Remark 3.5 (2)}.
\end{remark}

\begin{proposition}\label{smooth-HT-split-equiv}
Let \(\iota\colon X\to S\) be a smooth representable morphism of fpqc algebraic \(\F_p\)-stacks.
Then \(X\) is HT-split over \(S\) if and only if \(F^{\dR}_{X/S}\) is ind-split in \(\mcalD_{\qcoh}(X^{(1)})\).
\end{proposition}

\begin{proof}
By the degree-zero Cartier isomorphism
\begin{equation} \label{DegreeZeroCartier}
  \mcalH^0\left(F_{X/S,*}\Omega^\bullet_{X/S}\right)
=
\ker\left(
F_{X/S,*}\cO_X
\xrightarrow{d}
F_{X/S,*}\Omega^1_{X/S}
\right)
\simeq
\cO_{X^{(1)}},
\end{equation}
the morphism
\[
F^{\dR}_{X/S}\colon
\cO_{X^{(1)}}
 \to 
F_{X/S,*}\Omega^\bullet_{X/S}
\]
induces an isomorphism on zeroth cohomology.

By \cite{bhatt2022Prismatic}*{Corollary~2.7.2~(3)} and \cite{terentiuk2026OgusVologodsky}*{Lemma~3.28 and Remark~3.29}, there exists an isomorphism
\[
\alpha\colon
R\nu_{X/S,*}\cO_{(X/S)^{\dR}}
\xrightarrow{\sim}
F_{X/S,*}\Omega^\bullet_{X/S}
\]
in \(\mcalD_{\qcoh}(X^{(1)})\).
Combining this with \eqref{DegreeZeroCartier}, we obtain a morphism
\[
m\colon
\cO_{X^{(1)}}
\xrightarrow{\mcalH^0(\nu_{X/S}^{\sharp})}
\mcalH^0\left(
R\nu_{X/S,*}\cO_{(X/S)^{\dR}}
\right) \xrightarrow{\mcalH^0(\alpha)} \cO_{X^{(1)}}
\]
which is an isomorphism since $\mcalH^0(\nu_{X/S}^{\sharp})$ is \(\mcalO_{X^{(1)}}\)-linear and \(1\) goes to \(1\) under the composition.
This shows that \(\mcalH^0(\alpha\circ \nu_{X/S}^{\sharp})\) is an isomorphism and that
\[
\mcalH^0(\alpha\circ \nu_{X/S}^{\sharp})\colon
\cO_{X^{(1)}}
\xrightarrow{\sim}
\mcalH^0\left(
F_{X/S,*}\Omega^\bullet_{X/S}
\right)
\]
is an isomorphism.

Therefore there is a unique automorphism
\[
g\colon
\mcalH^0\left(
F_{X/S,*}\Omega^\bullet_{X/S}
\right)
\xrightarrow{\sim}
\mcalH^0\left(
F_{X/S,*}\Omega^\bullet_{X/S}
\right)
\]
such that
\[
g\circ\mcalH^0(\alpha\circ \nu_{X/S}^{\sharp}) \simeq \mcalH^0(F^{\dR}_{X/S}).
\]
Under the Cartier identification \eqref{DegreeZeroCartier} above, the automorphism \(g\) is multiplication by a section $c$ of \(\cO_{X^{(1)}}^\times\).  
Then we define an automorphism
\[
\times c\colon
F_{X/S,*}\Omega^\bullet_{X/S}
\xrightarrow{\sim}
F_{X/S,*}\Omega^\bullet_{X/S}
\]
in \(\mcalD_{\qcoh}(X^{(1)})\).
Replacing \(\alpha\) by \(\times c\circ\alpha\), we may assume that
\[
\mcalH^0(\alpha\circ \nu_{X/S}^{\sharp}) \simeq \mcalH^0(F^{\dR}_{X/S}).
\]

We claim that then
\[
\alpha\circ \nu_{X/S}^{\sharp} \simeq F^{\dR}_{X/S}
\]
holds in \(\mcalD_{\qcoh}(X^{(1)})\): Since \(X \to S\) is smooth representable, the morphism \(F^{\dR}_{X/S}\) is concentrated in nonnegative cohomological degrees.
By using the \(t\)-structure, there is a natural identification
\begin{equation*}
  \Hom_{\mcalD_{\qcoh}(X^{(1)})}(\mcalO_{X^{(1)}}, F_{X/S,*}\Omega^\bullet_{X/S}) \simeq \Hom_{\mcalD_{\qcoh}(X^{(1)})}(\mcalO_{X^{(1)}}, \mcalH^0(F_{X/S,*}\Omega^\bullet_{X/S})).
\end{equation*}
Hence two morphisms from \(\mcalO_{X^{(1)}}\) to
\(F_{X/S,*}\Omega^\bullet_{X/S}\) are equal if and only if they induce
the same element of zeroth cohomology.  The claim follows.

We have therefore chosen the comparison isomorphism \(\alpha\) so that
\[
\alpha\circ \nu_{X/S}^{\sharp} \simeq F^{\dR}_{X/S}.
\]
Since \(\alpha\) is an isomorphism, \(\nu_{X/S}^{\sharp}\) is ind-split if and only if
\(F^{\dR}_{X/S}\) is ind-split.  Thus \(X\) is HT-split over \(S\) if and
only if \(F^{\dR}_{X/S}\) is ind-split in \(\mcalD_{\qcoh}(X^{(1)})\).
\end{proof}

\subsection{Decomposition of the de Rham complex}

\begin{theorem}\label{decomposition-after-pullback}
Let \(\iota\colon X\to S\) be a smooth representable morphism of fpqc algebraic \(\F_p\)-stacks.
Then there is an isomorphism
\begin{equation*}
  \OdR{X/S} \otimes^L_{\cO_{X^{(1)}}} F_{X/S,*}\Omega^\bullet_{X/S} \cong \OdR{X/S} \otimes^L_{\cO_{X^{(1)}}} \left( \bigoplus_{i \geq 0} \Omega^i_{X^{(1)}/S}[-i] \right)
\end{equation*}
in $\Mod_{\OdR{X/S}} \bigl(\mcalD_{\qcoh}(X^{(1)})\bigr)$. 
\end{theorem}

\begin{proof}
  By \cite{barz2025Logarithmic}*{Theorem~2.46} and \cite{bhatt2022Prismatic}*{Proposition 2.7.1}, the morphism
\begin{equation*}
  \nu_{X/S} \colon (X/S)^{\dR}  \to  X^{(1)}
\end{equation*}
is a gerbe banded by \(T^{\sharp}_{X^{(1)}/S}\).

As mentioned in the proof of \Cref{smooth-HT-split-equiv}, we have an isomorphism
\begin{equation*}
  \OdR{X/S} \cong F_{X/S,*}\Omega^\bullet_{X/S}
\end{equation*}
in \(\mcalD_{\qcoh}(X^{(1)})\).
Using \cite{bhatt2022Prismatization}*{Lemma~7.8} to compute the pushforward of the structure sheaf of the classifying stack, we obtain
\begin{equation*}
  Rq_*\mcalO_{BT^{\sharp}_{X^{(1)}/S}} \cong \parenlr{\bigoplus_{i\geq 0}\Omega^i_{X^{(1)}/S}[-i]}
\end{equation*}
in \(\mcalD_{\qcoh}(X^{(1)})\).
Since \(\iota \colon X \to S\) is smooth representable, these objects are perfect complexes of \(\cO_{X^{(1)}}\)-modules.

Then we can apply \Cref{prop:tensor-formula-for-banded-gerbe} to the gerbe \(\nu_{X/S}\) and obtain an isomorphism
\begin{equation*}
  \OdR{X/S} \otimes^L_{\cO_{X^{(1)}}} \OdR{X/S} \cong \OdR{X/S} \otimes^L_{\cO_{X^{(1)}}} Rq_*\mcalO_{BT^{\sharp}_{X^{(1)}/S}}
\end{equation*}
in \(\Mod_{\OdR{X/S}}(\mcalD_{\qcoh}(X^{(1)}))\), where \(q \colon BT^{\sharp}_{X^{(1)}/S} \to X^{(1)}\) is the natural projection.

Combining these isomorphisms, we obtain the desired isomorphism.
\end{proof}

\begin{lemma}\label{refinement-splitting-bounded-complex}
  Let \(X\) be an fpqc algebraic stack and let \(K \in \mcalD_{\qcoh}(X)\).
  Assume that \(\mcalH^i(K)\) is a flat \(\cO_X\)-module for every \(i \in \Z\).
  Let \(E_i\) be a quasi-coherent \(\cO_X\)-module with an isomorphism
  \begin{equation*}
    c_i\colon \mcalH^i(K) \xrightarrow{\simeq} E_i
  \end{equation*}
  for every \(i \in \Z\).

  Let \(A \in \CAlg(\mcalD_{\qcoh}(X))\).
  Suppose that $A$ is coconnective and there is an isomorphism
  \begin{equation*}
    \alpha\colon A\otimes_{\cO_X}^{L}K \xrightarrow{\simeq} A\otimes_{\cO_X}^{L}\parenlr{\bigoplus_{i \in \Z}E_i[-i]}
  \end{equation*}
  in \(\Mod_A\parenlr{\mcalD_{\qcoh}(X)}\).

  Then there exists an isomorphism
  \begin{equation*}
    \varphi(\alpha)\colon A\otimes_{\cO_X}^{L}K \xrightarrow{\simeq} A\otimes_{\cO_X}^{L}\parenlr{\bigoplus_{i \in \Z}E_i[-i]}
  \end{equation*}
  in \(\Mod_A\parenlr{\mcalD_{\qcoh}(X)}\) with the following properties.
  \begin{enumerate}
    \item Equipped with the filtered structures on both sides given by the canonical truncations
  \begin{equation*}
    \bracelr{\tau_{\leq m}K}_{m \in \Z} \qquad\text{and}\qquad \bracelr{\bigoplus_{i \leq m}E_i[-i]}_{m \in \Z},
  \end{equation*}
  the isomorphism \(\varphi(\alpha)\) induces a filtered isomorphism
  \begin{equation*}
    \bracelr{\varphi(\alpha)_{\leq m}\colon A\otimes_{\cO_X}^{L}\tau_{\leq m}K \xrightarrow{\simeq} A\otimes_{\cO_X}^{L}\parenlr{\bigoplus_{i \leq m}E_i[-i]}}_{m \in \Z}
  \end{equation*}
  in \(\Mod_A\parenlr{\mcalD_{\qcoh}(X)}\).

  \item For every \(m \in \Z\), the associated graded part
  \begin{equation*}
    \gr_m\parenlr{\varphi(\alpha)_{\leq\bullet}}\colon A\otimes_{\cO_X}^{L}\mcalH^m(K)[-m] \xrightarrow{\simeq} A\otimes_{\cO_X}^{L}E_m[-m]
  \end{equation*}
  coincides with
  \begin{equation*}
    \parenlr{\id_A\otimes c_m}[-m].
  \end{equation*}
  Equivalently, the induced isomorphism
  \begin{equation*}
    \varphi(\alpha)_m\colon A\otimes_{\cO_X}^{L}\mcalH^m(K) \xrightarrow{\simeq} A\otimes_{\cO_X}^{L}E_m
  \end{equation*}
  coincides with \(\id_A\otimes c_m\).
  \end{enumerate}
\end{lemma}

\begin{proof}
Set
\[
E \defeq \bigoplus_{i\in\Z}E_i[-i].
\]
Since each $E_i$ is isomorphic to $\mcalH^i(K)$, it is flat.
Hence the cohomology sheaves of $E$ are flat.

Applying \Cref{automatic-postnikov-compatibility} to the isomorphism $\alpha$, we obtain an isomorphism of filtered objects
\begin{equation*}
    \{\alpha_{\leq m}\}_{m \in \setZ} \colon \bracelr{A \otimes^L \tau_{\leq m} K}_{m \in \setZ} \xrightarrow{\cong} \bracelr{A \otimes^L \tau_{\leq m} E}_{m \in \setZ} \cong \bracelr{A \otimes^L \parenlr{\bigoplus_{i \leq m} E_i[-i]}}_{m \in \setZ}
\end{equation*}
of \(\Mod_A(\mcalD_{\qcoh}(X))\).

Moreover, the morphisms $\alpha_{\leq m}$ are compatible with the
transition morphisms in $m$. Taking associated graded pieces, we obtain
equivalences
\[
\alpha_m\colon
A\otimes^L\mcalH^m(K)
\xrightarrow{\simeq}
A\otimes^L E_m.
\]

For each $i\in\Z$, define an automorphism
\[
\beta_i
 \defeq 
(\id_A\otimes c_i)\circ\alpha_i^{-1}
\colon
A\otimes^L E_i
\xrightarrow{\simeq}
A\otimes^L E_i.
\]
Since tensor product with $A$ preserves direct sums, these automorphisms
define an automorphism
\[
\beta
 \defeq 
\bigoplus_{i\in\Z}\beta_i[-i]
\colon
A\otimes^L E
\xrightarrow{\simeq}
A\otimes^L E.
\]
We set
\[
\varphi(\alpha) \defeq \beta\circ\alpha.
\]

For every $m\in\Z$, let
\[
\beta_{\leq m}
 \defeq 
\bigoplus_{i\leq m}\beta_i[-i]
\]
and define
\[
\varphi(\alpha)_{\leq m}
 \defeq 
\beta_{\leq m}\circ\alpha_{\leq m}.
\]
Thus
\[
\varphi(\alpha)_{\leq m}\colon
A\otimes^L\tau_{\leq m}K
\xrightarrow{\simeq}
A\otimes^L\parenlr{\bigoplus_{i\leq m}E_i[-i]}.
\]
Since the morphisms $\alpha_{\leq m}$ are compatible with the
transition morphisms and the automorphisms $\beta_{\leq m}$ are
compatible under the natural inclusions, the morphisms
$\varphi(\alpha)_{\leq m}$ are compatible with the transition
morphisms in $m$.
Moreover, they are compatible with the global morphism
$\varphi(\alpha)$.
Hence $\varphi(\alpha)$ gives the required filtered isomorphism.

Finally, the $i$-th associated graded piece is
\[
\begin{aligned}
\gr_i\parenlr{\varphi(\alpha)_{\leq\bullet}}
&=
\parenlr{\beta_i\circ\alpha_i}[-i]\\
&=
\parenlr{
(\id_A\otimes c_i)
\circ\alpha_i^{-1}
\circ\alpha_i
}[-i]\\
&=
(\id_A\otimes c_i)[-i].
\end{aligned}
\]
Equivalently, the induced equivalence
\[
A\otimes^L\mcalH^i(K)
\xrightarrow{\simeq}
A\otimes^L E_i
\]
is $\id_A\otimes c_i$, as desired.
\end{proof}

\begin{theorem}\label{decomposition-positive}
Let \(\iota\colon X\to S\) be a smooth representable morphism of fpqc algebraic \(\F_p\)-stacks of relative dimension \(d\).
Assume that  $F^{\dR}_{X/S}$ splits. 
Then there is an isomorphism
\[
F_{X/S,*}\Omega_{X/S}^{\bullet}
\simeq
\bigoplus_{i=0}^d
\Omega_{X^{(1)}/S}^i[-i]
\]
in $\mcalD_{\qcoh}(X^{(1)})$.
This isomorphism is canonically obtained from a given retraction
\[
(\varphi \colon F_{X/S,*}\Omega^\bullet_{X/S}
\longrightarrow
\mcalO_{X^{(1)}},\, \varphi \circ F^{\dR}_{X/S} \simeq \id_{\mcalO_{X^{(1)}}})
\]
in \(\mcalD_{\qcoh}(X^{(1)})\).

In particular, if \(\iota\) is projective, then, for every
\(\iota\)-ample line bundle \(L\) on \(X\), we have
\[
R^i\iota_*
\bigl(\Omega_{X/S}^j\otimes L^{-1}\bigr)=0
\qquad
(i+j<d) \quad \text{and} \quad R^i\iota_*
\bigl(\Omega_{X/S}^j\otimes L\bigr)=0
\qquad
(i+j>d) 
\]
\end{theorem}

\begin{proof}
Combining \Cref{decomposition-after-pullback} and \Cref{refinement-splitting-bounded-complex}, there exists an isomorphism
\[
\alpha\colon
\OdR{X/S}
\otimes_{\cO_{X^{(1)}}}^{L}
F_{X/S,*}\Omega_{X/S}^{\bullet}
\xrightarrow{\simeq}
\OdR{X/S}
\otimes_{\cO_{X^{(1)}}}^{L}
\left(
\bigoplus_{i=0}^d
\Omega_{X^{(1)}/S}^i[-i]
\right)
\]
such that, for every \(m\), it induces an isomorphism
\[
\alpha_{\leq m}\colon
\OdR{X/S}
\otimes_{\cO_{X^{(1)}}}^{L}
\tau_{\leq m}F_{X/S,*}\Omega_{X/S}^{\bullet}
\xrightarrow{\simeq}
\OdR{X/S}
\otimes_{\cO_{X^{(1)}}}^{L}
\left(
\bigoplus_{i=0}^m
\Omega_{X^{(1)}/S}^i[-i]
\right).
\]
Moreover, the induced isomorphism
\[
\alpha_m\colon
\OdR{X/S}
\otimes_{\cO_{X^{(1)}}}^{L}
\mcalH^m\bigl(F_{X/S,*}\Omega_{X/S}^{\bullet}\bigr)
\xrightarrow{\simeq}
\OdR{X/S}
\otimes_{\cO_{X^{(1)}}}^{L}
\Omega_{X^{(1)}/S}^m
\]
coincides with $\id\otimes C_{X/S}$,
where
\[
C_{X/S}\colon
\mcalH^m\bigl(F_{X/S,*}\Omega_{X/S}^{\bullet}\bigr)
\xrightarrow{\simeq}
\Omega_{X^{(1)}/S}^m
\]
is the Cartier isomorphism.

By using $\alpha_{\leq m}$, we define a filtered morphism
\[
\beta_{\leq m}\colon
\tau_{\leq m}F_{X/S,*}\Omega_{X/S}^{\bullet}
 \to 
\bigoplus_{i=0}^m
\Omega_{X^{(1)}/S}^i[-i]
\]
as the composition
\[
\begin{aligned}
\tau_{\leq m}F_{X/S,*}\Omega_{X/S}^{\bullet}
&\xrightarrow{\nu_{X/S}^{\sharp}\otimes\id}
\OdR{X/S}
\otimes_{\cO_{X^{(1)}}}^{L}
\tau_{\leq m}F_{X/S,*}\Omega_{X/S}^{\bullet} \\
&\xrightarrow{\alpha_{\leq m}}
\OdR{X/S}
\otimes_{\cO_{X^{(1)}}}^{L}
\left(
\bigoplus_{i=0}^m
\Omega_{X^{(1)}/S}^i[-i]
\right) \\
&\xrightarrow{\varphi\otimes\id}
\bigoplus_{i=0}^m
\Omega_{X^{(1)}/S}^i[-i],
\end{aligned}
\]
where $\varphi \colon \OdR{X/S} \to \cO_{X^{(1)}}$ is a retraction of $F^{\dR}_{X/S}$.
Then the graded part of \(\beta_{\leq m}\) is given by the (\(-m\))-shift of the composition
\[
\begin{aligned}
\mcalH^m\bigl(F_{X/S,*}\Omega_{X/S}^{\bullet}\bigr)
&\xrightarrow{\nu^{\sharp}_{X/S}\otimes\id}
\OdR{X/S}
\otimes_{\cO_{X^{(1)}}}^{L}
\mcalH^m\bigl(F_{X/S,*}\Omega_{X/S}^{\bullet}\bigr) \\
&\xrightarrow{\alpha_m}
\OdR{X/S}
\otimes_{\cO_{X^{(1)}}}^{L}
\Omega_{X^{(1)}/S}^m \\
&\xrightarrow{\varphi\otimes\id}
\Omega_{X^{(1)}/S}^m.
\end{aligned}
\]
Since \(\alpha_m=\id\otimes C_{X/S}\) and \(\varphi\circ \nu^{\sharp}_{X/S} \simeq \id_{\cO_{X^{(1)}}}\), this is the Cartier isomorphism \(C_{X/S}\).
Since the filtration is finite and each graded part is an equivalence, we conclude that \(\beta \defeq \beta_{\leq d}\) is an equivalence, and hence
\[
F_{X/S,*}\Omega_{X/S}^{\bullet}
\simeq
\bigoplus_{i=0}^d
\Omega_{X^{(1)}/S}^i[-i]
\]
holds in \(\mcalD_{\qcoh}(X^{(1)})\).

Finally, assume that \(\iota\) is projective, and let \(L\) be an
\(\iota\)-ample line bundle on \(X\).
For every geometric point
\[
\bar{s}\colon \Spec(\bar{k}) \to  S,
\]
the above decomposition is preserved under base change and gives
\[
F_{X_{\bar{s}}/\bar{k},*}
\Omega_{X_{\bar{s}}/\bar{k}}^{\bullet}
\simeq
\bigoplus_{i=0}^d
\Omega_{X_{\bar{s}}^{(1)}/\bar{k}}^i[-i].
\]
Since \(\bar{k}\) is perfect, the argument of
\cite{DI87}*{Lemma~2.9} applies and yields
\[
H^i\left(
X_{\bar{s}},
\Omega_{X_{\bar{s}}/\bar{k}}^j
\otimes L_{\bar{s}}^{-1}
\right)=0
\quad
(i+j<d) \quad \text{and} \quad H^i\left(
X_{\bar{s}},
\Omega_{X_{\bar{s}}/\bar{k}}^j
\otimes L_{\bar{s}}
\right)=0
\quad
(i+j>d).
\]

Let $T=\Spec(A)\to S$ be any morphism from an affine scheme, and let $\iota_T:X_T\to T$ be the resulting smooth morphism. Writing the pullback of $L$ to $X_T$ as $L_T$, for either
\[
  E_{j,T}^{\pm} \defeq \Omega^j_{X_T/T}\otimes L_T^{\pm1},
\]
the complex $R\iota_{T,*}E_{j,T}^{\pm}$ is perfect over $A$ and its formation commutes with arbitrary base change. The preceding fiberwise vanishing therefore implies, by the fiberwise criterion for the Tor-amplitude of a perfect complex,
\[
  R\iota_{T,*}E_{j,T}^{-}\in D^{\ge d-j}(A),
  \qquad
  R\iota_{T,*}E_{j,T}^{+}\in D^{\le d-j}(A).
\]
Hence
\[
  R^i\iota_{T,*}E_{j,T}^{-}=0 \quad (i+j<d),
  \qquad
  R^i\iota_{T,*}E_{j,T}^{+}=0 \quad (i+j>d).
\]
Since this holds for every affine $T\to S$, the asserted vanishing holds in $\mcalD_{\qcoh}(S)$.
\end{proof}

\begin{corollary}\label{first-chara-HT-split}
Let \(\iota\colon X\to S\) be a smooth representable morphism of fpqc algebraic \(\F_p\)-stacks.
Then $F^{\dR}_{X/S}$ splits  if and only if there is an isomorphism
\[
F_{X/S,*}\Omega_{X/S}^{\bullet}
\simeq
\bigoplus_{i=0}^d
\Omega_{X^{(1)}/S}^i[-i]
\]
in $\mcalD_{\qcoh}(X^{(1)})$.
\end{corollary}

\begin{proof}
The ``only if'' part follows from \Cref{decomposition-positive}.
We prove the ``if'' part.
We consider the morphism
\[
\varphi \colon F_{X/S,*}\Omega_{X/S}^{\bullet}
\xrightarrow{\simeq}
\bigoplus_{i=0}^d
\Omega_{X^{(1)}/S}^i[-i] \xrightarrow{\proj} \cO_{X^{(1)}}.
\]
We note that $\mcalH^0(\varphi)$ is an isomorphism by construction.
Then the composition
\[
\alpha \colon \cO_{X^{(1)}} \xrightarrow{F^{\dR}_{X/S}} F_{X/S,*}\Omega_{X/S}^{\bullet}
\xrightarrow{\varphi}\cO_{X^{(1)}}
\]
induces
\[
\cO_{X^{(1)}} \xrightarrow{\mcalH^0(F^{\dR}_{X/S})} \mcalH^0(F_{X/S,*}\Omega_{X/S}^{\bullet})
\xrightarrow{\mcalH^0(\varphi)}\cO_{X^{(1)}}.
\]
Since $\mcalH^0(F^{\dR}_{X/S})$ is an isomorphism as explained in \Cref{smooth-HT-split-equiv}, the morphism $\alpha$ is an isomorphism too.
Thus, the morphism $\alpha^{-1} \circ \varphi$ gives a splitting of $F^{\dR}_{X/S}$.
By \Cref{smooth-HT-split-equiv}, $X$ is HT-split over $S$, as desired.
\end{proof}

\subsection{Logarithmic HT-splitting for smooth schemes}

\begin{notation}\label{notation:log-dR}
Let $k$ be a perfect field of characteristic $p$. 
Let $X$ be a qcqs smooth $k$-scheme of dimension $d \geq 1$, and let $D$ be a simple normal crossings divisor on $X$ over $k$, possibly with $D=0$.
\begin{itemize}
\item Let $F_{X/k} \colon X \to X' \defeq X \times_{\Spec(k),F} \Spec(k)$ be the relative Frobenius morphism.
Set the Frobenius twist $D' \defeq (X' \to X)^*D$ of \(D\) relative to \(k\).

We define a morphism of complexes
\[
F_{(X,D)/k}^{\dR}\colon
\cO_{X'}
 \to 
F_{X/k,*}\Omega^\bullet_{X/k}(\log D),
\]
whose degree \(0\) component is \(F_{X/k}^{\sharp}\) and whose
positive-degree components are zero. 
\item Let $D=D_1+\cdots+D_n$ be the irreducible decomposition of $D$.
By \cite{barz2025Logarithmic}*{Theorem~5.4}, building on work of Olsson \cites{olsson2003Logarithmic,olsson2018Crystalline}, there exists a smooth representable morphism
\[
\iota \colon X \to (\setA^1_k/\setG_{m,k})^n \eqdef S
\]
such that $\Omega^{\bullet}_{X/S} \simeq \Omega_{X/k}^{\bullet}(\log D)$.
Let
\[
X^{(1)} \defeq X\times_{S,F_S}S
\]
be the relative Frobenius twist, and let
\[
F_{X/S}\colon X \to  X^{(1)}
\]
be the relative Frobenius.
We use the same notation in \Cref{DefdeRhamStack}.
Then we have the natural morphism
\[
\pi \colon X^{(1)} \to X'.
\]
We note that $F_{X/k}=\pi \circ F_{X/S}$.
\end{itemize}  
\end{notation}

\begin{definition} \label{Def-log-HT-split}
We use the notation in \Cref{notation:log-dR}.
We say that $(X,D)$ is \emph{log HT-split} if the morphism
\[
F^{\dR}_{(X,D)/k} \colon \cO_{X'} \to F_{X/k,*}\Omega_{X/k}^{\bullet}(\log D)
\]
splits in $\mcalD_{\qcoh}(X')$.
The equivalence of HT-splitting in \Cref{HTSplitMorphism} and this definition will be proved in \Cref{equivalence-log-relative}.
\end{definition}

\begin{remark}\label{relative-vs-absolute}
We can define the homomorphism $F^{\dR}_{(X,D)}$ by
\[
F^{\dR}_{(X,D)} \colon \cO_X \to F_*\Omega_X^{\bullet}(\log D).
\]
Since \(k\) is perfect and thus $X \simeq X'$, the splitting of $F^{\dR}_{(X,D)}$ is equivalent to the log HT-splitting of $(X,D)$. 

Moreover, \((X, 0)\) is log HT-split if and only if \(X\) itself is HT-split in the sense of \Cref{HTSplitMorphism}.
\end{remark}

We give a description of the derived pushforward of the structure sheaf of the log de Rham stack \((X/S)^{\dR}\) in terms of the log de Rham complex \(\Omega_{X/k}^{\bullet}(\log D)\) in \Cref{pullback-log-dR} below.
To show this, we first describe the structure of the Frobenius twist \(X^{(1)}\) as follows.

Here and below, a coarse moduli space is understood in the standard sense recalled, for example, in the paragraph preceding \cite{cadman2007Using}*{Corollary~2.3.7}: it is universal for morphisms to algebraic spaces and induces a bijection on geometric points modulo isomorphism.

\begin{proposition}\label{root-stack-description}
  We use the notation in \Cref{notation:log-dR}.
  Then the canonical morphism
  \begin{equation*}
    \pi\colon X^{(1)} \to  X'
  \end{equation*}
  induced from the structure morphism \(S\to\Spec(k)\) is the coarse moduli morphism of \(X^{(1)}\).

  Moreover, let \(\bar{x}\to X'\) be a geometric point lying on precisely the components \(D_{i_1}',\ldots,D_{i_r}'\), where \(D_i'\) is the pullback of \(D_i\) to \(X'\).
  Then the fiber of \(X^{(1)}\) over \(\bar{x}\) has a unique isomorphism class of geometric points, and the stabilizer group scheme of any representative is canonically isomorphic to
  \begin{equation*}
    \prod_{a=1}^r\mu_p\cong\mu_p^r.
  \end{equation*}
\end{proposition}

\begin{proof}
  By \cite{barz2025Logarithmic}*{Proposition~5.9}, there is a canonical isomorphism
  \begin{equation*}
    X^{(1)}\simeq\sqrt[p]{(X',D_1')}\times_{X'}\cdots\times_{X'}\sqrt[p]{(X',D_n')}\defeq X'\times_{c_{D'},S,([p],\ldots,[p])}S,
  \end{equation*}
  where the right-hand side is the \((p,\ldots,p)\)-multiroot stack of the morphism \(c_{D'}\colon X'\to S\) induced by the simple normal crossings divisor \(D_1',\ldots,D_n'\).
  Under this equivalence, \(\pi\) is identified with the canonical projection from the multiroot stack.

  We first identify its coarse moduli space.
  Let \(U=\Spec(A)\to X'\) be an étale morphism such that \(D_i'|_U=V(f_i)\) for some \(f_i\in A\).
  By the local calculation of root stacks in \cite{cadman2007Using}*{Example~2.4.1}, applied to the \(n\) root constructions, there is an isomorphism
  \begin{equation*}
    X^{(1)}\times_{X'}U\simeq\bracketlr{\Spec(B)/\mu_p^n},
  \end{equation*}
  where
  \begin{equation*}
    B\defeq A[z_1,\ldots,z_n]/(z_1^p-f_1,\ldots,z_n^p-f_n),
  \end{equation*}
  and \(\mu_p^n\) acts by
  \begin{equation*}
    (\lambda_1,\ldots,\lambda_n)\cdot z_i=\lambda_i z_i.
  \end{equation*}
  The evident \((\setZ/p\setZ)^n\)-grading of \(B\), given by \(\deg(z_i)=e_i\), has degree-zero part \(A\), and hence \(B^{\mu_p^n}=A\).  
  Using the product of the quotient presentations of root stacks in \cite{cadman2007Using}*{Proposition~2.3.5 and Corollary~2.3.7}, we can identify the induced morphism \(\bracketlr{\Spec(B)/\mu_p^n}\to U\) with the coarse moduli morphism.  These local identifications are compatible on étale overlaps, so \(\pi\colon X^{(1)}\to X'\) is the coarse moduli morphism.

  We next compute the geometric fiber and its stabilizer.
  Write
  \begin{equation*}
    \bar{x}\colon\Spec(\Omega) \to  X'
  \end{equation*}
  and set \(I\defeq\bracelr{i_1,\ldots,i_r}\).
  Using the preceding quotient presentation, the geometric fiber over \(\bar{x}\) is
  \begin{equation*}
    \bracketlr{\Spec\parenlr{\Omega[z_1,\ldots,z_n]/(z_1^p-f_1(\bar{x}),\ldots,z_n^p-f_n(\bar{x}))}/\mu_{p,\Omega}^n}.
  \end{equation*}
  Since \(\Omega\) is algebraically closed of characteristic \(p\), its Frobenius is bijective, so each equation \(z_i^p=f_i(\bar{x})\) has exactly one solution in \(\Omega\).  Hence the affine scheme in the numerator has a unique \(\Omega\)-point \(y=(\zeta_1,\ldots,\zeta_n)\).
  By the geometric-point condition in the definition of a coarse moduli space, the morphism \(\pi\) induces a bijection
  \begin{equation*}
    X^{(1)}(\Omega)/{\cong}\xrightarrow{\sim}X'(\Omega).
  \end{equation*}
  Hence the inverse image of \(\bar{x}\) consists of a unique isomorphism class.
 
  Finally, \(\zeta_i=0\) if and only if \(i\in I\). The scheme-theoretic stabilizer of \(y\) consists of those \((\lambda_1,\ldots,\lambda_n)\in\mu_{p,\Omega}^n\) satisfying \(\lambda_i\zeta_i=\zeta_i\) for every \(i\).  If \(i\in I\), this condition imposes no restriction on \(\lambda_i\), whereas if \(i\notin I\), then \(\zeta_i\in\Omega^\times\) and the condition forces \(\lambda_i=1\).  Consequently,
  \begin{equation*}
    \operatorname{Stab}(y)\cong\prod_{i\in I}\mu_{p,\Omega}\cong\mu_{p,\Omega}^r,
  \end{equation*}
  which proves the assertion.
\end{proof}

\begin{proposition}
  \label{pullback-log-dR}
  We use \Cref{notation:log-dR}. Then the natural morphism
  \begin{equation*}
    L\pi^*F_{X/k,*}\Omega_{X/k}^{\bullet}(\log D) \to  F_{X/S,*}\Omega_{X/S}^{\bullet}
  \end{equation*}
  is an equivalence in \(\mcalD_{\qcoh}(X^{(1)})\).
\end{proposition}

\begin{proof}
  Set \(M\defeq F_{X/S,*}\Omega_{X/S}^{\bullet}\). Since \(F_{X/k}=\pi\circ F_{X/S}\) and \(\Omega_{X/S}^{\bullet}\simeq\Omega_{X/k}^{\bullet}(\log D)\), we have
  \begin{equation*}
    R\pi_*M\simeq F_{X/k,*}\Omega_{X/k}^{\bullet}(\log D).
  \end{equation*}
  It therefore suffices to prove that the adjunction counit
  \begin{equation*}
    \epsilon_M\colon L\pi^*R\pi_*M \to  M
  \end{equation*}
  is an equivalence.

  By \Cref{root-stack-description}, the morphism \(\pi\colon X^{(1)}\to X'\) is the coarse moduli morphism of the multiroot stack obtained by taking the \(p\)-th roots of the irreducible components of \(D'\), and its geometric stabilizers are products of copies of \(\mu_p\). Since \(\mu_p\) is diagonalizable, all these stabilizers are linearly reductive. Hence \(X^{(1)}\) is tame by \cite{abramovich2008Tame}*{Theorem~3.2}, and the (underived) pushforward \(\pi_* \colon \QCoh(X^{(1)}) \to \QCoh(X')\) is exact by \cite{abramovich2008Tame}*{Definition~3.1}. Moreover, the invariant-ring calculation in the proof of \Cref{root-stack-description} gives
  \begin{equation*}
    \pi_*\cO_{X^{(1)}}\simeq\cO_{X'}.
  \end{equation*}

  By the relative Cartier isomorphism and the natural base-change isomorphisms for relative differential forms, we have
  \begin{equation*}
    \mcalH^i(M)\simeq\Omega_{X^{(1)}/S}^i\simeq F_S^*\Omega_{X/S}^i\simeq F_S^*\Omega_{X/k}^i(\log D)\simeq\pi^*\Omega_{X'/k}^i(\log D'),
  \end{equation*}
  where the third isomorphism follows from \Cref{notation:log-dR}.
  Set \(E_i\defeq\Omega_{X'/k}^i(\log D')\). Since \(\pi_*\) is exact, the projection formula gives
  \begin{equation*}
    \mcalH^i(R\pi_*M)\simeq\pi_*\mcalH^i(M)\simeq\pi_*\pi^*E_i\simeq E_i\otimes_{\cO_{X'}}\pi_*\cO_{X^{(1)}}\simeq E_i.
  \end{equation*}
  Since \(E_i\) is locally free, the spectral sequence for derived pullback has no nonzero higher Tor terms, and hence
  \begin{equation*}
    \mcalH^i(L\pi^*R\pi_*M)\simeq\pi^*E_i.
  \end{equation*}
  Under the identifications
  \begin{equation*}
    \mcalH^i(L\pi^*R\pi_*M)\simeq\pi^*E_i\simeq\mcalH^i(M),
  \end{equation*}
  the morphism \(\mcalH^i(\epsilon_M)\) is identified with the counit
  \begin{equation*}
    \pi^*\pi_*\pi^*E_i \to \pi^*E_i.
  \end{equation*}
  Since \(\pi_*\pi^*E_i\simeq E_i\), this is the identity morphism of \(\pi^*E_i\). Thus \(\epsilon_M\) induces an isomorphism on every cohomology sheaf and is therefore an equivalence.
\end{proof}

\begin{proposition}\label{equivalence-log-relative}
We use \Cref{notation:log-dR}.
Then the following conditions are equivalent:
\begin{enumerate}
    \item $X$ is HT-split over $S$.
    \item $(X,D)$ is log HT-split over $k$.
    \item The morphism
    \[
    F^{\dR}_{X/S}\colon
    \cO_{X^{(1)}} \to F_{X/S,*}\Omega^\bullet_{X/S}
    \]
    splits in $\mcalD_{\qcoh}(X^{(1)})$.
\end{enumerate}
\end{proposition}

\begin{proof}
We first prove $(1)\Rightarrow(2)$.
By \Cref{smooth-HT-split-equiv}, the HT-splitting of $X$ over $S$
implies that
\[
F^{\dR}_{X/S}\colon
\cO_{X^{(1)}} \to F_{X/S,*}\Omega^\bullet_{X/S}
\]
is ind-split.
By \Cref{root-stack-description} and  \Cref{pullback-log-dR}, we have
\[
R\pi_*F^{\dR}_{X/S}
\simeq
F^{\dR}_{(X,D)}\colon
\cO_{X'}\to
F_{X/k,*}\Omega^\bullet_{X/k}(\log D).
\]
Thus $F^{\dR}_{(X,D)}$ is ind-split.

Since $X'$ is a qcqs scheme and
$F_{X/k,*}\Omega^\bullet_{X/k}(\log D)$ is perfect over $X'$, it is compact.
Therefore $F^{\dR}_{(X,D)}$ splits, and hence $(X,D)$ is log HT-split.

Next, we prove $(2)\Rightarrow(3)$.
By \Cref{pullback-log-dR}, there is a natural equivalence
\[
L\pi^*
F_{X/k,*}\Omega^\bullet_{X/k}(\log D)
\simeq
F_{X/S,*}\Omega^\bullet_{X/S},
\]
under which
\[
L\pi^*F^{\dR}_{(X,D)}
\simeq
F^{\dR}_{X/S}.
\]
Hence a splitting of $F^{\dR}_{(X,D)}$ pulls back to a splitting of
$F^{\dR}_{X/S}$.

Finally, $(3)\Rightarrow(1)$ follows immediately from
\Cref{smooth-HT-split-equiv}, since every split morphism is
ind-split.
\end{proof}

\begin{theorem}\label{criterion-decomp-log}
We use \Cref{notation:log-dR}.
Then the following conditions are equivalent:
\begin{enumerate}
\item The morphism $F^{\dR}_{(X,D)}$ splits in $\mcalD_{\qcoh}(X')$.
\item We have an isomorphism
\[
F_{X/k,*}\Omega_{X/k}^{\bullet}(\log D) \simeq \bigoplus_{i=0}^d \Omega_{X'/k}^i(\log D')[-i]
\]
in $\mcalD_{\qcoh}(X')$.
\item We have an isomorphism
\[
\tau_{\leq d-1}F_{X/k,*}\Omega_{X/k}^{\bullet}(\log D) \simeq \bigoplus^{d-1}_{i=0} \Omega_{X'/k}^i(\log D')[-i].
\]
\end{enumerate}
\end{theorem}

\begin{proof}
First, we prove (1) $\Rightarrow$ (2).
It follows from \Cref{equivalence-log-relative} that $F^{\dR}_{X/S}$ splits. 
By \Cref{first-chara-HT-split}, we have the decomposition
\[
F_{X/S,*}\Omega^{\bullet}_{X/S} \simeq \bigoplus_{i=0}^d \Omega^i_{X^{(1)}/S}[-i].
\]
Taking $R\pi_*$, we have an isomorphism
\[
F_{X/k,*}\Omega_{X/k}^{\bullet}(\log D) \simeq \bigoplus_{i=0}^d \Omega_{X'/k}^i(\log D')[-i]
\]
in $\mcalD_{\qcoh}(X')$ by \Cref{pullback-log-dR}, as desired.

The implication (2) $\Rightarrow$ (3) follows from taking $\tau_{\leq d-1}$.

Finally, we prove the implication (3) $\Rightarrow$ (1).
By the proof of \Cref{first-chara-HT-split}, the morphism 
\[
F^{\dR}_{(X,D),\leq d-1} \colon \cO_{X'} \to \tau_{\leq d-1} F_{X/k,*}\Omega^{\bullet}_{X/k}(\log D)
\]
splits in $\mcalD_{\qcoh}(X')$.
We consider the fiber sequence
\[
\tau_{\leq d-1} F_{X/k,*}\Omega^{\bullet}_{X/k}(\log D) \to  F_{X/k,*}\Omega^{\bullet}_{X/k}(\log D) \to \Omega_{X'/k}^d(\log D')[-d].
\]
Taking $\Hom_{\cO_{X'}}(-,\cO_{X'})$, we obtain the surjection
\[
\Hom_{\cO_{X'}}(F_{X/k,*}\Omega^{\bullet}_{X/k}(\log D),\cO_{X'}) \to \Hom_{\cO_{X'}}(\tau_{\leq d-1}F_{X/k,*}\Omega^{\bullet}_{X/k}(\log D),\cO_{X'})
\]
since
\[
\Ext^1(\Omega_{X'/k}^d(\log D')[-d],\cO_{X'}) \simeq H^{d+1}(X',\cHom(\Omega_{X'/k}^d(\log D'),\cO_{X'}))=0.
\]
Therefore, the lift of the splitting map of $F^{\dR}_{(X,D),\leq d-1}$ defines a splitting map of $F^{\dR}_{(X,D)}$, as desired.
\end{proof}

\begin{corollary}\label{vanishing}
We use \Cref{notation:log-dR}.
We assume $X$ is projective over $k$ and $F^{\dR}_{(X,D)}$ splits in $\mcalD_{\qcoh}(X')$.
Then  for every ample line bundle \(L\) on \(X\), we have
\[
H^i(X,\Omega^j_{X/k}(\log D)\otimes L^{-1})=H^i(X,\Omega^j_{X/k}(\log D)(-D)\otimes L^{-1})=0 \qquad
(i+j<d) 
\]
and
\[
H^i(X,\Omega^j_{X/k}(\log D)\otimes L)=H^i(X,\Omega^j_{X/k}(\log D)(-D)\otimes L)=0
\qquad
(i+j>d).
\]
Furthermore, the logarithmic Hodge-to-de Rham spectral sequence is $E_1$-degenerate.
\end{corollary}

\begin{proof}
It follows from the condition (2) in \Cref{criterion-decomp-log}.
\end{proof}

\section{Criteria and permanence properties for HT-splitting} \label{SectionDescent}
In this section, we study criteria and permanence properties for HT-splitting of schemes and pairs.
We recall the definition of HT-splitting for schemes and simple normal crossings pairs (\Cref{Def-log-HT-split} and \Cref{relative-vs-absolute}):
Let $X$ be a $\F_p$-scheme.
We set $\OdR{X} \defeq R\nu_{X,*}\cO_{X^{\dR}}$.
\begin{itemize}
\item We say that $X$ is \emph{HT-split} if the natural homomorphism 
\[
\nu_X^{\sharp} \colon \cO_X \to \OdR{X}
\]
is ind-split in $\mcalD_{\qcoh}(X)$.
\item
We assume that $X$ is a qcqs smooth scheme over a perfect field and $D$ is a simple normal crossings divisor on $X$.
Then we say that $(X,D)$ is log HT-split if the homomorphism
\[
F^{\dR}_{(X,D)} \colon \cO_X \to F_*\Omega_X^{\bullet}(\log D)
\]
splits in $\mcalD_{\qcoh}(X)$.
\end{itemize}

\subsection{Basic criteria and permanence properties}

\begin{proposition}\label{chara-injectivity}
Let $k$ be a perfect field of characteristic $p>0$, and let
$X$ be a smooth proper $k$-scheme of dimension $d$.
Let $D$ be a simple normal crossings divisor on $X$.
Then the following conditions are equivalent:
\begin{enumerate}
\item The pair $(X,D)$ is log HT-split.
\item The homomorphism
\[
H^d(X,\omega_X)
\xrightarrow{
  F^{\dR}_{(X,D)}\otimes\omega_X
}
H^d\left(
X,
F_*\Omega^\bullet_X(\log D)\otimes\omega_X
\right)
\]
is injective.
\end{enumerate}
\end{proposition}

\begin{proof}
It follows from \cite{ishizuka2026Localglobal}*{Proposition~3.6}.
\end{proof}

\begin{proposition}\label{finite-cover}
Let $f \colon Y \to X$ be a qcqs morphism of qcqs $\F_p$-schemes.
\begin{enumerate}
\item If $Y$ is HT-split and $\cO_X \to Rf_*\cO_Y$ splits, then $X$ is HT-split.
\item If $X$ is HT-split and $f$ is \'etale, then $Y$ is HT-split.
\end{enumerate}
\end{proposition}

\begin{proof}
The assertion (1) follows from the commutative diagram
\[
\begin{tikzcd}
\cO_X \arrow[r,"\nu_X^{\sharp}"] \arrow[d,"f^{\sharp}"] & \OdR{X} \arrow[d] \\
Rf_*\cO_Y \arrow[r,"Rf_*\nu_Y^{\sharp}"] & Rf_*\OdR{Y}
\end{tikzcd}
\]
and $Rf_*\nu_Y^{\sharp}$ is ind-split since $f$ is qcqs.

We prove (2).
Since $f$ is \'etale, we have
\[
Y^{\dR} \simeq X^{\dR} \times_X Y
\]
by \cite{barz2025Logarithmic}*{Proposition~2.38}.
Therefore, we have
\[
\nu_Y^{\sharp} \simeq Lf^*\nu_X^{\sharp}.
\]
Thus, if $\nu_X^{\sharp}$ splits, then so does $\nu_Y^{\sharp}$, as desired.
\end{proof}

\begin{proposition}\label{W_2-lift-to-HT-split}
Let $k$ be a perfect field and $X$ a qcqs smooth $k$-scheme.
Let $D$ be a simple normal crossings divisor on $X$.
If $(X,D)$ is $W_2$-liftable and $p \geq \dim X$, then $(X,D)$ is log HT-split.
\end{proposition}

\begin{proof}
It follows from \cite{DI87}, \cite{Kat89}*{Theorem~(4.12)}, and \Cref{criterion-decomp-log}.
\end{proof}

\begin{proposition}\label{Decomposable-case}
Let $k$ be a perfect field and $X$ a qcqs smooth $k$-scheme.
Let $D$ be a simple normal crossings divisor on $X$.
We assume that $(X,D)$ is $W_2$-liftable and $\Omega^1_X(\log D)$ is a direct sum of subbundles of rank less than $p$.
Then $(X,D)$ is log HT-split.
\end{proposition}

\begin{proof}
It follows from \cite{ZhangDeRhamHiggsComparison}*{Corollary~1.3}.
\end{proof}

\begin{lemma}\label{relative-dR-base-change}
Let $\iota\colon X \to  S$ be a smooth representable morphism of \(\setF_p\)-stacks, where \(S\) is an fpqc algebraic stack. 
Let $h\colon S' \to  S$ be a morphism of fpqc algebraic stacks, and set
\[
X' \defeq X\times_S S'.
\]
Let $g\colon (X')^{(1)}_{/S'} \to  X^{(1)}_{/S}$ be the induced morphism between the relative Frobenius twists. 
Then we have
\[
(X'/S')^{\dR}
\simeq
(X/S)^{\dR}
\times_{X^{(1)}_{/S}}
(X')^{(1)}_{/S'}.
\]
In particular, if \(X\) is HT-split over \(S\), then \(X'\) is
HT-split over \(S'\).
\end{lemma}

\begin{proof}
Since the construction \(T\mapsto T^{\dR}\) is defined by precomposition with
\[
R \mapsto  W(R)/^L p,
\]
it preserves fiber products. 
Hence
\[
(X')^{\dR}
\simeq
X^{\dR}\times_{S^{\dR}}(S')^{\dR}.
\]
It follows that
\[
\begin{aligned}
(X'/S')^{\dR}
&=
(X')^{\dR}\times_{(S')^{\dR}}S'\\
&\simeq
\left(
X^{\dR}\times_{S^{\dR}}(S')^{\dR}
\right)
\times_{(S')^{\dR}}S'\\
&\simeq
X^{\dR}\times_{S^{\dR}}S'\\
&\simeq
(X/S)^{\dR}\times_S S'.
\end{aligned}
\]
On the other hand, the naturality of the Frobenius gives
\[
(X')^{(1)}_{/S'}
\simeq
X^{(1)}_{/S}\times_S S'.
\]
The compatibility of the morphisms \(\nu_{X/S}\) with base change therefore gives the asserted Cartesian diagram.

Applying \Cref{E-sharp-gerbe-arbitrary-base-change} for the pullback diagram
\[
\begin{tikzcd}
(X'/S')^{\dR} \arrow[r] \arrow[d,"\nu_{X'/S'}"] & (X/S)^{\dR} \arrow[d,"\nu_{X/S}"] \\
(X')^{(1)}_{/S'} \arrow[r, "g"] & X^{(1)}_{/S},
\end{tikzcd}
\]
we obtain an equivalence
\[
\mcalO^{\dR}_{X'/S'}
\simeq
Lg^*\mcalO^{\dR}_{X/S} \in \mcalD_{\qcoh}((X')^{(1)}_{/S'}),
\]
and the natural homomorphism
\[
\nu^\sharp_{X'/S'}\colon
\mcalO_{(X')^{(1)}_{/S'}}
 \to 
\mcalO^{\dR}_{X'/S'}
\]
identifies with \(Lg^*\nu^\sharp_{X/S}\).
Since \(Lg^*\) preserves filtered colimits and carries split morphisms to split morphisms, it preserves ind-split morphisms.
Hence the ind-splitting of \(\nu^\sharp_{X/S}\) induces that of \(\nu^\sharp_{X'/S'}\).
\end{proof}

\begin{proposition}\label{HT-fiber-product}
Let \(S\) be an fpqc algebraic \(\F_p\)-stack, and let
\[
\iota_X\colon X \to  S,
\qquad
\iota_Y\colon Y \to  S
\]
be smooth representable morphisms.
Assume that \(X\) and \(Y\) are HT-split over \(S\).
Then
\[
Z \defeq X\times_S Y
\]
is HT-split over \(S\).
\end{proposition}

\begin{proof}
Let
\[
p_X\colon Z \to  X,
\qquad
p_Y\colon Z \to  Y
\]
be the projections.
There is a canonical isomorphism
\[
Z^{(1)}_{/S}
\simeq
X^{(1)}_{/S}\times_S Y^{(1)}_{/S}.
\]
We denote by
\[
p_X^{(1)}\colon Z^{(1)}_{/S} \to  X^{(1)}_{/S},
\qquad
p_Y^{(1)}\colon Z^{(1)}_{/S} \to  Y^{(1)}_{/S}
\]
the induced projections.

The K\"unneth isomorphism for relative de Rham complexes and the
compatibility of relative Frobenius with fiber products give a
canonical isomorphism
\[
F_{Z/S,*}\Omega^\bullet_{Z/S}
\simeq
L(p_X^{(1)})^*
F_{X/S,*}\Omega^\bullet_{X/S}
\otimes_{\mcalO_{Z^{(1)}_{/S}}}^{L}
L(p_Y^{(1)})^*
F_{Y/S,*}\Omega^\bullet_{Y/S}
\]
in
\(\mcalD_{\qcoh}(Z^{(1)}_{/S})\).
Under this isomorphism, the morphism
\[
F^{\dR}_{Z/S}\colon
\mcalO_{Z^{(1)}_{/S}}
 \to 
F_{Z/S,*}\Omega^\bullet_{Z/S}
\]
is identified with the composition of the tensor product of the pullbacks of \(F^{\dR}_{X/S}\) and \(F^{\dR}_{Y/S}\).

By \Cref{smooth-HT-split-equiv}, the morphisms \(F^{\dR}_{X/S}\) and \(F^{\dR}_{Y/S}\) are ind-split.
Hence \(F^{\dR}_{Z/S}\) is ind-split, and
\Cref{smooth-HT-split-equiv} shows that \(Z\) is HT-split over
\(S\).
\end{proof}

\begin{corollary}\label{log-HT-product}
Let \(k\) be a perfect field of characteristic \(p>0\).
For \(i=1,2\), let \(X_i\) be a smooth qcqs \(k\)-scheme and let
\(D_i\) be a simple normal crossings divisor on \(X_i\).
Assume that both pairs
\[
(X_1,D_1)
\qquad\text{and}\qquad
(X_2,D_2)
\]
are log HT-split.
Then the pair
\[
\left(
X_1\times_k X_2,\,
\pr_1^*D_1+\pr_2^*D_2
\right)
\]
is log HT-split.
\end{corollary}

\begin{proof}
Write
\[
D_i=D_{i,1}+\cdots+D_{i,r_i}
\]
and set
\[
S_i \defeq 
\left[\mathbb{A}^1_k/\mathbb{G}_{m,k}\right]^{r_i}.
\]
The divisor \(D_i\) determines a smooth representable morphism
\[
\iota_i\colon X_i \to  S_i
\]
as in \Cref{notation:log-dR}.
By \Cref{equivalence-log-relative}, the log HT-splitting of
\((X_i,D_i)\) is equivalent to the HT-splitting of \(X_i\) over
\(S_i\).

Set
\[
S \defeq S_1\times_k S_2.
\]
By \Cref{relative-dR-base-change},
\[
X_1\times_k S_2 \to  S
\qquad\text{and}\qquad
S_1\times_k X_2 \to  S
\]
are HT-split over \(S\).
Their fiber product over \(S\) is canonically isomorphic to
\[
X_1\times_k X_2.
\]
Hence \Cref{HT-fiber-product} shows that
\(X_1\times_k X_2\) is HT-split over \(S\).

The induced morphism
\[
X_1\times_k X_2 \to  S_1\times_k S_2
\]
is precisely the morphism associated with the simple normal
crossings divisor
\[
\pr_1^*D_1+\pr_2^*D_2.
\]
Therefore \Cref{equivalence-log-relative} implies that
\[
\left(
X_1\times_k X_2,\,
\pr_1^*D_1+\pr_2^*D_2
\right)
\]
is log HT-split.
\end{proof}

\subsection{Comparison with quasi-\texorpdfstring{\(F\)}{F}-splitting}

For an integer \(r\geq 1\), let \(W_r(-)\) denote the fpqc
sheafification of the left Kan extension of the functor
\[
\Spec A \mapsto  \Spec W_r(A)
\]
from affine \(\setF_p\)-schemes. We put
\[
W_rT_{\setF_p}
 \defeq 
W_rT\times_{\Spec(\setZ/p^r)}\Spec(\setF_p)
\]
for every \(\setF_p\)-stack \(T\).

Recall the following construction from \cite{petrov2025Decomposition}*{Corollary~4.4}:

\begin{construction} \label{ConstPetrov}
    Let \(T\) be an algebraic fpqc \(\setF_p\)-stack. Then we construct a morphism
    \begin{equation*}
      \sigma_T \colon WT_{\setF_p} \to T^{\dR}
    \end{equation*}
    and a commutative diagram
    \begin{equation*}
      \begin{tikzcd}
      WT_{\setF_p} \arrow[r, "s_T"] \arrow[rd, "\sigma_T"'] & T                           \\
                                                            & T^{\dR} \arrow[u, "\nu_T"']
      \end{tikzcd}
    \end{equation*}
    of \(\setF_p\)-stacks, where \(\nu_T \colon T^{\dR} \to T\) is the natural morphism given in \Cref{DefdeRhamStack} and \(s_T \colon WT_{\setF_p} \to T\) is the morphism induced by the morphism \(\Phi_R \colon R \to W(R)/^L p\) sending  \(r \in R\) to \([r^p]\) in \(W(R)\).

    We recall this construction: Since \(T\) is an algebraic fpqc \(\setF_p\)-stack, we can consider its standard simplicial affine presentation and we may assume that \(T = \Spec(R)\) is an affine \(\setF_p\)-scheme.
    Fix any animated \(\setF_p\)-algebra \(S\). By the argument using \(\setW T_{\setF_p}\) in \cite{petrov2025Decomposition}*{Corollary~4.4}, the point of \(WT_{\setF_p}\) over \(\Spec(S)\) gives rise to a morphism \(f \colon W(R)/^L p \to S\) in \(\CAlg^{an}_{\setF_p}\).
    The composition \(s_T \colon WT_{\setF_p} \to T\) with \(\Spec(S) \to WT_{\setF_p}\) corresponds to the morphism
    \begin{equation*}
      R \xrightarrow{\Phi_R} W(R)/^L p \xrightarrow{f} S
    \end{equation*}
    in \(\CAlg^{an}_{\setF_p}\).

    On the other hand, we can apply \cite{petrov2025Decomposition}*{Proposition~4.2} to produce the morphism
    \begin{equation*}
      \Spec(S) \to WT_{\setF_p} \xrightarrow{\exists \widetilde{s}_{WT}} (WT_{\setF_p})^{\dR}
    \end{equation*}
    by the \(\delta\)-structure on \(W(R)\). More precisely, the \(\delta\)-structure on \(W(R)\) uniquely gives rise to a morphism \(\widetilde{f} \colon W(R) \to W(S)\) of animated \(\delta\)-rings, which is characterized by the property that its composition with the restriction morphism \(W(S) \to S\) is equal to the composition \(W(R) \to W(R)/^L p \xrightarrow{f} S\).
    Then we obtain a morphism
    \begin{equation*}
      R \xrightarrow{\Phi_R} W(R)/^L p \xrightarrow{\widetilde{f}} W(S)/^L p
    \end{equation*}
    in \(\CAlg^{an}_{\setF_p}\), which corresponds to the morphism \(\sigma_T \colon WT_{\setF_p} \to T^{\dR}\) over \(\Spec(S)\).
    The composition with \(\nu_T \colon T^{\dR} \to T\) is equal to the morphism \(R \xrightarrow{\Phi_R} W(R)/^L p \xrightarrow{f} S\), which corresponds to the morphism \(\Spec(S) \to WT_{\setF_p} \xrightarrow{s_T}  T\).

    Consequently, we have a commutative diagram
    \begin{equation} \label{eq:WT-dR}
      \begin{tikzcd}
      WT_{\setF_p} \arrow[r, "s_T"] \arrow[rd, "{\sigma_{T}}"'] \arrow[d, "\widetilde{s}_{WT}"'] & T                           \\
      (WT_{\setF_p})^{\dR} \arrow[r, "(s_T)^{\dR}"']                                                & T^{\dR} \arrow[u, "\nu_T"']
      \end{tikzcd}
    \end{equation}
    of \(\setF_p\)-stacks, where \(\widetilde{s}_{WT} \colon WT_{\setF_p} \to (WT_{\setF_p})^{\dR}\) is the natural morphism arising from the \(\delta\)-structure on \(WT\).
\end{construction}

\begin{proposition} \label{compatibility-delta-lift}
  Let \(k\) be a perfect field of characteristic \(p\) and let \(T\) be an algebraic fpqc \(k\)-stack.
  Assume that \(T\) admits an algebraic fpqc \(W(k)\)-\(\delta\)-lifting \(\widetilde T\), together with an isomorphism $\widetilde T_{\F_p}\simeq T$.
  Then there exists a morphism
  \begin{equation*}
    s' \colon WT_{\setF_p} \to T
  \end{equation*}
  and a commutative diagram
  \begin{equation*}
    \begin{tikzcd}
    WT_{\setF_p} \arrow[d, "s'"'] \arrow[rd, "\sigma_{T}"] &         \\
    T \arrow[r, "\pi_T"']                                  & T^{\dR}
    \end{tikzcd}
  \end{equation*}
  of \(\setF_p\)-stacks, where \(\sigma_T\) is the morphism constructed in \Cref{ConstPetrov}.
\end{proposition}

\begin{proof}
  As in the proof of \cite{petrov2025Decomposition}*{Corollary~4.4}, we may assume that \(T = \Spec(R)\) is an affine \(k\)-scheme.
  Then the \(\delta\)-structure on \(\widetilde{T} = \Spec(\widetilde{R})\) gives rise to a unique morphism \(\Psi_{\widetilde{R}} \colon \widetilde{R} \to W(R)\) of \(\delta\)-rings whose composition with the restriction morphism \(\res \colon W(R) \to R\) is equal to the canonical morphism \(\widetilde{R} \to \widetilde{R}/^L p \cong R\).
  This morphism induces a morphism
  \begin{equation*}
    R \cong \widetilde{R}/^L p \xrightarrow{\Psi_{\widetilde{R}}/^L p} W(R)/^L p
  \end{equation*}
  of animated \(k\)-algebras, which corresponds to the morphism \(s' \colon WT_{\setF_p} \to T\) over \(\Spec(k)\).

  It remains to show that, for any \(f \colon W(R)/^L p \to S\), the composition
  \begin{equation*}
    R \xrightarrow{\Psi_{\widetilde{R}}/^L p} W(R)/^L p \xrightarrow{f} S \xrightarrow{\Phi_S} W(S)/^L p,
  \end{equation*}
  which is the image of \(f\) via the composition \(WT_{\setF_p} \xrightarrow{s'} T \xrightarrow{\pi_T} T^{\dR}\), is equal to the morphism
  \begin{equation*}
    R \xrightarrow{\Phi_R} W(R)/^L p \xrightarrow{\widetilde{f}} W(S)/^L p,
  \end{equation*}
  which is the image of \(f\) via the composition \(WT_{\setF_p} \xrightarrow{\sigma_T} T^{\dR}\) by \eqref{eq:WT-dR}.
  Consider the following diagram of animated \(k\)-algebras:
  \begin{equation*}
    \begin{tikzcd}
    R \arrow[r, "\Psi_{\widetilde{R}}/^L p"] \arrow[rd, "\Phi_R"'] & W(R)/^L p \arrow[d, "F_{W(R)/^L p}"] \arrow[r, "f"] & S \arrow[d, "\Phi_S"] \\
                                                    & W(R)/^L p \arrow[r, "\widetilde{f}"]                & W(S)/^L p.            
    \end{tikzcd}
  \end{equation*}
  Using the following diagram of quasi-ideals for the \(\setF_p\)-algebra \(S\)
  \begin{equation*}
    \begin{tikzcd}
    W(S) \arrow[d, "F"] \arrow[r, "\times p"] & W(S) \arrow[d, "\id"] \arrow[r] & W(S)/^L p \arrow[d, "\res"] \\
    W(S) \arrow[d, "\id"] \arrow[r, "V"]      & W(S) \arrow[d, "F"] \arrow[r]   & S \arrow[d, "\Phi_S"]       \\
    W(S) \arrow[r, "\times p"]                & W(S) \arrow[r]                  & W(S)/^L p,                  
    \end{tikzcd}
  \end{equation*}
  we can show that
  \begin{equation*}
    S \xrightarrow{\Phi_S} W(S)/^L p \xrightarrow{\res} S \quad \text{and} \quad W(S)/^L p \xrightarrow{\res} S \xrightarrow{\Phi_S} W(S)/^L p
  \end{equation*}
  are equivalent to the Frobenius morphism on \(S\) and \(W(S)/^L p\), respectively.
  Therefore, the above diagram of animated \(\setF_p\)-algebras is commutative as desired.
\end{proof}

\begin{corollary}\label{compatibility-pi-sigma}
  Let \(k\) be a perfect field and let
  \begin{equation*}
    S \defeq \bracketlr{\setA^1_k/\setG_{m,k}}^n.
  \end{equation*}
  This admits a \(\delta\)-lifting
  \begin{equation*}
    S_{W(k)} \defeq \bracketlr{\setA^1_{W(k)}/\setG_{m,W(k)}}^n
  \end{equation*}
  over \(W(k)\) with Frobenius lift given on the standard coordinates by
  \begin{equation*}
    t_i  \mapsto  t_i^p, \qquad u_i  \mapsto  u_i^p.
  \end{equation*}
  Then, for every integer \(r \geq 1\), there exists a morphism
  \begin{equation*}
    s'_r \colon W_rS_{\setF_p} \to S
  \end{equation*}
  and a commutative diagram
  \begin{equation*}
    \begin{tikzcd}
    W_rS_{\setF_p} \arrow[d, "s'_r"'] \arrow[rd, "\sigma_{S,r}"] &         \\
    S \arrow[r, "\pi_S"']                                  & S^{\dR}
    \end{tikzcd}
  \end{equation*}
  of \(\setF_p\)-stacks, where \(\pi_S \colon S \to S^{\dR}\) is the natural morphism given in \Cref{DefdeRhamStack} and \(\sigma_{S,r} \colon W_rS_{\setF_p} \to S^{\dR}\) is the morphism obtained by \cite{petrov2025Decomposition}*{Corollary~4.4} (\Cref{ConstPetrov}).
\end{corollary}

\begin{proof}
  Since \(S\) admits an algebraic fpqc \(\delta\)-lifting \(S_{W(k)}\) over \(W(k)\) and we can combine the morphism \(W_rS_{\setF_p} \to WS_{\setF_p}\), the assertion follows from \Cref{compatibility-delta-lift}.
\end{proof}

\begin{proposition}\label{sigma-over-S}
We use \Cref{notation:log-dR}.
We assume that $k$ is perfect.
Then there exists a morphism $\sigma_{X/S,r} \colon W_rX_{\F_p} \to (X/S)^{\dR}$ over $X$ for every $r \in \Z_{\geq 1}$.
\end{proposition}

\begin{proof}
We consider the commutative diagram:
\[
\begin{tikzcd}
W_rX_{\F_p} \arrow[r,"\sigma_{X,r}"] \arrow[d,"W\iota_{\F_p}"] & X^{\dR} \arrow[dd,"\iota^{\dR}"] \\
W_rS_{\F_p} \arrow[rd,"\sigma_{S,r}"] \arrow[d,"s_r'"] & \\
S \arrow[r,"\pi_S"] & S^{\dR},
\end{tikzcd}
\]
where $s'_r$ is defined in \Cref{compatibility-pi-sigma} and $\sigma_{X,r}$ and $\sigma_{S,r}$ are defined in \cite{petrov2025Decomposition}*{Corollary~4.4}.
Therefore, the commutative diagram induces
\[
\sigma_{X/S,r} \colon W_rX_{\F_p} \to X^{\dR} \times_{S^{\dR}} S = (X/S)^{\dR}
\]
such that the following diagram commutes:
\[
\begin{tikzcd}
    W_rX_{\F_p} \arrow[r,"\sigma_{X/S,r}"] \arrow[rrd,"\alpha_{X,r}"'] & (X/S)^{\dR} \arrow[r] & X^{\dR} \arrow[d,"\nu_X"]\\
    & & X
\end{tikzcd}
\]
by \cite{petrov2025Decomposition}*{Corollary~4.4}, where $\alpha_{X,r}$ is the composition of \(s_X \colon WX_{\setF_p} \to X\) and \(W_rX_{\setF_p} \to WX_{\setF_p}\) defined by \(\Phi_R\) and \(\res\) by \eqref{eq:WT-dR}.
\end{proof}

\begin{theorem}\label{qFs-to-HT}
Let $k$ be a perfect field and $X$ a qcqs $k$-scheme.
We assume that $X$ is quasi-$F$-split.
Then $X$ is HT-split.
Furthermore, if $X$ is smooth over $k$ and $D$ is a simple normal crossings divisor on $X$, then $(X,D)$ is log HT-split.
\end{theorem}

\begin{proof}
The first assertion follows from \cite{petrov2025Decomposition}*{Corollary~4.4}.

We prove the second assertion.
We use the notation \Cref{notation:log-dR}.
By \Cref{sigma-over-S}, the canonical morphism \(\Phi_X \colon \mcalO_X \to F_*W_r\mcalO_X/p\) can be identified with a composition
\begin{equation*}
    \mcalO_X \xrightarrow{F_{(X, D)}{^{\dR}}} F_*\Omega_X^{\bullet}(\log D) \cong R\pi_*\OdR{X/S} \xrightarrow{\sigma_{X/S, r}^{\sharp}} F_*W_r\mcalO_X/p
\end{equation*}
in \(\mcalD_{\qcoh}(X)\), where the middle isomorphism follows from \Cref{pullback-log-dR}.
So the splitting of the composition gives the splitting of the first morphism which shows that the pair \((X, D)\) is log HT-split.
\end{proof}

\subsection{Blow-ups along logarithmic strata}

Let \(k\) be a perfect field of characteristic \(p>0\).

We apply the preceding lemma to blow-ups along strata of an SNC divisor.

\begin{proposition}
  \label{HT-ascent-stratum-blowup}
  Let \(X\) be a qcqs smooth \(k\)-scheme and let \(D=D_1+\cdots+D_n\) be a simple normal crossings divisor. Let \(\varnothing\neq I\subseteq\{1,\ldots,n\}\), and assume that \(D_I\defeq\bigcap_{i\in I}D_i\) is nonempty. Let \(\pi\colon Y\defeq\Bl_{D_I}X\to X\) be the blow-up and put \(D_Y\defeq(\pi^{-1}D)_{\red}\). If the pair \((X,D)\) is log HT-split, then the pair \((Y,D_Y)\) is log HT-split. In particular, the underlying scheme \(Y\) is HT-split.
\end{proposition}

\begin{proof}
  If \(|I|=1\), then \(D_I\) is an effective Cartier divisor, so its blow-up is the identity by \citeSta{085V}; hence the assertion is immediate. We therefore assume that \(|I|\geq2\).

  Put \(\mcalA\defeq[\setA_k^1/\setG_{m,k}]\) and \(S\defeq\mcalA^n\). By \cite{barz2025Logarithmic}*{Theorem~5.4}, the divisor \(D\) determines a smooth representable morphism \(\iota_D\colon X\to S\) whose \(i\)-th component classifies the pair \((\mcalO_X(-D_i),s_{D_i})\), where \(s_{D_i}\colon\mcalO_X(-D_i)\to\mcalO_X\) is the canonical inclusion. Let \(j_0\colon B\setG_{m,k}\hookrightarrow\mcalA\) be the closed immersion induced by the origin, let \(\pr_I\colon S\to\mcalA^I\) be the projection, and set
  \begin{equation*}
    \mcalZ_I\defeq S\times_{\mcalA^I}(B\setG_{m,k})^I.
  \end{equation*}

  We explain its pullback to \(X\). For a line bundle \(L\) on a scheme \(T\) and an \(\mcalO_T\)-linear morphism \(s\colon L\to\mcalO_T\), its zero scheme is the closed subscheme
  \begin{equation*}
    Z(s)\defeq V(\operatorname{im}(s))\subseteq T.
  \end{equation*}
  After trivializing \(L\), the morphism \(s\) is multiplication by a function \(f\), and this definition becomes \(Z(s)=V(f)\). The closed immersion \(j_0\colon B\setG_{m,k}=[\{0\}/\setG_{m,k}]\to\mcalA\) is the universal zero locus: if \(c_{(L,s)}\colon T\to\mcalA\) classifies \((L,s)\), then there is a \(2\)-Cartesian square
  \begin{equation*}
    \begin{tikzcd} Z(s) \arrow[r] \arrow[d] & B\setG_{m,k} \arrow[d,"j_0"] \\ T \arrow[r,"c_{(L,s)}"'] & \mcalA. \end{tikzcd}
  \end{equation*}
  Indeed, this assertion may be checked fpqc locally after trivializing \(L\), where it reduces to the ordinary pullback square obtained by pulling back \(\{0\}\hookrightarrow\setA_k^1\) along the function \(f\colon T\to\setA_k^1\).

  Let \(\iota_{D,i}\colon X\to\mcalA\) denote the \(i\)-th component of \(\iota_D\). Applying the preceding universal property to \((\mcalO_X(-D_i),s_{D_i})\), we obtain
  \begin{equation*}
    X\times_{\iota_{D,i},\mcalA,j_0}B\setG_{m,k}\simeq Z(s_{D_i}) \cong V(\mcalI_{D_i})=D_i.
  \end{equation*}
  Writing \(I=\{i_1,\ldots,i_r\}\), associativity of fiber products therefore gives
  \begin{equation*}
    X\times_S\mcalZ_I\simeq X\times_{\mcalA^I}(B\setG_{m,k})^I\simeq D_{i_1}\times_X\cdots\times_XD_{i_r}.
  \end{equation*}
  The last fiber product is the scheme-theoretic intersection of \(D_{i_1},\ldots,D_{i_r}\): its ideal sheaf is \(\mcalI_{D_{i_1}}+\cdots+\mcalI_{D_{i_r}}\). Since \(D\) is a simple normal crossings divisor, étale locally these ideals are generated by distinct members of a regular system of parameters, and hence this scheme-theoretic intersection is precisely \(D_I\). Thus
  \begin{equation*}
    X\times_S\mcalZ_I\simeq D_I
  \end{equation*}
  as closed subschemes of \(X\).

  Set \(\rho\colon\widetilde{S}\defeq\Bl_{\mcalZ_I}S\to S\). Since \(\iota_D\) is flat, compatibility of blow-ups with flat base change, checked after a smooth atlas of \(S\) and using \citeSta{0805}, gives
  \begin{equation*}
    X\times_S\widetilde{S}\simeq\Bl_{X\times_S\mcalZ_I}X\simeq\Bl_{D_I}X=Y.
  \end{equation*}
  Thus there is a Cartesian diagram
  \begin{equation*}
    \begin{tikzcd} Y \arrow[r,"\widetilde{\iota}"] \arrow[d,"\pi"'] & \widetilde{S} \arrow[d,"\rho"] \\ X \arrow[r,"\iota_D"'] & S. \end{tikzcd}
  \end{equation*}
  Since \((X,D)\) is log HT-split, \Cref{equivalence-log-relative} shows that \(X\) is HT-split over \(S\). Hence \Cref{relative-dR-base-change} shows that \(Y\) is HT-split over \(\widetilde{S}\).

  We now identify the boundary target after the blow-up. Let \(q\colon\setA_k^n\to S=[\setA_k^n/\setG_{m,k}^n]\) be the standard atlas, with coordinates \(x_1,\ldots,x_n\). The pullback of \(B\setG_{m,k}\hookrightarrow\mcalA\) along the \(i\)-th coordinate morphism \(\setA_k^n\to\mcalA\) is \(V(x_i)\). Consequently,
  \begin{equation*}
    \setA_k^n\times_S\mcalZ_I\simeq V(x_i\mid i\in I)\defeq V_I.
  \end{equation*}
  Since \(q\) is a \(\setG_{m,k}^n\)-torsor, flat base change of blow-ups gives
  \begin{equation*}
    \widetilde{S}\times_S\setA_k^n\simeq\Bl_{V_I}\setA_k^n.
  \end{equation*}
  The descent datum on this blow-up is the lift of the coordinatewise \(\setG_{m,k}^n\)-action, because the center \(V_I\) is invariant. Therefore
  \begin{equation*}
    \widetilde{S}\simeq\bracketlr{(\Bl_{V_I}\setA_k^n)/\setG_{m,k}^n}.
  \end{equation*}

  Let \(U\defeq\setA_k^{n+1}\setminus V(y_i\mid i\in I)\), with coordinates \(y_1,\ldots,y_n,e\), and fix an action of \(\setG_{m,k}\) given by \(y_i\mapsto\lambda y_i\) for \(i\in I\), by \(y_i\mapsto y_i\) for \(i\notin I\), and by \(e\mapsto\lambda^{-1}e\). Define a morphism
  \begin{equation*}
    \psi\colon U \to \setA_k^n,\qquad x_i=y_ie\ \text{for \(i\in I\)},\qquad x_i=y_i\ \text{for \(i\notin I\)}.
  \end{equation*}
  This morphism is invariant under the \(\setG_{m,k}\)-action and induces a morphism \(\bracketlr{U/\setG_{m,k}}\to\setA_k^n\). To identify its quotient, cover \(U\) by the invariant open subsets \(U_i=D(y_i)\) for \(i\in I\). On \(U_i\), each orbit has a unique representative with \(y_i=1\), and hence
  \begin{equation*}
    \bracketlr{U_i/\setG_{m,k}}\simeq\Spec\left(k\left[x_i,\left(u_j\right)_{j\in I\setminus\bracelr{i}},\left(x_\ell\right)_{\ell\notin I}\right]\right),\qquad u_j=\frac{y_j}{y_i},\quad x_i=y_ie,\quad x_\ell=y_\ell.
  \end{equation*}
  Under the induced morphism to \(\setA_k^n\), one has \(x_j=u_jx_i\) for \(j\in I\setminus\{i\}\). By the Rees-algebra description in \citeSta{01OF}, this is exactly the standard affine chart of \(\Bl_{V_I}\setA_k^n\) on which the generator \(x_i\) of the center becomes principal. The transition maps on the intersections \(U_i\cap U_j\) are the usual changes of ratios \(x_\ell/x_i\), so these local identifications glue and give an isomorphism
  \begin{equation*}
    \Bl_{V_I}\setA_k^n\simeq\bracketlr{U/\setG_{m,k}}
  \end{equation*}
  over \(\setA_k^n\).

  The coordinatewise action of \(\setG_{m,k}^n\) on \(\setA_k^n\) lifts to \(U\) by \(y_i\mapsto t_iy_i\) and \(e\mapsto e\), and it commutes with the above \(\setG_{m,k}\)-action on \(U\). The resulting action of \(\setG_{m,k}^n\times\setG_{m,k}\) is given by
  \begin{equation*}
    (t_1,\ldots,t_n,\lambda)\cdot y_i=\begin{cases}t_i\lambda y_i&\text{if \(i\in I\)},\\t_iy_i&\text{if \(i\notin I\)},\end{cases}\qquad (t_1,\ldots,t_n,\lambda)\cdot e=\lambda^{-1}e.
  \end{equation*}
  Let \(\varepsilon_1,\ldots,\varepsilon_n,\varepsilon\) denote the standard basis of the character lattice of \(\setG_{m,k}^n\times\setG_{m,k}\). The coordinate characters of \(y_1,\ldots,y_n,e\) are
  \begin{equation*}
    \chi_i=\varepsilon_i+\varepsilon\ \text{for \(i\in I\)},\qquad \chi_i=\varepsilon_i\ \text{for \(i\notin I\)},\qquad \chi_e=-\varepsilon.
  \end{equation*}
  These characters form a basis, since
  \begin{equation*}
    \varepsilon=-\chi_e,\qquad \varepsilon_i=\chi_i+\chi_e\ \text{for \(i\in I\)},\qquad \varepsilon_i=\chi_i\ \text{for \(i\notin I\)}.
  \end{equation*}
  Hence an automorphism of the torus \(\setG_{m,k}^n\times\setG_{m,k}\) identifies this action on \(U\) with the standard coordinatewise action of \(\setG_{m,k}^{n+1}\). Combining the preceding quotient descriptions, we obtain
  \begin{equation*}
    \widetilde{S}\simeq\bracketlr{(\Bl_{V_I}\setA_k^n)/\setG_{m,k}^n}\simeq\bracketlr{U/(\setG_{m,k}^n\times\setG_{m,k})}\simeq\bracketlr{U/\setG_{m,k}^{n+1}}.
  \end{equation*}

  Set \(T_Y\defeq\mcalA^{n+1}=[\setA_k^{n+1}/\setG_{m,k}^{n+1}]\). Since \(U\subseteq\setA_k^{n+1}\) is an invariant open subscheme, it induces a representable open immersion
  \begin{equation*}
    j\colon\widetilde{S}=\bracketlr{U/\setG_{m,k}^{n+1}} \to \bracketlr{\setA_k^{n+1}/\setG_{m,k}^{n+1}}=T_Y.
  \end{equation*}

  Let \((\mcalL_i,s_i^{\mathrm{univ}})\) for \(1\leq i\leq n\) and \((\mcalL_E,s_E^{\mathrm{univ}})\) denote the \(n+1\) universal line bundles with sections on \(T_Y\). Their zero loci are the coordinate divisors
  \begin{equation*}
    \mcalH_i\defeq\bracketlr{V(y_i)/\setG_{m,k}^{n+1}},\qquad \mcalH_E\defeq\bracketlr{V(e)/\setG_{m,k}^{n+1}}
  \end{equation*}
  of \(T_Y\).
  The restrictions \(j^*\mcalH_i\) and \(j^*\mcalH_E\) are represented as quotients of \(V(y_i)\cap U\) and \(V(e)\cap U\) by \(\setG_{m, k}^{n+1}\), respectively.
  We denote them by \(V(y_i)\) and \(V(e)\) for simplicity.

  Write \(D_Y=\widetilde D_1+\cdots+\widetilde D_n+E\), where \(E\) is the exceptional divisor and \(\widetilde D_i\) is the strict transform of \(D_i\). On the Cox presentation, the blow-up formulas give
  \begin{equation*}
    \rho^*V(x_i)=V(y_i)+V(e)\ \text{for \(i\in I\)},\qquad \rho^*V(x_i)=V(y_i)\ \text{for \(i\notin I\)}.
  \end{equation*}
  For \(i\in I\), the divisor \(V(y_i)\) is the closure of the inverse image of \(V(x_i)\setminus V_I\), and hence is the strict transform of \(V(x_i)\), while \(V(e)\) is the inverse image of the center and hence is the exceptional divisor. For \(i\notin I\), the center \(V_I\) is not contained in \(V(x_i)\), and the equality \(x_i=y_i\) shows that \(V(y_i)\) is again the strict transform of \(V(x_i)\). After base change along \(\iota_D\), the coordinate divisor \(V(x_i)\) pulls back to \(D_i\), so
  \begin{equation*}
    (j\circ\widetilde{\iota})^*\mcalH_i=\widetilde D_i,\qquad (j\circ\widetilde{\iota})^*\mcalH_E=E.
  \end{equation*}
  Equivalently, one has isomorphisms of line bundles with sections
  \begin{equation*}
    (j\circ\widetilde{\iota})^*(\mcalL_i,s_i^{\mathrm{univ}})\simeq(\mcalO_Y(-\widetilde D_i),s_{\widetilde D_i}),\qquad (j\circ\widetilde{\iota})^*(\mcalL_E,s_E^{\mathrm{univ}})\simeq(\mcalO_Y(-E),s_E).
  \end{equation*}
  Notice that \(E\) occurs with multiplicity \(|I|\) in the total transform \(\pi^*D\), since \(\pi^*D_i=\widetilde D_i+E\) for every \(i\in I\), but it occurs only once in the reduced total transform. Hence
  \begin{equation*}
    D_Y=(\pi^{-1}D)_{\red}=\widetilde D_1+\cdots+\widetilde D_n+E.
  \end{equation*}
  A morphism to \(T_Y=\mcalA^{n+1}\) is equivalent to an ordered collection of \(n+1\) line bundles with sections. The preceding isomorphisms therefore give a canonical \(2\)-isomorphism
  \begin{equation*}
    \iota_{D_Y}\simeq j\circ\widetilde{\iota}\colon Y \to \widetilde{S} \to  T_Y,
  \end{equation*}
  where \(\iota_{D_Y}\) is the divisor-classifying morphism associated with the ordered components \(\widetilde D_1,\ldots,\widetilde D_n,E\). The standard blow-up charts also show that \(Y\) is smooth and that \(D_Y\) is a simple normal crossings divisor.

  Since \(j\) is a representable open immersion, it is étale. Its relative Frobenius is an isomorphism by \citeSta{0EBS}, after checking on a scheme atlas, and \cite{barz2025Logarithmic}*{Proposition~2.40} gives an equivalence
  \begin{equation*}
    \widetilde{S}\xrightarrow{\sim}(\widetilde{S}/T_Y)^{\dR}.
  \end{equation*}
  Unwinding the definition of relative de Rham stacks and relative Frobenius twists, we obtain natural equivalences
  \begin{equation*}
    (Y/\widetilde{S})^{\dR}\simeq(Y/T_Y)^{\dR},\qquad Y^{(1)}_{/\widetilde{S}}\simeq Y^{(1)}_{/T_Y},
  \end{equation*}
  under which \(\nu_{Y/\widetilde{S}}\) and \(\nu_{Y/T_Y}\) agree by functoriality. The retraction expressing that \(Y\) is HT-split over \(\widetilde{S}\) therefore gives a retraction expressing that \(Y\) is HT-split over \(T_Y\). Applying \Cref{equivalence-log-relative} to the divisor-classifying morphism \(\iota_{D_Y}\), we conclude that the pair \((Y,D_Y)\) is log HT-split.
\end{proof}




\begin{example}\label{example:b-up}
Let \(k\) be an algebraically closed field of characteristic \(p>2\), and take $n \geq p+2$.
Write
\[
S \defeq k[x_0,\ldots,x_n],
\qquad
\mfrakm \defeq (x_0,\ldots,x_n)
\]
and
\[
f \defeq x_0^{p+1}+\cdots+x_n^{p+1}\in S_{p+1}, \qquad D \defeq V(f)\subseteq \setP^n
\]

Set
\[
V
 \defeq 
\sum_{i=0}^n x_i^p S_2
\subseteq
S_{p+2}.
\]
Choose a general element
\[
g=\sum_{i=0}^n x_i^p q_i\in V,
\qquad
q_i\in S_2,
\]
and put $B \defeq V(g)\subseteq\setP^n$.

We claim that $B+D$ is a simple normal crossings divisor.
Indeed, fix a point \(P\in\setP^n\), and choose \(i\) such that
\(x_i(P)\neq 0\). On the standard affine open subset \(D_+(x_i)\), the
subspace
\[
x_i^pS_2\subseteq V
\]
generates all first-order jets at \(P\), because multiplication by \(x_i^p\)
is invertible near \(P\), while the global sections of
\(\mcalO_{\setP^n}(2)\) generate all first-order jets.
Consequently, the natural map
\[
V
 \to 
\mcalO_{\setP^n}(p+2)
\otimes_{\mcalO_{\setP^n}}
\mcalO_{\setP^n,P}/\mfrakm_P^2
\]
is surjective for every \(P\in\setP^n\).
Thus, for a general choice of \(g\), the divisor $B+D$ is simple normal crossings. 

Let
\[
\pi\colon
X \defeq \operatorname{Bl}_{D\cap B}\setP^n
 \to 
\setP^n
\]
be the blow-up along $D\cap B$. 
Denote by \(E\) its exceptional divisor and set
\[
H \defeq \pi^*\mcalO_{\setP^n}(1).
\]

Since \(\setP^n\) is \(F\)-split,  the pair $(\setP^n,D+B)$ is log HT-split by \Cref{qFs-to-HT}.  
Hence \Cref{HT-ascent-stratum-blowup} shows that
\[
\left(
X,
\widetilde D+\widetilde B+E
\right)
\]
is log HT-split, where $E$ is the exceptional divisor on $X$.
In particular, \(X\) is HT-split.

It remains to prove that \(X\) is not quasi-\(F\)-split.
Let
\[
[Z_0:\cdots:Z_n:T]
\]
be the homogeneous coordinates on \(\setP^{n+1}\).
We consider the rational map
\[
\P^n \dashrightarrow \P^{n+1}, \qquad [x_0:\cdots:x_n]
\mapsto
[x_0f:\cdots:x_nf:-g].
\]
Then the rational map is resolved by \(\pi\) and we obtain the morphism
\[
\varphi\colon X \to \setP^{n+1}.
\]
We set
\[
Y
 \defeq 
V\bigl(Tf(Z_0,\ldots,Z_n)+g(Z_0,\ldots,Z_n)\bigr)
\subseteq
\setP^{n+1}.
\]
Since we have
\[
(-g)f(x_0f,\ldots,x_nf)
+
g(x_0f,\ldots,x_nf)=-gf^{p+2}+gf^{p+2}
=0.
\]
Thus, the morphism $\varphi$ induces $\varphi \colon X \to Y$.

Set
\[
U \defeq \setP^n\setminus V(f).
\]
Since \(V(f,g)\subseteq V(f)\), the blow-up morphism induces an
isomorphism
\[
\pi^{-1}(U)\simeq U.
\]
One can show that
\[
U\xrightarrow{\ \sim\ }Y_f,
\qquad
Y_f \defeq Y\cap\{f(Z_0,\ldots,Z_n)\neq0\}.
\]
Consequently, \(\varphi\colon X\to Y\) is surjective birational.

We verify that \(Y\) is smooth away from \(P \defeq (0\colon \cdots \colon 0 \colon 1)\).
We take a closed point \(Q \defeq [z_0:\cdots:z_n:t]\neq P\) of $Y$. 
Differentiating the defining equation with respect to \(T\)
gives
\[
f(z_0,\ldots,z_n)=0.
\]
The defining equation then gives \(g(z_0,\ldots,z_n)=0\), so $[z_0:\ldots:z_n] \in V(f,g)$.
The remaining partial derivatives give $(T\,df+dg)(Q)=0$.
This contradicts the transversality of \(D\) and \(B\).
Thus
\[
\operatorname{Sing}(Y)=\{P\}.
\]
Since \(Y\) is a hypersurface and its singular locus has codimension at least two, \(Y\) is normal.
In particular, if $X$ is quasi-$F$-split, then so is $Y$.

Thus, it suffices to show that $\cO_{Y,P}$ is not quasi-$F$-split.
On the affine chart \(T\neq0\), the local equation of \(Y\) at \(P\) is
\[
h \defeq f+g
\in
k[Z_0,\ldots,Z_n].
\]
By construction,
\[
f=\sum_{i=0}^n Z_i^{p+1}, \quad g=\sum_{i=0}^n Z_i^pq_i
\in
\mfrakm^{[p]}.
\]
Therefore $h\in\mfrakm^{[p]}$.
Since \(p>2\), we have $h^{p-2}\in\mfrakm^{[p]}$.
It follows from the Fedder-type criterion \cite{KTY22}*{Corollary~4.19} that
\[
\mcalO_{Y,P}
\simeq
\left(
k[Z_0,\ldots,Z_n]/(h)
\right)_{\mfrakm}
\]
is not quasi-\(F\)-split.

Consequently,
\[
X=\operatorname{Bl}_{V(f,g)}\setP^{n}
\]
is HT-split but not quasi-\(F\)-split.
\end{example}

\subsection{HT-splitting for locally complete intersection rings}
Complete intersection rings provide a basic class of affine HT-split schemes, including singular ones.
We first establish the following general result for locally complete intersection rings.
In the next subsection, we apply it to homogeneous complete intersections in Cox rings.

\begin{proposition}\label{lci-to-HT-split}
Let $R$ be a Noetherian $\F_p$-algebra such that $R$ is locally a complete intersection.
Then $\Spec(R)$ is HT-split.
\end{proposition}

\begin{proof}
Since $R$ is a local complete intersection $\F_p$-algebra, it is quasisyntomic.
By \cite{bhatt2019Topologicala}*{Lemma~4.28}, there exists a quasisyntomic cover
\[
R\longrightarrow S
\]
with $S$ quasiregular semiperfectoid. 
By \cite{bhatt2022Prismsa}*{Proposition~7.10}, the absolute prismatic cohomology \(A \defeq \prism_S\) of $S$ is concentrated in degree zero, and
\[
S\longrightarrow \overline{\prism}_S=A/pA
\]
is faithfully flat.
This shows that the composition \(R \to  A/pA\) is faithfully flat and hence is ind-split in \(\mcalD_{\qcoh}(\Spec(R))\).

We construct a section
\[
\rho \colon \Spec(A/pA) \to \Spec(A/pA)^{\dR}
\]
of $\nu_{A/pA}$ as follows:
Since $A$ is $\delta$-ring, we have the homomorphism
\[
s_A \colon A \to W(A)
\]
such that $\res \circ s_A=\id_A$ by the right adjointness of the Witt functor.
Put
\[
\var{s}_A\colon A/pA \xrightarrow{s_A/^Lp} W(A)/^L p \to W(A/pA)/^Lp,
\]
where the second morphism comes from the natural quotient morphism \(A \to A/pA\).
We define a morphism $\rho$ corresponding to
\[
\var{s}_A \in \Map_{\CAlg^{an}_{\F_p}}(A/pA,W(A/pA)/^Lp) =\Spec(A/pA)^{\dR}(A/pA).
\]

Combining the morphism
\[
\Spec(A/pA)^{\dR} \to \Spec(R)^{\dR},
\]
we obtain the homomorphism
\[
R \to \OdR{\Spec(R)} \to A/pA.
\]
Since the composition is ind-split, so is $R \to \OdR{\Spec(R)}$.
Therefore, \(\Spec(R)\) is HT-split.
\end{proof}

\begin{remark}
By the proof of \Cref{lci-to-HT-split}, for a quasisyntomic \(\setF_p\)-algebra \(R\), the affine scheme \(\Spec(R)\) is HT-split.
\end{remark}

\subsection{Complete intersections in toric varieties}
Let \(k\) be a perfect field of characteristic \(p>0\).
Let \(X=X_\Sigma\) be a toric variety over \(k\).
We write
\[
S_X
=
\mathrm{Cox}(X)
=
k[x_\rho\mid \rho\in\Sigma(1)]
\]
for its Cox ring, and let
\[
B_X
 \defeq 
\left(
x^{\widehat{\sigma}}
\mid
\sigma\in\Sigma
\right)
\subset
S_X
\]
be the irrelevant ideal. 
Set
\[
U_X
 \defeq 
\Spec(S_X)\setminus V(B_X).
\]
By the Cox quotient construction
\cite{Cox-HomogeneousCoordinateRing}*{Theorem~2.1}, there is a good quotient
\[
q_X\colon U_X \to  X
\]
by the diagonalizable group scheme $G_X \defeq \Spec(k[\Cl(X)])$.

\begin{proposition}\label{complete-intersection}
Let $f_1,\ldots,f_r\in S_X$ be a homogeneous regular sequence, and set
\[
A \defeq S_X/(f_1,\ldots,f_r).
\]
Let
\[
\widehat Z
 \defeq 
V(f_1,\ldots,f_r)\cap U_X
\subset
\operatorname{Spec}(A)
\]
and $Z\defeq\widehat Z/G_X\subseteq X$.

Then \(Z\) is HT-split.
In particular, if \(Z\) is smooth, then there is an isomorphism
\[
F_{*}\Omega^\bullet_{Z/k}
\simeq
\bigoplus_{i=0}^{\dim Z}
\Omega^i_{Z/k}[-i]
\]
in \(\mcalD_{\qcoh}(Z)\).
Consequently, the Hodge-to-de Rham spectral sequence of \(Z\)
degenerates at \(E_1\), and \(Z\) satisfies the
Akizuki--Nakano-type vanishing theorem.
\end{proposition}

\begin{proof}
By \Cref{lci-to-HT-split}, the affine scheme $\Spec(A)$ is HT-split.
Thus, the open subset $\widehat{Z}$ is HT-split.
Since $G_X$ is linearly reductive, the scheme $Z$ is HT-split by \Cref{finite-cover}, as desired.
\end{proof}

\begin{example}[Godeaux--Serre varieties]\label{Godeaux-Serre-quotients}
Let \(G\) be a finite flat group scheme over \(k\).
By the Godeaux--Serre construction, one can choose a complete
intersection \(X\) in a projective space endowed with a free
\(G\)-action such that the quotient
\[
q\colon X \to  Y \defeq X/G
\]
is smooth and projective; the quotient $Y$ is called a \emph{Godeaux--Serre variety}; see
\cite{KothariDeRhamJumps}*{Lemma~3.4}.

By \Cref{complete-intersection} and \Cref{finite-cover}, the Godeaux--Serre variety $Y$ is HT-split if $G$ is linearly reductive.
Since $Y$ is smooth and HT-split, the Hodge-to-de Rham spectral sequence degenerates at \(E_1\), and the Akizuki--Nakano vanishing holds on \(Y\).

A particularly interesting case is $G=\mu_p$.
Since \(\mu_p\) is diagonalizable, it is linearly reductive.
Antieau, Bhatt, and Mathew apply the Godeaux--Serre construction to $\mu_p$ and obtain a smooth projective \(2p\)-dimensional variety $Y$ such that the Hochschild--Kostant--Rosenberg spectral sequence does not degenerate. 
More precisely, the differential
\[
d_p\colon
H^0(Y,\Omega^1_{Y/k})
 \to 
H^p(Y,\Omega^p_{Y/k})
\]
is nonzero; see
\cite{AntieauBhattMathewHKR}*{Theorem~1.1 and \S 6}.

On the other hand, \Cref{complete-intersection} and \Cref{finite-cover} imply that \(Y\) is HT-split. 
Hence its Hodge-to-de Rham spectral sequence degenerates at \(E_1\), and it satisfies Akizuki--Nakano vanishing, although its HKR spectral sequence does not degenerate.

Moreover, \cite{AntieauBhattMathewHKR}*{Theorem~4.6 and Theorem~1.2} implies that
\[
H^2_{\crys}(Y/W(k))
\]
contains \(p\)-torsion.

Therefore, \(Y\) is an HT-split smooth projective variety whose Hodge-to-de Rham spectral sequence degenerates at \(E_1\), while its crystalline cohomology is not torsion-free.
\end{example}

\subsection{HT-splitting for fibrations}

\begin{proposition}\label{prop:log-HT-split-fibration-over-liftable-base}
Let \(k\) be a perfect field of characteristic \(p>0\), and let
\[
\pi\colon Y \to  X
\]
be a smooth projective morphism between smooth projective \(k\)-schemes.
Let \(D\) be a simple normal crossings divisor on \(X\), and let \(E\)
be a relative simple normal crossings divisor on \(Y\) over \(X\).
Set $B \defeq E+\pi^*D$ and
\[
d \defeq \dim X,\qquad
n \defeq \dim(Y/X),\qquad
N \defeq \dim Y=d+n.
\]
Suppose that the following conditions hold:
\begin{enumerate}
\item the pair \((Y,B)\) is \(W_2(k)\)-liftable;
\item \(d\leq p\);
\item for every \(0\leq b\leq n\), one has
\[
R^q\pi_*
\left(
\Omega^b_{Y/X}(\log E)\otimes\omega_{Y/X}
\right)=0
\qquad(q<n).
\]
\end{enumerate}
Then the pair \((Y,B)\) is log HT-split.
\end{proposition}

\begin{proof}
Since the pair \((Y,B)\) is \(W_2(k)\)-liftable, the logarithmic
Deligne--Illusie decomposition gives a splitting
\[
\mcalO_Y
 \to 
\tau_{\leq p-1}
F_*\Omega_Y^\bullet(\log B).
\]
By the truncation triangle
\[
\tau_{\leq p-1}
F_*\Omega_Y^\bullet(\log B)
 \to 
F_*\Omega_Y^\bullet(\log B)
 \to 
\tau_{\geq p}
F_*\Omega_Y^\bullet(\log B)
\xrightarrow{+1}
\]
and \Cref{chara-injectivity}, it suffices to show that
\[
H^{N-1}
\left(
Y,
\tau_{\geq p}
F_*\Omega_Y^\bullet(\log B)
\otimes\omega_Y
\right)=0.
\]
By the logarithmic Cartier isomorphism,
\[
\mcalH^j
\left(
\tau_{\geq p}
F_*\Omega_Y^\bullet(\log B)
\right)
\simeq
\Omega_Y^j(\log B)
\qquad(j\geq p).
\]
Thus, it suffices to show that
\[
H^i
\left(
Y,
\Omega_Y^j(\log B)\otimes\omega_Y
\right)=0 \quad (i+j=N-1, j\geq p).
\]

The logarithmic cotangent sequence
\[
0 \to 
\pi^*\Omega_X^1(\log D) \otimes \omega_Y
 \to 
\Omega_Y^1(\log B) \otimes \omega_Y
 \to 
\Omega_{Y/X}^1(\log E) \otimes \omega_Y
 \to 0
\]
and $\omega_Y
\simeq
\omega_{Y/X}\otimes\pi^*\omega_X$ 
induce a filtration on \(\Omega_Y^j(\log B) \otimes \omega_Y \) whose associated graded
pieces are
\[
\pi^*
\left(
\Omega_X^a(\log D)\otimes\omega_X
\right)
\otimes
\Omega_{Y/X}^{j-a}(\log E)
\otimes
\omega_{Y/X}
\]
by \cite{stacks-project}*{Tag 0FIC}
By the projection formula and assumption~(3), their higher direct images
under \(\pi\) vanish in degrees \(q<n\). The Leray spectral sequence
therefore gives
\[
H^i
\left(
Y,
\pi^*
\left(
\Omega_X^a(\log D)\otimes\omega_X
\right)
\otimes
\Omega_{Y/X}^{j-a}(\log E)
\otimes
\omega_{Y/X}
\right)=0
\qquad(i<n).
\]
Using the filtration, we obtain
\[
H^i
\left(
Y,
\Omega_Y^j(\log B)\otimes\omega_Y
\right)=0
\qquad(i<n).
\]

Now suppose that
\[
i+j=N-1
\qquad\text{and}\qquad
j\geq p.
\]
Since \(p\geq d\), we have \(j\geq d\), and hence
\[
i
=
d+n-1-j
\leq
n-1<n.
\]
It follows that
\[
H^i
\left(
Y,
\Omega_Y^j(\log B)\otimes\omega_Y
\right)=0,
\]
as desired.
\end{proof}

\begin{corollary}\label{W_2-liftable-projective-bundle}
Let $k$ be a perfect field of characteristic $p$.
Let $X$ be a smooth projective variety over $k$ of dimension $d \leq p$ and $E$ be a vector bundle on $X$.
We assume that there exists a \(W_2\)-lift \(\widetilde X\) of \(X\) such that \(E\) lifts to a vector bundle on \(\widetilde X\).
Then $\P_X(E)$ is HT-split.
\end{corollary}

\begin{proof}
It follows from \Cref{prop:log-HT-split-fibration-over-liftable-base}.
\end{proof}

\begin{proposition}\label{prop:HT-split-fibration-over-curve}
Let \(k\) be a perfect field of characteristic \(p>0\), let \(C\) be a
smooth projective curve over \(k\), and let
\[
\pi\colon Y \to  C
\]
be a smooth projective morphism of relative dimension \(n\).
Assume that
\[
R^{q}\pi_*
\left(
\Omega^b_{Y/C}\otimes\omega_{Y/C}
\right)=0
\]
for every \(0\leq b\leq n\) and every \(q<n\).
Then \(Y\) is HT-split.
\end{proposition}

\begin{proof}
By \Cref{prop:log-HT-split-fibration-over-liftable-base}, it suffices to show that $Y$ is $W_2$-liftable.
In particular, it suffices to show that
\[
H^2(Y,T_Y) \simeq H^{n-1}(\Omega^1_Y \otimes \omega_Y)=0.
\]
Combining the cotangent sequence
\[
0 \to 
\pi^*\Omega_C^1
 \to 
\Omega_Y^1
 \to 
\Omega_{Y/C}^1
 \to 0
\]
and
\[
\omega_Y
\simeq
\omega_{Y/C}\otimes\pi^*\omega_C,
\]
the desired vanishing follows from
\[
R^{a}\pi_*\left(
\Omega_{Y/C}^1\otimes\omega_{Y/C}
\otimes\pi^*\omega_C
\right)=R^a\pi_*\left(
\omega_{Y/C}
\otimes
\pi^*(\Omega_C^1\otimes\omega_C)
\right)=0
\]
for $a \leq n-1$.
\end{proof}

\begin{corollary}\label{cor:HT-split-fiberwise-vanishing-over-curve}
Let \(k\) be a perfect field of characteristic \(p>0\), let \(C\) be a smooth projective curve over \(k\), and let
\[
\pi\colon Y \to  C
\]
be a smooth projective morphism of relative dimension \(n\).
Assume that, for every geometric point $\overline{c} \to  C$, one has
\[
H^i\left(
Y_{\overline{c}},
\Omega^j_{Y_{\overline{c}}}
\otimes\omega_{Y_{\overline{c}}}
\right)=0
\]
for every  \(i<n\) and \(0\leq j\leq n\).
Then \(Y\) is HT-split.
\end{corollary}

\begin{proof}
It follows from \Cref{prop:HT-split-fibration-over-curve} and the cohomology and base change theorem.
\end{proof}

\begin{corollary}\label{projective-bundle-over-smooth-curve}
Let $k$ be a perfect field of characteristic $p$.
Let $C$ be a smooth projective curve over $k$.
Then every projective bundle over $C$ is HT-split.
\end{corollary}

\begin{proof}
It follows from \Cref{cor:HT-split-fiberwise-vanishing-over-curve}.
\end{proof}

\begin{example}[Symmetric product of smooth curves]\label{symmetric-product-HT-split}
Let \(k\) be a perfect field of characteristic \(p>0\),
and let \(C\) be a smooth projective curve of genus \(g\).
Let $m$ be a positive integer.
If $m \geq 2g-1$ and $p \geq g$, then the \(m\)-th symmetric product $C^{(m)} \defeq \operatorname{Sym}^m(C)$ of
\(C\) is HT-split.
In particular, Akizuki--Nakano vanishing holds on $\operatorname{Sym}^m(C)$ and the Hodge-to-de Rham spectral sequence degenerates at the $E_1$-page.

Indeed, we take a smooth lifting
\[
\widetilde C \to \Spec W_2(k)
\]
of \(C\). 
The Abel--Jacobi morphism admits a lifting
\[
\widetilde a_m\colon
\widetilde C^{(m)}
 \to 
\operatorname{Pic}^m_{\widetilde C/W_2(k)}.
\]
In particular, \(C^{(m)}\) is \(W_2(k)\)-liftable.

For every \(L\in\operatorname{Pic}^m(C)\), the fiber of
\[
a_m\colon C^{(m)} \to \operatorname{Pic}^m(C)
\]
over \(L\) is the complete linear system
\[
|L|=\setP H^0(C,L).
\]
Since \(m\geq2g-1\), one has $H^1(C,L)=0$ and hence
\[
h^0(C,L)=m+1-g.
\]
Therefore every geometric fiber of \(a_m\) is isomorphic to $\setP^{m-g}$.
Since
\[
\dim\operatorname{Pic}^m(C)=g\leq p,
\]
\Cref{prop:log-HT-split-fibration-over-liftable-base} shows that \(C^{(m)}\) is HT-split.
\end{example}

\begin{remark}[Comparison with the Lauritzen--Rao examples]
\label{rem:Lauritzen-Rao-comparison}
The preceding result stands in contrast to the projective bundles constructed
by Lauritzen and Rao in \cite{LauritzenRao}.

Let $V$ be a vector space of dimension $n+1$, where
\[
n\geq 3
\qquad\text{and}\qquad
p\geq n-1.
\]
Let
\[
Y
=
\left\{
(x,H)\in\setP(V)\times\setP(V^\vee)
\ \middle|\
x\in H
\right\}
\]
be the incidence variety parametrizing pairs consisting of a point and a
hyperplane containing it in $\setP(V)$. Lauritzen and Rao construct a
natural vector bundle $\mcalG'$ of rank $n-1$ on $Y$ and consider the
projective bundle
\[
X_{\mathrm{LR}}
 \defeq 
\setP_Y\left(F_Y^*\mcalG'\right).
\]
They prove that $X_{\mathrm{LR}}$ does not satisfy Kodaira vanishing.
Consequently, $X_{\mathrm{LR}}$ is not HT-split.

This example shows that assumption (1) in
\Cref{prop:log-HT-split-fibration-over-liftable-base} is essential and cannot
be omitted, even in the range $p\geq n-1$.
\end{remark}

\section{HT-splitting for Fano varieties} \label{SectionFano}
For Fano varieties, HT-splitting admits a particularly simple
cohomological characterization in terms of Akizuki--Nakano
vanishing. We use this criterion together with the permanence results
of the previous section to study Fano threefolds, del Pezzo
varieties, and Casagrande--Druel Fano fourfolds.

In this section, $k$ is a perfect field of characteristic $p>0$.

\subsection{Fano threefolds and del Pezzo varieties}

\begin{proposition}\label{Fano-chara}
Let $X$ be a smooth proper $k$-scheme of dimension $d$.
We assume $\omega_X$ is anti-ample.
Then the following conditions are equivalent:
\begin{enumerate}
\item $X$ is HT-split.
\item We have
\[
H^{i}(X,\Omega_X^j \otimes \omega_X)=0 \quad (i+j=d-1,j>0).
\]
\item Akizuki--Nakano-type vanishing holds on $X$.
\end{enumerate}
\end{proposition}

\begin{proof}
The implication (1) $\Rightarrow$ (3) follows from \Cref{vanishing}.
Since $\omega_X$ is anti-ample, (3) implies (2).

Finally, we prove the implication (2) $\Rightarrow$ (1).
We note that
\[
\cofib(\cO_X \xrightarrow{F^{\dR}_X} F_*\Omega_X^{\bullet})=\tau_{\geq 1}F_*\Omega_X^{\bullet}
\]
By \Cref{chara-injectivity}, it suffices to show that
\[
H^{d-1}(X,\tau_{\geq 1}F_*\Omega_X^{\bullet} \otimes \omega_X)=0.
\]
By the spectral sequence, it suffices to show that
\[
H^i(X,\mcalH^j(F_*\Omega_X^{\bullet}) \otimes \omega_X) \simeq H^i(X,\Omega_X^j \otimes \omega_X)=0 \quad (i+j=d-1,j >0).
\]
It follows from (2), as desired.
\end{proof}

\begin{proposition}\label{example-Fano-varieties}
The following varieties are HT-split.
\begin{enumerate}
\item Del Pezzo surfaces over $k$ with rational double points.
\item Fano threefolds over $k$.
\end{enumerate}
\end{proposition}

\begin{proof}
By \Cref{finite-cover} (1), we may assume that $k$ is an algebraically closed field.
First, let $X$ be a del Pezzo surface over $k$ with rational double points.
Let $\pi \colon Y \to X$ be a minimal resolution.
Then $Y$ is $W_2$-liftable.
Indeed, it is either \(\setP^1\times\setP^1\), \(\setF_2\), or a successive blow-up of \(\setP^2\) at most eight points \cite{MartinStadlmayr}*{Lemma~2.7}.
The assertion then follows from the compatibility of liftability with blowing up closed points \cite{LiedtkeSatriano}*{Propositions~2.2 and~3.1}.

In particular, $Y$ is HT-split by \Cref{W_2-lift-to-HT-split}.
Since $X$ has rational singularities, $X$ is also HT-split by \Cref{finite-cover}.

Next, let $X$ be a smooth Fano threefold.
Then the Akizuki--Nakano-type vanishing holds on $X$ by \cite{KawakamiTanakaFanoLiftabilityI}*{Theorem~B}.
By \Cref{Fano-chara}, $X$ is HT-split.
\end{proof}

\begin{theorem}\label{del-Pezzo-variety}
Let $k$ be a perfect field and $X$ a del Pezzo variety over $k$.
Then for every simple normal crossings divisor $D$ on $X$, the pair $(X,D)$ is log HT-split.
In particular, we have
\[
F_*\Omega^{\bullet}_X(\log D) \simeq \bigoplus_{i=0}^{\dim X} \Omega_X^i(\log D)[-i]
\]
and Akizuki--Nakano-type vanishing for $(X,D)$ as in \Cref{vanishing}.
\end{theorem}

\begin{proof}
It follows from \cite{KawakamiTanakaWeakQuasiFSplitting}*{Theorem~A} and \Cref{qFs-to-HT}.
\end{proof}

\subsection{Casagrande--Druel Fano fourfolds}

In characteristic zero, the Casagrande--Druel construction
(\cite{CD15}*{Lemma~3.1(iii)}) produces a large class of
Fano fourfolds with Lefschetz defect two.
The case of Picard number three was classified in \cite{Sec23};
more precisely, every smooth Fano fourfold with Lefschetz defect two and
Picard number three arises from the Casagrande--Druel construction.
More recently, Casagrande--Druel Fano fourfolds with Lefschetz defect two
and Picard number at least four were completely classified in
\cite{Pas25}, resulting in 147 distinct families.

In this subsection, we consider the HT-splitting of Fano fourfolds obtained by the Casagrande--Druel construction in positive characteristic.
We now recall the Casagrande--Druel construction, following the formulation
in \cite{CheltsovEtAl2025}*{Definition~1.2 and Section~2}.

Let \(V\) be a smooth projective variety over $k$, let \(L\) be a line bundle on \(V\) such that $\omega_V \otimes L$ is anti-ample, and let $R\in |2L|$ be a smooth divisor. 
Set
\[
\pi\colon
Y \defeq \mbfP_V(\mcalO_V\oplus L)
 \to  V.
\]
We denote by $S_+,S_-\subset Y$ the two sections induced by the two quotient line bundles of \(\mcalO_V\oplus L\), and put $F \defeq \pi^{-1}(R)$.
The Casagrande--Druel variety associated with the triple \((V,L,R)\) is
\[
X \defeq \operatorname{Bl}_{S_+\cap F}Y.
\]

\begin{corollary}\label{cor:CD-HT-split}
We use the above notation.
If $V$ is quasi-$F$-split, then the Casagrande--Druel variety
\[
X=\operatorname{Bl}_{S_+\cap F}Y
\]
is HT-split.

In particular, if \(V\) is a Fano threefold, then \(X\) is HT-split. Consequently, Akizuki--Nakano-type vanishing holds on \(X\), and the Hodge-to-de Rham spectral sequence degenerates at the \(E_1\)-page.
\end{corollary}

\begin{proof}
First, we assume $V$ is quasi-$F$-split.
By \Cref{cor:qFs-split-projective-bundle}, the projective bundle $Y$ is quasi-$F$-split.
Thus, the pair $(Y,F+S_+)$ is log HT-split by \Cref{qFs-to-HT}.
The center
\[
S_+\cap F
\]
is a stratum of the SNC divisor \(S_++F\). 
Thus, $X$ is HT-split by \Cref{HT-ascent-stratum-blowup}.

Next, we assume that $V$ is a Fano threefold. 
Since $-(K_V+L)$ is ample, we have
\[
\rho(V)>1
\qquad\text{or}\qquad
r_V>1.
\]
By \cite{KT25II}*{Theorem~E}, the variety $V$ is quasi-$F$-split.
Thus, the result follows from the first assertion.
\end{proof}

\section{HT-splitting for varieties with trivial canonical bundle} \label{SectionCY}
When the canonical bundle is trivial, HT-splitting is closely related
to the degeneration of the Hodge-to-de Rham spectral sequence.
We apply this observation to K3 surfaces, abelian varieties, and
Cynk--Hulek Calabi--Yau varieties.

In this section, $k$ is a perfect field of characteristic $p>0$.

\begin{proposition}\label{chara-CY}
Let $X$ be a smooth proper $k$-scheme of dimension $d$.
We assume $\omega_X \simeq \cO_X$.
Then the following conditions are equivalent:
\begin{enumerate}
\item $X$ is HT-split.
\item The homomorphism
\[
H^d(X,\cO_X) \xrightarrow{F^{\dR}_X} H^d(X,F_*\Omega_X^{\bullet})
\]
is injective.
\item The Hodge-to-de Rham spectral sequence is $E_1$-degenerate.
\end{enumerate}
\end{proposition}

\begin{proof}
The equivalence (1) $\Leftrightarrow$ (2) follows from \Cref{chara-injectivity}.
The implication (1) $\Rightarrow$ (3) follows from \Cref{criterion-decomp-log}.

Next, we assume (3).
Consider the conjugate spectral sequence
\[
{}^{\mathrm{conj}}E_2^{a,b}
=
H^a\left(
X,\mcalH^b\left(F_*\Omega_X^\bullet\right)
\right)
 \to 
H^{a+b}\left(X,F_*\Omega_X^\bullet\right).
\]
By the Cartier isomorphism and the assumption (3), we have
\[
\dim_k H^m\left(X,F_*\Omega_X^\bullet\right)
=
\sum_{a+b=m}
\dim_k {}^{\mathrm{conj}}E_2^{a,b}.
\]
Since the conjugate spectral sequence converges to
$H^{a+b}(X,F_*\Omega_X^\bullet)$, we also have
\[
\dim_k H^m\left(X,F_*\Omega_X^\bullet\right)
=
\sum_{a+b=m}
\dim_k {}^{\mathrm{conj}}E_\infty^{a,b}.
\]
Each term ${}^{\mathrm{conj}}E_\infty^{a,b}$ is a subquotient of
${}^{\mathrm{conj}}E_2^{a,b}$. The preceding equalities therefore imply
that
\[
{}^{\mathrm{conj}}E_2^{a,b}
=
{}^{\mathrm{conj}}E_\infty^{a,b}
\]
for all $a$ and $b$. Thus, the conjugate spectral sequence degenerates at
$E_2$.

Under the Cartier identification
\[
\mcalH^0\left(F_*\Omega_X^\bullet\right)
\simeq
\cO_X,
\]
the homomorphism induced by \(F_X^{\dR}\),
\[
H^d(X,\cO_X)
\xrightarrow{F_X^{\dR}}
H^d\left(X,F_*\Omega_X^\bullet\right),
\]
is the edge homomorphism
\[
{}^{\mathrm{conj}}E_2^{d,0}
\longrightarrow
H^d(X,F_*\Omega_X^\bullet).
\]
Thus, it is injective, as desired.
\end{proof}

\subsection{K3 surfaces and abelian varieties}

\begin{theorem}\label{CY-examples}
K3 surfaces and abelian varieties are HT-split.
\end{theorem}

\begin{proof}
By \cite{DeligneK3Lifting}*{Corollaire~1.2}, every K3 surface is $W_2$-liftable.
Thus, every K3 surface is HT-split by \Cref{W_2-lift-to-HT-split}.

Every abelian variety is HT-split by \cite{DI87}*{Remarque 2.6(iv)} (see also \Cref{Decomposable-case}).
\end{proof}

\begin{proposition}\label{lemma:K3-involution-pair-lift}
We assume $p>2$.
Let \(S\) be a K3 surface over \(k\), and let $\sigma\colon S \to  S$ be a non-symplectic involution.
Put 
\[
R \defeq \operatorname{Fix}(\sigma).
\]
Then the pair \((S,R)\) admits a lifting to \(W_2(k)\). Consequently, \((S,R)\) is log HT-split.
\end{proposition}

\begin{proof}
Since \(\sigma\) has order two and \(p\neq2\), it is a tame automorphism.
By \cite{JangLifting}*{Theorem~3.3}, there exist a smooth projective
lifting
\[
\mathscr S \to \operatorname{Spec}W(k)
\]
of \(S\) and an involution
\[
\widetilde\sigma\colon\mathscr S \to \mathscr S
\]
lifting \(\sigma\).

Let
\[
\mathscr R \defeq \operatorname{Fix}(\widetilde\sigma)
\subset\mathscr S
\]
be the scheme-theoretic fixed locus. 
Formation of the fixed locus commutes with base change, so
\[
\mathscr R\times_{W(k)}k=R.
\]
Since \(2\) is invertible in \(W(k)\), the fixed locus of the tame involution \(\widetilde\sigma\) on the smooth scheme \(\mathscr S\) is smooth over \(W(k)\).
Thus \((S,R)\) is \(W_2\)-liftable.
By \Cref{W_2-lift-to-HT-split}, \((S,R)\) is log HT-split.
\end{proof}

\subsection{Cynk--Hulek Calabi--Yau varieties}
Cynk--Hulek varieties form a class of higher-dimensional
Calabi--Yau varieties of Kummer type.  They are obtained as crepant
resolutions of finite quotients of products of elliptic curves and
were originally introduced in order to construct explicit
higher-dimensional modular Calabi--Yau manifolds
\cite{CynkHulek}.
Their quotient description makes both their middle-dimensional
motive and their arithmetic accessible.

Let \(k\) be an algebraically closed field of characteristic \(p>2\). 
Let
\[
E_1,\ldots,E_n
\]
be elliptic curves over \(k\). For each \(i\), let
\[
\iota_i=[-1]_{E_i}
\]
and put
\[
B_i \defeq E_i[2].
\]
Since \(p\neq2\), the scheme \(B_i\) is a reduced étale divisor of
degree four on \(E_i\).

Set
\[
A_n \defeq E_1\times\cdots\times E_n
\]
and
\[
G_n \defeq 
\left\{
(\epsilon_1,\ldots,\epsilon_n)\in(\mbfZ/2\mbfZ)^n
\ \middle|\
\epsilon_1+\cdots+\epsilon_n=0
\right\}
\simeq(\mbfZ/2\mbfZ)^{n-1}.
\]
The group \(G_n\) acts on \(A_n\) coordinate-wise, where the
nontrivial element in the \(i\)-th factor acts by \(\iota_i\).

By the Cynk--Hulek construction \cite{CynkHulek}*{Propositions~2.1 and~2.2}, the quotient $A_n/G_n$ admits a smooth crepant resolution
\[
X_n \to  A_n/G_n.
\]
The resulting variety \(X_n\) is a smooth Calabi--Yau \(n\)-fold.

\begin{proposition}\label{prop:Cynk-Hulek-HT}
The Cynk--Hulek Calabi--Yau variety \(X_n\) is HT-split.
\end{proposition}

\begin{proof}
Consider the divisor
\[
\Delta_n
 \defeq 
\sum_{i=1}^n
E_1\times\cdots\times E_{i-1}
\times B_i
\times E_{i+1}\times\cdots\times E_n
\]
on \(A_n\). It is a simple normal crossings divisor.

Since every pair \((E_i,B_i)\) is log HT-split by \Cref{W_2-lift-to-HT-split}, the product pair
\[
(A_n,\Delta_n)
\]
is log HT-split.

The fixed loci occurring in the Cynk--Hulek Kummer construction are disjoint unions of subvarieties of the form
\[
\prod_{i\in I}B_i
\times
\prod_{i\notin I}E_i
\]
and their strict transforms. 
In particular, the equivariant modification
\[
\widetilde A_n \to  A_n
\]
appearing in the standard resolution is obtained by a sequence of blow-ups along strata of the reduced transforms of \(\Delta_n\) (see \cite{CynkHulek}*{Propositions~2.1 and~2.2 and the proof of Lemma~2.4}).
Then \(\widetilde A_n\) is HT-split by \Cref{HT-ascent-stratum-blowup}.

The group \(G_n\) is linearly reductive because \(p\neq2\), and the
Cynk--Hulek model is the smooth quotient
\[
X_n\simeq\widetilde A_n/G_n.
\]
The quotient descent theorem therefore implies that \(X_n\) is
HT-split by \Cref{finite-cover}.
\end{proof}

\section{Applications to vanishing in mixed characteristic} \label{SectionMixedCharacteristic}
The main goal of this section is to prove the Akizuki--Nakano-type vanishing theorem for smooth projective lim-perfectoid split schemes over \(W(k)\), and hence for globally \(+\)-regular schemes, stated in \Cref{intro:ANV-G+R}.
We construct a natural morphism from the closed fiber of the perfectization of a \(p\)-adic formal scheme to the de Rham stack of its closed fiber (\Cref{PerfectizationToDeRham}).
This implies that the closed fiber of a lim-perfectoid split formal scheme is HT-split (\Cref{LimPerfectoidSplitImpliesHTSplit}), so that the positive-characteristic vanishing results developed above apply.

\subsection{Comparison between perfectizations and de Rham stacks}

\begin{construction}
    Let \(\Perfd\) be the category of perfectoid \(\setZ_p\)-algebras.
    Consider the functors
\[
\mcalF,\mcalG \colon \Perfd \longrightarrow \CAlg^{an}_{\F_p}
\]
defined by
\[
\mcalF(R) \defeq R/^Lp, \qquad \mcalG(R) \defeq W(R/^Lp)/^Lp.
\]
The restriction morphisms of Witt vectors induce a natural transformation
\[
\rho\colon \mcalG\longrightarrow\mcalF.
\]
\end{construction}

\begin{proposition}\label{ConstructionRetractionTorsfree}
Let $\Perfd_{\tf}$ denote the category of $p$-torsion-free perfectoid rings.
Write the restriction of functors \(\mcalF\) and \(\mcalG\) to the full subcategory \(\Perfd_{\tf} \subset \Perfd\) as
\begin{equation*}
    \mcalF_{\tf}, \mcalG_{\tf} \colon \Perfd_{\tf} \to \CAlg^{an}_{\setF_p}.
\end{equation*}
Then there exists a natural transformation
\[
s\colon \mcalF_{\tf}\longrightarrow\mcalG_{\tf}
\]
together with an equivalence of natural transformations
\[
\rho\circ s
\simeq
\id_{\mcalF_{\tf}}.
\]
\end{proposition}

\begin{proof}
Let $\PerfPrism_{\tf}$ denote the category of perfect prisms $(A,I)$ such that $A/I$ is $p$-torsion-free.
By the equivalence between perfect prisms and perfectoid rings, the quotient functor
\[
q\colon
\PerfPrism_{\tf}
\longrightarrow
\Perfd_{\tf},
\qquad
(A,I)\longmapsto A/I,
\]
is an equivalence of categories.
It is therefore enough to construct, on $N(\PerfPrism_{\tf})$, a natural transformation
\[
\widetilde{s}\colon
q^*\mcalF_{\tf}\longrightarrow q^*\mcalG_{\tf}
\]
together with an equivalence
\[
(q^*\rho)\circ\widetilde{s}
\simeq
\id_{q^*\mcalF_{\tf}}.
\]

For $(A,I)\in\PerfPrism_{\tf}$, put
\begin{equation*}
    R_A \defeq A/I, \qquad J_A \defeq I \otimes_A A/pA.
\end{equation*}
Since $R_A$ is $p$-torsion-free, the natural morphism \(J_A \to A/pA\) is injective.
We denote by
\[
\pi_A \colon A/pA\longrightarrow R_A/pR_A \cong (A/pA)/J_A
\]
the quotient morphism.

We first work in the ordinary category \(\CRing\) of commutative rings.
Let \([J_A] \subseteq W(A/pA)\) denote the ideal of \(W(A/pA)\) generated by the Teichm\"uller lifts $[x]$ for $x\in J_A$.
Define functors
\[
\widetilde{\mcalF_{\tf}},\widetilde{\mcalG_{\tf}} \colon \PerfPrism_{\tf} \longrightarrow \CRing
\]
by
\begin{equation} \label{EquivAlpha}
    \widetilde{\mcalF_{\tf}}(A, I) \defeq W(A/pA)/[J_A] \qquad \text{and} \qquad \widetilde{\mcalG_{\tf}}(A, I) \defeq W(R_A/pR_A)
\end{equation}
since all operations in this construction are functorial.

Since
\[
W(\pi_A)([x])=[\pi_A(x)]=0 \qquad (x\in J_A),
\]
the morphism
\[
W(\pi_A)\colon W(A/pA)\longrightarrow W(R_A/pR_A)
\]
kills the ideal $[J_A]$.
It therefore induces a morphism of ordinary rings
\[
\sigma_A\colon \widetilde{\mcalF_{\tf}}(A,I) \longrightarrow \widetilde{\mcalG_{\tf}}(A,I).
\]
This defines a natural transformation
\begin{equation} \label{ConstSigmaTorsf}
    \sigma \colon \widetilde{\mcalF_{\tf}} \to \widetilde{\mcalG_{\tf}}
\end{equation}
in the ordinary category of functors from \(\PerfPrism_{\tf}\) to \(\CRing\).

We next compare the derived reductions modulo $p$ of these functors with $q^*\mcalF_{\tf}$ and $q^*\mcalG_{\tf}$.
We first claim that the commutative ring
\[
\widetilde{\mcalF_{\tf}}(A,I) = W(A/pA)/[J_A]
\]
is $p$-torsion-free.
To verify this, choose a generator $d$ of $I$.
Since $W(A/pA)/p \simeq A/pA$ is $d$-torsion free, the sequence $p,[d]$ is regular.
Thus, the sequence $[d],p$ is also regular.
Consequently, taking \(\pi_0\) gives an isomorphism
\begin{equation} \label{FTildeIsomModp}
    \widetilde{\mcalF_{\tf}}(A,I)/^Lp \xrightarrow{\simeq} \widetilde{\mcalF_{\tf}}(A,I)/p
\end{equation}
in \(\CAlg^{an}_{\setF_p}\).
The restriction morphism \(W(A/pA) \to A/pA\) induces a canonical isomorphism \(W(A/pA)/p \xrightarrow{\cong} A/pA\) and we obtain an equivalence
\begin{equation} \label{EquivAlphaDerived}
  \alpha_A\colon \widetilde{\mcalF_{\tf}}(A,I)/^Lp \xrightarrow{\simeq} (W(A/pA)/[J_A])/p \xrightarrow{\cong} ((A/pA)/J_A)/p \cong R_A/pR_A = q^*\mcalF_{\tf}(A,I)
\end{equation}
in \(\CAlg^{an}_{\setF_p}\).

Since all operations appearing in $\alpha_A$ are functorial, \(\alpha_A\) is therefore strictly compatible with morphisms of perfect prisms.
Thus the morphisms $\alpha_A$ assemble to a natural equivalence
\[
\alpha\colon \widetilde{\mcalF_{\tf}}/^Lp \xrightarrow{\sim} q^*\mcalF_{\tf}
\]
in \(\Fun(\PerfPrism_{\tf}, \CAlg^{an}_{\setF_p})\), where we use
\[
R_A/^Lp\simeq R_A/pR_A
\]
because $R_A$ is $p$-torsion-free.

On the other hand, by definition,
\[
\widetilde{\mcalG_{\tf}}(A,I)/^Lp = W(R_A/pR_A)/^Lp = (q^*\mcalG_{\tf})(A,I).
\]
Hence there is a canonical identification of functors
\[
\widetilde{\mcalG_{\tf}}/^Lp = q^*\mcalG_{\tf}.
\]

Applying derived reduction modulo $p$ to the natural transformation
$\sigma$ gives a natural transformation
\[
\sigma/^Lp\colon \widetilde{\mcalF_{\tf}}/^Lp \longrightarrow q^*\mcalG_{\tf}
\]
in \(\Fun(\PerfPrism_{\tf}, \CAlg^{an})\).
Then we can define a natural transformation
\[
\widetilde{s}  \defeq  (\sigma/^Lp)\circ\alpha^{-1} \colon q^*\mcalF_{\tf} \longrightarrow q^*\mcalG_{\tf}
\]
in \(\Fun(\PerfPrism_{\tf}, \CAlg^{an})\).

It remains to verify that $\widetilde{s}$ is a section of
$q^*\rho$.
Let
\[
\res_A\colon
W(R_A/pR_A)\longrightarrow R_A/pR_A
\]
denote the restriction morphism.
Since the composition
\begin{equation*}
    W(A/pA)/[J_A] \xrightarrow{\sigma_A} W(R_A/pR_A) \xrightarrow{\res} R_A/pR_A
\end{equation*}
induces the isomorphism \(\alpha_A\) after the factorization
\begin{equation*}
    (W(A/pA)/[J_A])/^L p \xrightarrow{\alpha_A} R_A/pR_A,
\end{equation*}
we have an equivalence of natural transformations
\begin{equation} \label{ConstructionRetractionTorsfreeEq}
  q^*\rho\circ(\sigma/^Lp) \simeq \alpha.
\end{equation}
Therefore
\begin{equation*}
    q^*\rho\circ\widetilde{s} = q^*\rho\circ(\sigma/^Lp)\circ\alpha^{-1} \simeq \alpha\circ\alpha^{-1} \simeq \id_{q^*\mcalF_{\tf}}.
\end{equation*}

Finally, using the equivalence
\[
q\colon\PerfPrism_{\tf}\xrightarrow{\sim}\Perfd_{\tf},
\]
we can pass \(\widetilde{s} \colon q^*\mcalF_{\tf} \to q^*\mcalG_{\tf}\) and the equivalence
\[
q^*\rho\circ\widetilde{s} \simeq \id_{q^*\mcalF_{\tf}}
\]
to the desired natural transformation \(s \colon \mcalF_{\tf} \to \mcalG_{\tf}\) and an equivalence \(\rho \circ s \simeq \id_{\mcalF_{\tf}}\).
\end{proof}

\begin{proposition} \label{ConstructionRetractionPerfect}
    Let \(\Perf\) be the category of perfect \(\setF_p\)-algebras.
    Write the restriction of functors \(\mcalF\) and \(\mcalG\) to the full subcategory \(\Perf \subset \Perfd\) as
    \begin{equation*}
        \mcalF_{\crys}, \mcalG_{\crys} \colon \Perf \to \CAlg^{an}_{\setF_p}.
    \end{equation*}
    Then there exists a natural transformation
    \begin{equation*}
        s \colon \mcalF_{\crys} \to \mcalG_{\crys}
    \end{equation*}
    together with an equivalence of natural transformations
    \begin{equation*}
        \rho \circ s \simeq \id_{\mcalF_{\crys}}.
    \end{equation*}
\end{proposition}

\begin{proof}
    Pick a perfect \(\setF_p\)-algebra \(T\).
    By the construction of \(\mcalF\) and \(\mcalG\), we have
    \begin{equation*}
      \mcalF_{\crys}(T) = T/^L p \quad \text{and} \quad \mcalG_{\crys}(T) = W(T/^L p)/^L p.
    \end{equation*}
    
    The quotient morphism \(\pi_T \colon T \to T/^L p\) in \(\CAlg^{an}_{\setF_p}\) gives a morphism
    \begin{equation*}
      W(T) \xrightarrow{W(\pi_T)} W(T/^L p)
    \end{equation*}
    in \(\CAlg^{an}_{\setZ_p}\) together with an equivalence
    \begin{equation} \label{EquivResPi}
      \res_{T/^Lp} \circ W(\pi_T) \simeq \pi_T \circ \res_T
    \end{equation}
    in \(\Map_{\CAlg^{an}_{\setF_p}}(W(T), T/^L p)\), where \(\res_T \colon W(T) \to T\) is the restriction morphism.

    Taking the derived modulo \(p\), we obtain a morphism
    \begin{equation} \label{CrystallineStructureMap}
      u_T \colon W(T)/^L p \xrightarrow{W(\pi_T)/^L p} W(T/^L p)/^L p = \mcalG_{\crys}(T).
    \end{equation}
    Since \(T\) is perfect, the restriction morphism \(\res_T \colon W(T) \to T\) induces an equivalence
    \begin{equation} \label{EquivPerfectWitt}
      W(T)/^L p \xrightarrow{\simeq} T.
    \end{equation}
    Since \(\mcalG_{\crys}(T)\) is an \(\setF_p\)-algebra, the above morphism uniquely factors through
    \begin{equation*}
        \mcalF_{\crys}(T) = T/^L p \to \mcalG_{\crys}(T)
    \end{equation*}
    in \(\CAlg^{an}_{\setF_p}\).
    These morphisms assemble to a natural transformation
    \begin{equation*}
      s \colon \mcalF_{\crys} \to \mcalG_{\crys}
    \end{equation*}
    in \(\Fun(\Perf, \CAlg^{an}_{\setF_p})\).

    Moreover, the equivalence \eqref{EquivResPi} is also functorial in \(T\) and induces an equivalence
    \begin{equation*}
      \rho \circ s \simeq \id_{\mcalF_{\crys}}
    \end{equation*}
    in \(\Fun(\Perf, \CAlg^{an}_{\setF_p})\).
\end{proof}

\begin{proposition} \label{lem:perfectoid-witt-reduction-compatibility}
  Let \(\mcalM\) be the category of morphisms \(R \to T\) in \(\Perfd\) such that \(R\) is \(p\)-torsion-free and \(T\) is an \(\setF_p\)-algebra.
  Let
  \begin{equation*}
    \src \colon\mcalM\to\Perfd_{\tf},\qquad \tgt\colon\mcalM\to\Perf
  \end{equation*}
  be the source and target functors.
  The functors \(\mcalF\) and \(\mcalG\) induce functors
  \begin{equation*}
    \mcalF', \mcalG' \colon \mcalM \to \Fun(\Delta^1, \CAlg^{an}_{\setF_p})
  \end{equation*}
  where the target is the \(\infty\)-category of morphisms in \(\CAlg^{an}_{\setF_p}\).
  Then there exists a natural transformation
  \begin{equation*}
    s' \colon \mcalF' \to \mcalG'
  \end{equation*}
  in \(\Fun(\mcalM, \Fun(\Delta^1, \CAlg^{an}_{\setF_p}))\) whose restrictions to the source and target are \(s^{\tf}\) and \(s^{\crys}\) respectively, together with an equivalence of natural transformations
  \begin{equation*}
      \rho' \circ s' \simeq \id_{\mcalF'},
  \end{equation*}
  where \(\rho'\) is the natural transformation induced from \(\rho \colon \mcalG \to \mcalF\).
\end{proposition}

\begin{proof}
  Define functors
  \begin{equation*}
    \mcalF_0,\mcalG_0,\mcalU,\mcalF_1,\mcalG_1\colon \mcalM\to\CAlg^{an}_{\setF_p}
  \end{equation*}
  by
  \begin{equation*}
    \begin{aligned}
      \mcalF_0(f\colon R\to T)&\defeq\mcalF_{\tf}(R),&
      \mcalG_0(f\colon R\to T)&\defeq\mcalG_{\tf}(R),&
      \mcalU(f\colon R\to T)&\defeq T,\\
      \mcalF_1(f\colon R\to T)&\defeq\mcalF_{\crys}(T),&
      \mcalG_1(f\colon R\to T)&\defeq\mcalG_{\crys}(T).
    \end{aligned}
  \end{equation*}
  We use the natural equivalence \(\mcalF_0(f)\simeq R/pR\) without further mention. Let
  \begin{equation*}
    s_0\colon\mcalF_0\to\mcalG_0,\qquad s_1\colon\mcalF_1\to\mcalG_1
  \end{equation*}
  be the natural transformations constructed in \Cref{ConstructionRetractionTorsfree,ConstructionRetractionPerfect}, respectively, and write
  \begin{equation*}
    \rho_0\colon\mcalG_0\to\mcalF_0,\qquad \rho_1\colon\mcalG_1\to\mcalF_1
  \end{equation*}
  for the restrictions of Witt restriction.
  Under the equivalence,
  \begin{equation*}
    \Fun(\mcalM, \Fun(\Delta^1, \CAlg^{an}_{\setF_p})) \simeq \Fun(\Delta^1, \Fun(\mcalM, \CAlg^{an}_{\setF_p})),
  \end{equation*}
  the functors \(\mcalF'\) and \(\mcalG'\) correspond to natural transformations
  \begin{equation} \label{DefQuotientNaturalTransformations}
    a \colon\mcalF_0\to\mcalF_1,\qquad b \colon\mcalG_0\to\mcalG_1
  \end{equation}
  in \(\Fun(\mcalM, \CAlg^{an}_{\setF_p})\), whose values at \(f \colon R \to T\) are
  \begin{equation*}
    f/^Lp \colon R/pR \to T/^L p, \qquad W(f/^Lp)/^Lp \colon W(R/pR)/^Lp \to W(T/^Lp)/^Lp,
  \end{equation*}
  respectively.

  For \(R\in\Perfd_{\tf}\), let \((A_R,I_R)\) be its functorially associated perfect prism and put
  \begin{equation*}
    J_R\defeq I_R\otimes_{A_R}A_R/pA_R \subseteq A_R/pA_R.
  \end{equation*}
  Since the perfect prism associated with \(T\in\Perf\) is \((W(T),(p))\), every object \(f\colon R\to T\) of \(\mcalM\) induces a morphism
  \begin{equation*}
    \widetilde f\colon(A_R,I_R)\to(W(T),(p))
  \end{equation*}
  of perfect prisms. Reducing modulo \(p\) gives a morphism \(A_R/pA_R\to T\) which kills \(J_R\), and hence a morphism of ordinary rings
  \begin{equation*}
    \lambda_f\colon W(A_R/pA_R)/[J_R]\to W(T).
  \end{equation*}
  Write \(\overline f\colon R/pR\to T\) for the reduction of \(f\). The construction \eqref{ConstSigmaTorsf} of \Cref{ConstructionRetractionTorsfree} gives a morphism
  \begin{equation*}
    \sigma_R\colon W(A_R/pA_R)/[J_R]\to W(R/pR)
  \end{equation*}
  of ordinary rings, and the diagram
  \begin{equation}
    \label{DiscreteCompatibilitySquare}
    \begin{tikzcd}
      W(A_R/pA_R)/[J_R] \arrow[r,"\sigma_R"] \arrow[d,"\lambda_f"'] & W(R/pR) \arrow[d,"W(\overline f)"] \\
      W(T) \arrow[r,equal] & W(T)
    \end{tikzcd}
  \end{equation}
  commutes strictly. These squares are functorial in commutative squares in \(\mcalM\), and therefore define functors
  \begin{equation*}
    \widetilde{\mcalH},\widetilde{\mcalK}\colon\mcalM\to\Fun(\Delta^1,\CRing)
  \end{equation*}
  and a natural transformation
  \begin{equation}
    \label{DiscreteCompatibilityTransformation}
    \widetilde\kappa\colon\widetilde{\mcalH}\to\widetilde{\mcalK},
  \end{equation}
  where
  \begin{align*}
    \widetilde{\mcalH}(f)&\defeq\left(W(A_R/pA_R)/[J_R]\xrightarrow{\lambda_f}W(T)\right),\\
    \widetilde{\mcalK}(f)&\defeq\left(W(R/pR)\xrightarrow{W(\overline f)}W(T)\right).
  \end{align*}
  Its source and target components are \(\sigma_R\) and \(\id_{W(T)}\), respectively.

  Apply derived reduction modulo \(p\) to \eqref{DiscreteCompatibilityTransformation}. The equivalences
  \begin{equation*}
    \alpha_R \colon\left(W(A_R/pA_R)/[J_R]\right)/^Lp\xrightarrow{\sim}R/pR
  \end{equation*}
  constructed in \eqref{EquivAlphaDerived} in \Cref{ConstructionRetractionTorsfree} identify the source components of \(\widetilde{\mcalH}/^Lp\) and \(\mcalF_0\).
  The equivalence \eqref{EquivPerfectWitt} gives a natural equivalence
  \begin{equation*}
    \beta \colon W(\mcalU)/^Lp\xrightarrow{\sim}\mcalU.
  \end{equation*}
  Let \(\beta^{-1}\) denote an inverse in the functor category.
  Define natural transformations
  \begin{equation*}
    \overline{\alpha}\colon\mcalF_0\to\mcalU,\qquad \gamma\colon\mcalG_0\to\mcalU
  \end{equation*}
  by
  \begin{equation*}
    \overline{\alpha}_f\defeq\overline f\colon R/pR\to T,\qquad \gamma_f \colon W(R/pR)/^L p \xrightarrow{W(\overline{f})/^L p} W(T)/^L p \xrightarrow{\beta_T} T
  \end{equation*}
  Transporting \(\widetilde\kappa/^Lp\) through these natural equivalences in the arrow functor category gives a natural transformation
  \begin{equation}
    \label{IntermediateArrowTransformation}
    \kappa_0\colon\mcalH\to\mcalK
  \end{equation}
  in \(\Fun(\mcalM,\Fun(\Delta^1,\CAlg^{an}_{\setF_p}))\), where
  \begin{equation*}
    \mcalH\defeq(\mcalF_0\xrightarrow{\overline{\alpha}}\mcalU),\qquad
    \mcalK\defeq(\mcalG_0\xrightarrow{\gamma}\mcalU).
  \end{equation*}
  Its source and target components are \(s_0\) and \(\id_{\mcalU}\), respectively.

  We next use the characteristic-\(p\) construction. Let
  \begin{equation*}
    \pi\colon\mcalU\to\mcalF_1; \quad f \mapsto (T \xrightarrow{\pi_T} T/^L p)
  \end{equation*}
  be the natural transformation and define
  \begin{equation*}
    u\colon\mcalU\xrightarrow{\beta^{-1}}W(\mcalU)/^Lp\xrightarrow{W(\pi)/^Lp}\mcalG_1.
  \end{equation*}
  In the proof of \Cref{ConstructionRetractionPerfect}, \(s^{\crys}\) is obtained by factoring \(u_T\) through \(\pi_T\) by the base-change adjunction from animated rings to animated \(\setF_p\)-algebras. 
  Then so is \(s_1 \colon \mcalF_1 \to \mcalG_1\) in the functor category on \(\mcalM\) and this gives a canonical equivalence
  \begin{equation}
    \label{CrystallineSectionStructuralFactorization}
    \epsilon\colon u\simeq s_1\circ\pi
  \end{equation}
  in \(\Map_{\Fun(\mcalM, \CAlg^{an}_{\setF_p})}(\mcalU, \mcalG_1)\).
  Hence \eqref{CrystallineSectionStructuralFactorization} determines a natural transformation
  \begin{equation*}
    \kappa_{\crys}\colon(\mcalU\xrightarrow{\pi}\mcalF_1)\to(\mcalU\xrightarrow{u}\mcalG_1)
  \end{equation*}
  whose source and target components are \(\id_{\mcalU}\) and \(s_1\), respectively.

  The transformations \(\kappa_0\) and \(\kappa_{\crys}\) have the common middle component \(\id_{\mcalU}\), so they can be pasted in the functor \(\infty\)-category. Equivalently, they form the two squares in
  \begin{equation}
    \label{PastedCompatibilityDiagram}
    \begin{tikzcd}
      \mcalF_0 \arrow[r,"s_0"] \arrow[d,"\overline{\alpha}"'] & \mcalG_0 \arrow[d,"\gamma"] \\
      \mcalU \arrow[r,"\id_{\mcalU}"] \arrow[d,"\pi"'] & \mcalU \arrow[d,"u"] \\
      \mcalF_1 \arrow[r,"s_1"'] & \mcalG_1.
    \end{tikzcd}
  \end{equation}
  The left outer arrow \(\pi\circ\overline{\alpha}\) is naturally identified with \(a\) defined in \eqref{DefQuotientNaturalTransformations}.
  Indeed, at \(f\colon R\to T\), it is the canonical factorization
  \begin{equation*}
    f/^Lp\colon R/^Lp\simeq R/pR\xrightarrow{\overline f}T\xrightarrow{\pi_T}T/^Lp.
  \end{equation*}
  The right outer arrow \(u\circ\gamma\) is naturally identified with \(b\) defined in \eqref{DefQuotientNaturalTransformations}: At \(f\), the counit equivalence \(\beta^{-1}\circ\beta\simeq\id\) and Witt functoriality give
  \begin{align*}
    u_f\circ\gamma_f &\simeq\left(W(\pi_T)/^Lp\right)\circ\left(W(\overline f)/^Lp\right)\\
    &=W(\pi_T\circ\overline f)/^Lp\simeq W(f/^Lp)/^Lp.
  \end{align*}
  Transporting the outer square of \eqref{PastedCompatibilityDiagram} through these natural identifications gives the desired natural transformation
  \begin{equation*}
    s'\colon\mcalF'\to\mcalG'.
  \end{equation*}

  It remains to construct the retraction homotopy. Naturality of Witt restriction gives a canonical equivalence
  \begin{equation*}
    \delta_0\colon\overline{\alpha}\circ\rho_0\simeq\gamma.
  \end{equation*}
  Thus \(\rho_0\), \(\id_{\mcalU}\), and \(\delta_0\) define a natural transformation
  \begin{equation*}
    r_0\colon\mcalK = (\mcalG_0 \xrightarrow{\gamma} \mcalU) \to \mcalH = (\mcalF_0 \xrightarrow{\overline{\alpha}} \mcalU).
  \end{equation*}
  Let
  \begin{equation*}
    h_0^{\tf}\colon\rho_0\circ s_0\simeq\id_{\mcalF_0}
  \end{equation*}
  be induced by the retraction homotopy of \Cref{ConstructionRetractionTorsfree}. In that construction \eqref{ConstructionRetractionTorsfreeEq}, the retraction homotopy is obtained by comparing the derived reduction \(\rho\) of Witt restriction composed with \(\sigma_R\) with the natural equivalence \(\alpha_R\).
  Together with the strict square \eqref{DiscreteCompatibilitySquare}, this shows that \(h_0^{\tf}\) is compatible with the identity homotopy of \(\mcalU\). Hence they determine an equivalence
  \begin{equation*}
    h_0\colon r_0\circ\kappa_0\simeq\id_{\mcalH}.
  \end{equation*}

  On the characteristic-\(p\) side, Witt restriction gives a canonical equivalence
  \begin{equation*}
    \delta_1\colon\rho_1\circ u\simeq\pi.
  \end{equation*}
  Thus \(\id_{\mcalU}\), \(\rho_1\), and \(\delta_1\) define a natural transformation
  \begin{equation*}
    r_1\colon(\mcalU\xrightarrow{u}\mcalG_1)\to(\mcalU\xrightarrow{\pi}\mcalF_1).
  \end{equation*}
  Let
  \begin{equation*}
    h_1^{\crys}\colon\rho_1\circ s_1\simeq\id_{\mcalF_1}
  \end{equation*}
  be induced by the retraction homotopy of \Cref{ConstructionRetractionPerfect}. By the triangle compatibility in the same base-change adjunction that defines \eqref{CrystallineSectionStructuralFactorization}, one has
  \begin{equation*}
    \pi \simeq (\rho_1 \circ s_1) \circ \pi \simeq \rho_1 \circ u \simeq \pi
  \end{equation*}
  by \(h_1^{\crys}\) at the first equivalence, \(\epsilon\) at the second equivalence, and \(\delta_1\) at the third equivalence.
  Consequently, \(\id_{\mcalU}\) and \(h_1^{\crys}\) determine an equivalence
  \begin{equation*}
    h_1\colon r_1\circ\kappa_{\crys}\simeq\id_{(\mcalU\xrightarrow{\pi}\mcalF_1)}.
  \end{equation*}

  Pasting \(r_0\) and \(r_1\), and transporting through the natural identifications of the outer arrows above, gives the natural transformation \(\rho'\colon\mcalG'\to\mcalF'\) induced by Witt restriction. Pasting \(h_0\) and \(h_1\) gives an equivalence of natural transformations
  \begin{equation*}
    \rho'\circ s'\simeq\id_{\mcalF'}.
  \end{equation*}
  All squares, paths, and pasting operations above are formed in functor \(\infty\)-categories, so compatibility with morphisms in \(\mcalM\) and all higher coherences are part of the construction.
\end{proof}

\begin{proposition} \label{ConstructionRetraction}
  There is a natural transformation
  \begin{equation*}
    s\colon\mcalF\to\mcalG
  \end{equation*}
  together with a natural homotopy
  \begin{equation*}
    \rho \circ s\simeq\id_{\mcalF}
  \end{equation*}
  in \(\Fun(\Perfd, \CAlg^{an}_{\setF_p})\)
\end{proposition}

\begin{proof}
  For a perfectoid ring \(R\), put
  \begin{equation*}
    R_{\tf}\defeq R/R[p^\infty], \qquad R_0\defeq(R/p)_{\red}, \qquad R_1\defeq(R_{\tf}/p)_{\red}.
  \end{equation*}
  By \cite{cesnavicius2024Purity}*{(2.1.3.9)}, these are functorial perfectoid rings, \(R_{\tf}\) is \(p\)-torsion-free, \(R_0\) and \(R_1\) are perfect \(\F_p\)-algebras, and there is a functorial pullback square
  \begin{equation}
    \label{eq:perfectoid-canonical-pullback}
    \begin{tikzcd}
      R \arrow[r] \arrow[d] & R_{\tf} \arrow[d] \\
      R_0 \arrow[r] & R_1.
    \end{tikzcd}
  \end{equation}
  Equivalently, the underlying additive groups form a fiber sequence
  \begin{equation*}
    R\longrightarrow R_{\tf}\oplus R_0\longrightarrow R_1.
  \end{equation*}

  Derived reduction modulo \(p\) is exact on underlying spectra, and limits of animated rings are created on underlying spectra.  Thus \eqref{eq:perfectoid-canonical-pullback} gives
  \begin{equation}
    \label{eq:F-perfectoid-pullback}
    \mcalF(R)\simeq\mcalF(R_{\tf})\times_{\mcalF(R_1)}\mcalF(R_0).
  \end{equation}
  The animated Witt functor preserves limits because it is a right adjoint, and a second derived reduction modulo \(p\) is again exact on underlying spectra.  Hence
  \begin{equation}
    \label{eq:G-perfectoid-pullback}
    \mcalG(R)\simeq\mcalG(R_{\tf})\times_{\mcalG(R_1)}\mcalG(R_0).
  \end{equation}

  By \Cref{ConstructionRetractionTorsfree} and \Cref{ConstructionRetractionPerfect}, we can get morphisms \(s^{\tf}(R_{\tf}) \colon \mcalF(R_{\tf}) \to \mcalG(R_{\tf})\), \(s^{\crys}(R_0) \colon \mcalF(R_0) \to \mcalG(R_0)\), and \(s^{\crys}(R_1) \colon \mcalF(R_1) \to \mcalG(R_1)\).
  Applying \Cref{lem:perfectoid-witt-reduction-compatibility} for the functor
  \begin{equation*}
      \Perfd \to \mcalM; \quad R \mapsto (R_{\tf} \to R_1),
  \end{equation*}
  together with the pullback homotopy in \eqref{eq:F-perfectoid-pullback}, we can give a natural transformation
  \begin{equation*}
    s \colon\mcalF \to\mcalG
  \end{equation*}
  in \(\Fun(\Perfd, \CAlg^{an}_{\setF_p})\) together with a natural homotopy
  \begin{equation*}
    \rho \circ s\simeq\id_{\mcalF}.
  \end{equation*}
\end{proof}

\begin{theorem}\label{PerfectizationToDeRham}
Let $\mscrX$ be a bounded $p$-adic formal scheme and let
\[
\nu^{\pfd}_{\mscrX}\colon
\mscrX^{\pfd}\longrightarrow\mscrX
\]
be its perfectization in the sense of \cite{bhatt2025Aspects}*{\S4}.
Then there exists a morphism
\[
\Sigma_{\mscrX}^{\pfd}\colon
(\mscrX^{\pfd})_{\F_p}
\longrightarrow
(\mscrX_{\F_p})^{\dR}
\]
fitting into a commutative diagram
\[
\begin{tikzcd}
(\mscrX^{\pfd})_{\F_p}
  \arrow[rr,"\Sigma_{\mscrX}^{\pfd}"]
  \arrow[dr,"\nu^{\pfd}_{\mscrX,\F_p}"']
&&
(\mscrX_{\F_p})^{\dR}
  \arrow[dl,"\nu_{\mscrX_{\F_p}}"]
\\
&
\mscrX_{\F_p}.
&
\end{tikzcd}
\]
\end{theorem}

\begin{proof}
We construct the morphism on points.
Let $S$ be a semiperfectoid $\F_p$-algebra.
A point
\[
x\in(\mscrX^{\pfd})_{\F_p}(S)
\]
is given by a diagram
\[
\Spec(S)
\xrightarrow{i}
\Spf(R)
\xrightarrow{u}
\mscrX,
\]
where $R$ is a perfectoid ring and \(i\) induces an isomorphism \(R^\flat \xrightarrow{\cong} S^\flat\).
In particular, this \(i\) induces a morphism of animated $\F_p$-algebras
\[
\overline{i}\colon
R/^Lp\longrightarrow S.
\]

By \Cref{ConstructionRetraction}, there is a morphism
\[
s_R\colon
R/^Lp\longrightarrow W(R/^Lp)/^Lp,
\]
natural in $R$ such that the following diagram commutes up to the natural homotopy
\begin{equation}\label{eq:pfd-comm}
\begin{tikzcd}
R/^Lp
  \arrow[r,"s_R"]
  \arrow[dr,"\id_{R/^Lp}"']
&
W(R/^Lp)/^Lp
  \arrow[r,"W(\overline{i})/^Lp"]
  \arrow[d,"\res"]
&
W(S)/^Lp
  \arrow[d,"\res"]
\\
&
R/^Lp
  \arrow[r,"\overline{i}"']
&
S.
\end{tikzcd}    
\end{equation}
Hence the composition of upper morphisms
together with the morphism $u$ determines a morphism
\[
\Spec\bigl(W(S)/^Lp\bigr)
\longrightarrow
\Spec(R/^Lp)
\xrightarrow{\var{u}}
\mscrX_{\F_p}.
\]
By the definition of the de Rham stack, this gives an $S$-point
\[
\Sigma_{\mscrX}^{\pfd}(S)(x)
\in
(\mscrX_{\F_p})^{\dR}(S).
\]

Furthermore, the diagram \eqref{eq:pfd-comm} implies
\[
\nu_{\mscrX_{\F_p}}
\circ
\Sigma_{\mscrX}^{\pfd}
\simeq
\nu^{\pfd}_{\mscrX,\F_p}.
\]

The naturality of $s$ in \Cref{ConstructionRetraction}, together with the functoriality of Witt vectors, shows that this construction is compatible with morphisms of points and with base change in $S$.
Hence it defines a morphism
\[
\Sigma_{\mscrX}^{\pfd}\colon
(\mscrX^{\pfd})_{\F_p}
\longrightarrow
(\mscrX_{\F_p})^{\dR}.
\]
\end{proof}

\subsection{Akizuki--Nakano-type vanishing in mixed characteristic}

\begin{proposition}\label{local-van-vs-mod-p-van}
Let $k$ be a perfect field of characteristic $p$.
Let $X$ be a flat projective scheme over $W(k)$ and let $\mcalM$ be a coherent sheaf on $X$.
Let $X_{\F_p}$ be the closed fiber of $X$ and let $\iota \colon X_{\F_p} \to X$ be the closed immersion.
For $i \in \Z$, we have
\[
R\Gamma_{(p)}(X,\mcalM) \in \mcalD^{\geq i} \Leftrightarrow R\Gamma(X_{\F_p},Li^*\mcalM) \in \mcalD^{\geq i-1} \Rightarrow R\Gamma(X,\mcalM) \in \mcalD^{\geq i-1},
\]
where $R\Gamma_{(p)}(X,\mcalM) \defeq R\Gamma_{(p)}R\Gamma(X,\mcalM)$.
\end{proposition}

\begin{proof}
The first equivalence follows from the fiber sequence
\[
R\Gamma_{(p)}(X,\mcalM) \xrightarrow{\times p} R\Gamma_{(p)}(X,\mcalM) \longrightarrow R\Gamma(X_{\F_p},L\iota^*\mcalM)
\]
since every cohomology of $R\Gamma_{(p)}(X,\mcalM)$ is $p$-nilpotent.

Next, we assume
\[
R\Gamma(X_{\F_p},L\iota^*\mcalM) \in \mcalD^{\geq i-1}.
\]
By the fiber sequence
\[
R\Gamma(X,\mcalM) \xrightarrow{\times p} R\Gamma(X,\mcalM) \longrightarrow R\Gamma(X_{\F_p},L\iota^*\mcalM),
\]
the homomorphisms
\[
H^m(X,\mcalM) \xrightarrow{\times p} H^m(X,\mcalM) \quad \text{for every $m < i-1$}
\]
are surjective.
Since $H^m(X,\mcalM)$ is finite module on $W(k)$, by the Nakayama's lemma, we have
\[
R\Gamma(X,\mcalM) \in \mcalD^{\geq i-1}.
\]
\end{proof}

\begin{theorem}\label{ANV-HT-split}
Let $k$ be a perfect field and let $X$ be a smooth projective scheme over $W(k)$.
If the closed fiber $X_{\F_p}$ is HT-split, then for every ample line bundle $L$ on $X$, the following vanishings hold;
\begin{align*}
    H^i_{(p)}(X,\Omega^j_{X/W(k)} \otimes L^{-1})=0 & \quad (i+j<\dim X), \\
    H^i(X_{\F_p}, \Omega^j_{X_{\F_p}} \otimes L_{\F_p}^{-1}) = 0 & \quad (i + j < \dim X - 1), \\
    H^i(X, \Omega^j_{X/W(k)} \otimes L^{-1}) = 0 & \quad (i + j < \dim X - 1 ),
\end{align*}
where \(H^i_{(p)}(X, -)\) is the \(i\)-th cohomology of \(R\Gamma_{(p)}(X, -)\).
\end{theorem}

\begin{proof}
By \Cref{vanishing}, we obtain 
\[
H^i(X_{\F_p},\Omega^j_{X_{\F_p}} \otimes L_{\F_p}^{-1})=0 \quad (i+j<\dim X_{\F_p})
\]
for every ample line bundle $L_{\F_p}$ on $X_{\F_p}$.
This shows the second vanishing. The other vanishings follow from \Cref{local-van-vs-mod-p-van}.
\end{proof}

\begin{corollary}
\label{LimPerfectoidSplitImpliesHTSplit}
Let $X$ be a Noetherian scheme such that every closed point of $X$ is contained in the closed fiber $X_{\F_p}$.
If $X$ is lim-perfectoid split, then \(X_{\F_p}\) is HT-split.

In particular, if $X$ is smooth and projective over $W(k)$ for a perfect field $k$ of characteristic $p$, then for every ample line bundle $L$ on $X$, we have
\[
H^i_{(p)}(X,\Omega^j_{X/W(k)} \otimes L^{-1})=0 \quad (i+j<\dim X).
\]
\end{corollary}

\begin{proof}
We prove the first assertion.
Let $\mscrX$ be the $p$-adic completion of $X$.
By \cite{ishizuka2026Algebraization}*{Theorem~4.28}, we have an equivalence
\[
R\nu^{\pfd}_{\mscrX,\F_p,*}\mcalO_{(\mscrX^{\pfd})_{\F_p}}
\simeq
\mcalO_{\mscrX,\perfd}/^Lp
\]
in \(\CAlg(\mcalD_{\qcoh}(\mscrX))\).
Since \(\mscrX\) is lim-perfectoid split, the morphism
\[
\mcalO_{X_{\F_p}}
\simeq
\mcalO_{\mscrX}/^Lp
\longrightarrow
R\nu^{\pfd}_{\mscrX,\F_p,*}\mcalO_{(\mscrX^{\pfd})_{\F_p}}
\]
is ind-split in \(\mcalD_{\qcoh}(X_{\setF_p})\).

By \Cref{PerfectizationToDeRham}, the morphism
\[
\cO_{X_{\F_p}} \to \OdR{X_{\F_p}}
\]
is ind-split.
Thus \(X_{\F_p}\) is HT-split.

The second assertion follows from the first assertion and \Cref{ANV-HT-split}.
\end{proof}

\begin{corollary}\label{GPS-to-ANV}
Let $X$ be a Noetherian scheme such that every closed point of $X$ is contained in the closed fiber $X_{\F_p}$.
Assume that $X$ is globally $+$-regular. 
Then \(X_{\F_p}\) is HT-split.

In particular, if $X$ is smooth and projective over $W(k)$ for a perfect field $k$ of characteristic $p$, then for every ample line bundle $L$ on $X$, we have
\[
H^i_{(p)}(X,\Omega^j_{X/W(k)} \otimes L^{-1})=0 \quad (i+j<\dim X).
\]
\end{corollary}

\begin{proof}
It follows from \Cref{LimPerfectoidSplitImpliesHTSplit} and the fact that $X$ is lim-perfectoid split by \cite{ishizuka2026Localglobal}*{Proposition 3.5}.
\end{proof}


\end{document}